\documentclass[review,hidelinks,onefignum,onetabnum]{siamart220329}
\newtheorem{example}{Example}[section]

\newenvironment{breakablealgorithm}
  {
   \begin{center}
     \refstepcounter{algorithm}
     \hrule height.8pt depth0pt \kern2pt
     \renewcommand{\caption}[2][\relax]{
       {\raggedright\textbf{Algorithm~\thealgorithm} ##2\par}%
       \ifx\relax##1\relax 
         \addcontentsline{loa}{algorithm}{\protect\numberline{\thealgorithm}##2}%
       \else 
         \addcontentsline{loa}{algorithm}{\protect\numberline{\thealgorithm}##1}%
       \fi
       \kern2pt\hrule\kern2pt
     }
  }
           {
   \end{center}
   \kern2pt\hrule\kern2pt
  }

\usepackage{lipsum}
\usepackage{amsfonts}
\usepackage{graphicx}
\usepackage{epstopdf}
\usepackage{algorithmic}
\ifpdf
  \DeclareGraphicsExtensions{.eps,.pdf,.png,.jpg}
\else
  \DeclareGraphicsExtensions{.eps}
\fi

\newsiamremark{remark}{Remark}
\newsiamremark{hypothesis}{Hypothesis}
\crefname{hypothesis}{Hypothesis}{Hypotheses}
\newsiamthm{claim}{Claim}

\headers{Two-dimensional greedy randomized extended Kaczmarz methods}{X.F. Zhang, M.L. Xiao, T. Li}
\title{Two-dimensional greedy randomized extended Kaczmarz methods \thanks{Corresponding author: T. Li (tli@hainanu.edu.cn).
\funding{This work was funded by the National Natural Science Foundation of China [grant numbers 12401493 and 12401019], and the Hainan Provincial Natural Science Foundation of China [grant number 125RC627].}}}
\author{Xin-Fang Zhang\thanks{School of Mathematics and Statistics, Hainan University, Haikou 570228, P. R. China
  (\email{995272@hainanu.edu.cn}).}
  \and Meng-Long Xiao\thanks{School of Mathematics and Statistics, Hainan University, Haikou 570228, P. R. China
  (\email{xml@hainanu.edu.cn}).}
  \and Tao Li\thanks{School of Mathematics and Statistics, Hainan University, Haikou 570228, P. R. China
  (\email{tli@hainanu.edu.cn}).}
}

\usepackage{amsopn}

\ifpdf
\hypersetup{
  pdftitle={Two-dimensional greedy randomized extended Kaczmarz methods},
  pdfauthor={Tao Li, Meng-Long Xiao, Xin-Fang Zhang}
}
\fi

\begin{document}

\maketitle

\begin{abstract}
The randomized extended Kaczmarz method, proposed by Zouzias and Freris (SIAM J. Matrix Anal. Appl. 34: 773-793, 2013), is appealing for solving least-squares problems. However, its randomly selecting rows and columns of $A$ with probability proportional to their squared norm is unattractive compared to the greedy strategy.
In this paper, we first consider a novel two-dimensional greedy randomized extended Kaczmarz method
for solving large linear least-squares problems. The proposed method randomly selects two rows and two columns of $A$ by grasping two larger entries in the magnitude of the corresponding residual vector per iteration. To improve its convergence, we then propose a two-dimensional semi-randomized extended Kaczmarz method and its modified version with simple random sampling, which is particularly favorable for big data problems. The convergence analysis of which is also established. Numerical results on some practical applications illustrate the superiority of the proposed methods compared with state-of-the-art randomized extended Kaczmarz methods, especially in terms of computing time.
\end{abstract}

\begin{keywords}
Randomized extended Kaczmarz method; Two-dimensional; Greedy randomized Kaczmarz method; Semi-randomized; Simple random sampling
\end{keywords}

\begin{MSCcodes}
65F10; 65F20; 94A08
\end{MSCcodes}
\section{Introduction}\label{section-1}
Consider the solution of the least-squares problem
\begin{equation}\label{1-1}
\min\limits_{x\in\mathbb{C}^n}\|b-Ax\|_2,  
\end{equation}
where $A\in \mathbb{C}^{m\times n} $ and $b\in \mathbb{C}^{m} $ are given, and $x\in \mathbb{C}^{n} $ is required to be determined. Such a least-squares problem has widespread applications in the field of scientific computing and engineering, such as image reconstruction \cite{JA,GT}, geophysics \cite{Menke,KC} and big data analysis \cite{Björck}, etc. Note that iterative methods for solving linear systems have been well developed \cite{CC,DC,Kaczmarz,D}. Among them, the randomized iterative methods \cite{Strohmer,Needell,Elble,Needell2,KDu,Ma,Bai1,Bai2,Miao,Zou,Needell3,WWT,Ke} have attracted much attention, since they do not even need to know the whole system, but only a small random part. The aim of this paper is to devise high-performance randomized iterative methods for solving the least-squares problem \eqref{1-1}.

Existing randomized iterative methods for Eq.\eqref{1-1} are randomized extended
Kaczmarz-type methods, which is a specific combination of the randomized orthogonal projection methods together with the randomized Kaczmarz-type methods. For example, Zouzias and Freris \cite{Zou} first proposed a randomized extended Kaczmarz (REK) method, which converges in the mean square to the least-squares solution $x_{LS}=A^\dagger b$ of \eqref{1-1}. Denote $b^{(i)}$, $A^{(i)}$ and $A_{(j)}$ as the $i$-th entry of a vector $b$, the $i$-th row, and $j$-th column of a matrix $A$, respectively. Then the REK solver iteratively updates two vectors $z_k$ and $x_k$ ($A^*z=0$ for $z_k$ and $Ax=b-z_k$ for $x_k$) by
$$
z_{k+1} =(I_m-\frac{A_{(j_k)}A_{(j_k)}^*}{\|A_{(j_k)}\|_2^2})z_{k},
$$
$$
x_{k+1} =x_k+\frac{{b}^{(i_{k})}-z_k^{(i_k)}-A^{(i_{k})}x_{k}} {\| A^{(i_{k})}\|_{2} ^{2} }(A^{(i_{k})})^{*},
$$
in which the symbol $(\cdot)^*$ denotes the conjugate transpose of the corresponding vector. To further improve its convergence, a block version of REK has 
 been proposed in \cite{Needell3}. However, the method often requires good row and column pavings, which consumes additional operations, especially when the row or column norms of $A$ vary significantly. After that, Wu \cite{WWT} developed a two-subspace randomized extended Kaczmarz (TREK) method, for solving the least-squares problem \eqref{1-1}, which does not require any row or column paving. The TREK randomly selects two working rows and columns of $A$ with probability proportional to its relevance, which converges faster than REK. Moreover, Du et al. devised a novel randomized extended average block Kaczmarz method \cite{KDu} for finding the least-squares solution of Eq.\eqref{1-1} without involving the Moore-Penrose inverse. Besides, Li and Wu \cite{LiS} developed a semi‐randomized and augmented Kaczmarz method with simple random sampling for inconsistent linear systems. This method requires only a small portion of the dataset, making it more competitive for big data problems. Recently, Ke et al. \cite{Ke} derived a greedy block extended Kaczmarz (GBEK) method with the average block projection technique. For other Kaczmarz-type methods, see \cite{Necoara,Schöpfer,Alderman,LW}.

From the above discussions, one can see that each iteration of them is a typical one-dimensional projection method. Specifically, these methods, at $k$-th iteration, extract approximate solutions $x_{k+1}$ and $z_{k+1}$ from one-dimensional affine subspaces $x_{k}+\mathcal K_{0}$ and $z_{k}+\tilde{\mathcal K}_{0}$ by imposing the Petrov-Galerkin conditions
\begin{equation}\label{1-4}
{b}-{A}x_{k+1} \perp \mathcal L_{0} \quad \mathrm{and} \quad {A}^*z_{k+1} \perp \tilde{\mathcal L}_{0},
\end{equation}
respectively, where $\mathcal K_{0}= \mathrm{span}\left \{(A^{(i_k)})^* \right \}$, $\tilde{\mathcal K}_{0}= \mathrm{span}\left \{A_{(j_k)} \right \}$, and $\mathcal L_{0}=\mathrm{span}\left \{ e _{i_k}  \right \}$ and $\tilde{\mathcal L}_{0}=\mathrm{span}\left \{ e _{j_k}  \right \}$. Note that the $x_{k+1}$ and $z_{k+1}$ generated from a two-dimensional or higher-dimensional subspace, which consists of multiple active rows or columns of $A$, may converge faster than the one-dimensional method. Nevertheless, the TREK method, which involves successive projections onto two different hyperplanes, is far more complex. Therefore, from a purely algebraic point of view,
 an alternative to TREK is concisely proposed, which performs more straightforwardly. It should be pointed out that TREK is costly for scanning all rows and columns of the data matrix in advance to calculate the probability, which is unfavorable for big data problems. To deal with this problem, motivated by Chebyshev's (weak) law of large numbers, we further apply the simple sampling strategy to the two-subspace randomized extended Kaczmarz method, and then propose a two-dimensional randomized extended Kaczmarz method with simple random sampling (TREKS).

As well-known that active working rows and columns play a key role in extended Kaczmarz-type methods, while the greedy selection is a powerful strategy. Thus, we propose a greedy randomized extended Kaczmarz (GREK) method, which incorporates a greedy randomized Kaczmarz iteration with a greedy randomized orthogonal projection iteration that approximates the projection of $b$ onto the orthogonal complement of the range space of $A$. To neither compute the probabilities nor construct the index sets for working rows and columns, we then devise a semi-randomized extended Kaczmarz (SREK) method with grasping the current maximal homogeneous residuals. From the two-dimensional Kaczmarz iteration, we further develop the two-dimensional greedy randomized extended Kaczmarz (TGREK) method and its semi-randomized variant. However, for big data problems, the proposed semi-randomized extended Kaczmarz methods may still be unavailability, since the residual in each iteration has to be computed. Hence, a novel two-dimensional semi-randomized extended Kaczmarz method with simple random sampling (TSREKS), based on Chebyshev's (weak) law of large numbers, is proposed.

The organization of this paper is as follows.  In Section \ref{section-3}, we propose several one-dimensional and two-dimensional extended Kaczmarz-type methods, for solving the least-squares problem \eqref{1-1}. The convergence of which is also established. In Section \ref{section-4}, we present some numerical examples to demonstrate the feasibility and validity of the proposed methods compared with some existing methods. Finally, in Section \ref{section-5}, we conclude this paper by giving some remarks.

In this paper, we denote the $i$-th row and
$j$-th column of a matrix $A\in \mathbb{C}^{m\times n}$ by $A^{(i)}$ and $A_{(j)}$, respectively. The Frobenius norm and 2-norm \cite{Saad1} of $A$ are defined as follows:
\begin{equation}\label{2-1}
\| A\|_F:= {(\sum\limits_{j = 1}^n {\sum\limits_{i = 1}^m {{{\left| {{A_{ij}}} \right|}^2}} } )^{\frac{1}{2}}}=\sqrt{\sum_{i=1}^{t}\sigma_i^2},\quad 
\|A \|_2:= \sigma_{s}, 
\end{equation} 
 where $\sigma_{i}$ denotes the nonzero singular values of $A$, and $\sigma_{s}$ is the largest nonzero singular value. Denote the conjugate transpose of $A$ by $A^*$ and the Moore-Penrose inverse by $A^{\dagger}$. Denote  the column space of $A$ by $\mathcal{R}(A)$ and its orthogonal complement subspace by  $\mathcal{R}(A)^{\bot}$. Let $b_{\mathcal{R}(A)}$ and $b_{\mathcal{R}(A)^{\bot}}$ be the orthogonal projection of $b$ onto $\mathcal{R}(A)$ and ${\mathcal{R}(A)^{\bot}}$, respectively. For the matrix $A^*A$, we denote its nonzero minimal eigenvalue by $\lambda_{min}(A^*A)$, and the maximal eigenvalue by $\lambda_{max}(A^*A)$. Denote by $\mathbb{E} _{k}$, the expected value conditional of the first $k$ iterations, i.e.,
$\mathbb{E} _{k} [\cdot ]=\mathbb{E}[\cdot \mid i_{0_1},i_{0_2},j_{0_1},i_{0_2} ,\dots,i_{(k-1)_1},i_{(k-1)_2},j_{(k-1)_1},j_{(k-1)_2}  ],$ where $i_{l_1},i_{l_2},j_{l_1},j_{l_2} \:(l=0,\cdots,
k-1)$ are the rows and columns chosen at the $l$-th iterate, and from the law of iterated expectations, we obtain $\mathbb{E}[\mathbb{E} _{k} [\cdot ]]=\mathbb{E}[\cdot]$. If $k$ is a real number, then ${\left \lfloor k \right \rfloor }$ represents the largest integer that is smaller than or equal to $k$. Define
\begin{equation}\label{2-2}  
\delta=\min\limits_{i_{k_1}\neq i_{k_2}}\frac{|A^{(i_{k_1})}(A^{(i_{k_2})})^*|}{\|A^{(i_{k_1})}\|_2\|A^{(i_{k_2})}\|_2} \quad\mathrm{and}\quad \Delta=\max\limits_{i_{k_1}\neq i_{k_2}}\frac{|A^{(i_{k_1})}(A^{(i_{k_2})})^*|}{\|A^{(i_{k_1})}\|_2\|A^{(i_{k_2})}\|_2}, i_{k_1},i_{k_2}\in[m]
\end{equation}.
\begin{equation}\label{2-3}
\tilde{\delta}=\min\limits_{i_{k_1}\neq i_{k_2}}\frac{|(A_{(i_{k_1})})^*A_{(i_{k_2})}|}{\|A_{(i_{k_1})}\|_2\|A_{(i_{k_2})}\|_2} \quad\mathrm{and}\quad \tilde{\Delta}=\max\limits_{i_{k_1}\neq i_{k_2}}\frac{|(A_{(i_{k_1})})^*A_{(i_{k_2})}|}{\|A_{(i_{k_1})}\|_2\|A_{(i_{k_2})}\|_2}, i_{k_1},i_{k_2}\in[n].
\end{equation}
Obviously we have $0\leq\delta\leq\Delta<1$ and $0\leq\tilde{\delta}\leq\tilde{\Delta}<1$.

\section{Two-dimensional randomized extended Kaczmarz-type methods}\label{section-3}
In this section, we devise several two-dimensional randomized extended Kaczmarz methods with different strategies, including greedy, semi-randomized, and simple random sampling strategies, for solving the least-squares problem \eqref{1-1}. 
In what follows, we first propose a fast orthogonal projection method built upon the greedy criterion, which approximately computes the orthogonal projection of $b$ onto the column space of $A$, and then establish the greedy randomized extended Kaczmarz (GREK) method.





\subsection{Greedy randomized extended Kaczmarz method}\label{subsection-3-1}
For randomized extended Kaczmarz methods, a crucial factor of convergence is how to project 
orthogonally $b$ onto the column space of $A$.
Observe from \cite{Bai} that the greedy randomized Kaczmarz (GRK) solver, with the greedy selection strategy, often outperforms the randomized Kaczmarz (RK) \cite{Strohmer} solver. Therefore, a greedy randomized orthogonal projection is as follows.
\begin{algorithm}
	\caption{Greedy randomized orthogonal projection}
	\label{algo-3-1}
	\begin{algorithmic}[1]
        \REQUIRE $A, b$\\
        \ENSURE $z_{k+1}$\\
        \STATE
        Initialize $z_0= b$
        \STATE
       For $k=0, 1, 2,\cdots $, until convergence, do:\\
\STATE   Compute
        \begin{equation}\label{3-1}
         \bar{\epsilon} _{k} = \frac{1}{2}\left(\frac{1}{{\parallel  A^*z_{k}{\parallel _{2}^2}}}\mathop {\max }\limits_{\substack{i_k \in [n]}} \left\{ \frac{|A_{(i_{k})}^*z_k|^2}{\|A_{(i_{k})}\|_2^2}\right\} + \frac{1}{{\parallel A\parallel _{F}^2}}\right),
        \end{equation}
         \STATE
         Determine the index set of positive integers\\
         \begin{equation}\label{3-2}
         \bar{\mathcal{U}}_k = \left\{ {i_k}\Big| {|A^*_{(i_{k})}z_{k}|^2 \ge \bar{\epsilon}_k} \parallel 
         A^*{z_{k}}{\parallel _{2}^2}\| A_{(i_{k})}\|_{2}^2 \right\}.\end{equation}\\
         \STATE
Compute the $i$-th element of the vector 
     $\tilde{\mathbf{r}} _{k_2}$ by  \\
     $$\tilde{\mathbf{r} } _{k_2} ^{(i)}=\left\{ \begin{array}{l}
    {A_{(i_{k})}^*z_k},\qquad  \mathrm{if} \quad i \in \bar{\mathcal{U}}_{k}
     \\
    0.\qquad \qquad \quad \enspace \mathrm{otherwise}
    \end{array}
    \right.$$
        \STATE
        Select $i_k \in {\bar{\mathcal{U}}_k}$ with probability $\mathbb{P}$(column =$i_k$)=$ \mid {{{\tilde {\mathbf{r} }}_{k_2}}^{(i_k)}}\mid^2  /\left \| {{{\tilde {\mathbf{r} }}}_{k_2}} \right \|_{2}^2 $.\\
    \STATE
 Compute $z_{k+1} =(I_m-\frac{A_{(i_k)}A_{(i_k)}^*}{\|A_{(i_k)}\|_2^2})z_{k}.$

	\end{algorithmic}  
\end{algorithm}

From Algorithm \ref{algo-3-1}, we can see that the index set $\bar{\mathcal{U}}_k$ is nonempty for all iteration index $k$, since
$$
\max\limits_{\substack{1\leq i_k \leq n}}\frac{|A_{(i_{k})}^*z_k|^2}{\|A_{(i_{k})}\|_2^2}\geq \sum_{i_k=1}^{n}\frac{\|A_{(i_{k})}\|_2^2}{\|A\|_{F}^2}\frac{|A_{(i_{k})}^*z_k|^2}{\|A_{(i_{k})}\|_2^2}=\frac{\|A^*z_k\|_2^2}{\|A\|_{F}^2}
$$
and
$$
\frac{|A_{(i)}^*z_k|^2}{\|A_{(i)}\|_2^2}=\max\limits_{\substack{1\leq i_k \leq n}}\frac{|A_{(i_{k})}^*z_k|^2}{\|A_{(i_{k})}\|_2^2}\geq \frac{1}{2}\max\limits_{\substack{1\leq i_k \leq n}}\frac{|A_{(i_{k})}^*z_k|^2}{\|A_{(i_{k})}\|_2^2}+\frac{1}{2}\frac{\|A^*z_k\|_2^2}{\|A\|_{F}^2}
$$
hold, i.e., the maximum is always greater than the average. Next, we present an important theorem for guaranteeing the convergence of Algorithm \ref{algo-3-1}.
\begin{theorem}\label{The-1}
{\rm Let $b$ be the initial guess of Algorithm \ref{algo-3-1}. Then the iterative sequence $\{z_{k}\}$ generated by Algorithm \ref{algo-3-1} converges to $b_{\mathcal{R}(A)^{\bot}}$ in expectation. Moreover, the corresponding error norm in expectation yields }
\end{theorem}
\begin{equation}\label{3-3}
\mathbb{E}_{k}\|z_k-b_{\mathcal{R}(A)^{\bot}}\|_2^2\leq \left[1-\frac{1}{2}\left(\frac{\|A\|_F^2}{\tilde{\tau}_{max}}+1\right)\frac{\lambda_{min}(A^*A)}{\|A\|_F^2}\right]^k\|b_{\mathcal{R}(A)}\|_2^2,   
\end{equation}
where $\tilde{\tau}_{max}=\|A\|_F^2-\min\limits_{i\in[n]}\|A_{(i)}\|_2^2$.
\begin{proof}
Let $P_i=I_m-\frac{A_{(i)}A_{(i)}^*}{\|A_{(i)}\|_2^2}$ for every $i\in[n]$. It is easy to verify that $P_i^2=P_i$ and $P_i^*=P_i$. This shows that $P$ is a projection matrix, and if $P_i$ is a projector, then so is $(I-P_i)$. Define $e_k:=z_k-b_{\mathcal{R}(A)^{\bot}}$ for every $k \geq 0$.
When $k=0$, it follows that
\begin{equation}\label{3-4}
\bar{\epsilon}_0\|A\|_{F}^2\ge 1 .
\end{equation}
From Eq.\eqref{3-1}, when $k=1,2,\cdots,$ it follows that
$$
\begin{aligned}
\bar{\epsilon}_k\|A\|_F^2&= \frac{1}{2}\frac{\max \limits_{\substack{i_k \in [n]}} \left\{ \frac{|A_{(i_{k})}^*z_k|^2}{\|A_{(i_{k})}\|_2^2}\right\}}{\sum\limits_{i_k=1}^n\frac{\|A_{(i_{k})}\|_2^2}{\|A\|_F^2}\frac{|A_{(i_{k})}^*z_k|^2}{\|A_{(i_{k})}\|_2^2}} +\frac{1}{2}\\
&=\frac{1}{2}\frac{\max \limits_{\substack{i_k \in [n]}} \left\{ \frac{|A_{(i_{k})}^*z_k|^2}{\|A_{(i_{k})}\|_2^2}\right\}}{\sum\limits_{\substack{i_k=1\\i_k\neq i_{k-1}}}^n\frac{\|A_{(i_{k})}\|_2^2}{\|A\|_F^2}\frac{|A_{(i_{k})}^*z_k|^2}{\|A_{(i_{k})}\|_2^2}} +\frac{1}{2}\\
&\geq\frac{1}{2}\left(\frac{\|A\|_F^2}{\sum\limits_{\substack{i_k=1\\i_k\neq i_{k-1}}}^n\|A_{(i_{k})}\|_2^2}+1\right)\\
&\geq\frac{1}{2}\left(\frac{\|A\|_F^2}{\tilde{\tau}_{max}}+1\right).
\end{aligned}
$$
Moreover, from Algorithm \ref{algo-3-1} and $P_{i_k}b_{\mathcal{R}(A)^{\bot}}=b_{\mathcal{R}(A)^{\bot}}$, it follows that
\begin{equation}\label{3-5}
\begin{aligned}
e_k&=z_k-b_{\mathcal{R}(A)^{\bot}}\\
&=P_{i_k}z_{k-1}-b_{\mathcal{R}(A)^{\bot}}\\
&=P_{i_k}z_{k-1}-P_{i_k}b_{\mathcal{R}(A)^{\bot}}\\
&=P_{i_k}(z_{k-1}-b_{\mathcal{R}(A)^{\bot}})\\
&=P_{i_k}e_{k-1}.
\end{aligned}
\end{equation}
According to the above results, it can be stated that
$$
\begin{aligned}
\mathbb{E}_{k}\|e_k\|_2^2&=\mathbb{E}_{k}\|P_{i_k}e_{k-1}\|_2^2\\
&=\mathbb{E}_{k}\left \langle P_{i_k}e_{k-1},P_{i_k}e_{k-1}  \right \rangle \\
&=\mathbb{E}_{k}\left \langle e_{k-1},P_{i_k}e_{k-1}  \right \rangle\\
&=\|e_{k-1}\|_2^2-\mathbb{E}_{k}\left \langle e_{k-1},\frac{A_{(i_k)}A_{(i_k)}^*}{\|A_{(i_k)}\|_2^2}e_{k-1}  \right \rangle\\
&=\|e_{k-1}\|_2^2-\sum\limits_{{i_k} \in {\bar{\mathcal{U}}_{k}}}\frac{|A_{(i_{k})}^*z_k|^2}{\sum\limits_{{i} \in {\bar{\mathcal{U}}_{k}}}|A_{(i)}^*z_k|^2} \frac{|A_{i_k}^*e_{k-1}|^2}{\|A_{i_k}\|_2^2} \\
&\leq \|e_{k-1}\|_2^2-\bar{\epsilon}_k\|
         A^*{z_{k}}\|_{2}^2\\
&\leq \left[1-\frac{1}{2}\left(\frac{\|A\|_F^2}{\tilde{\tau}_{max}}+1\right)\frac{\lambda_{min}(A^*A)}{\|A\|_F^2}\right]\|e_{k-1}\|_2^2.
\end{aligned}
$$
The theorem follows immediately by induction.
\end{proof}

Together with Algorithm \ref{algo-3-1} gives the following greedy randomized extended Kaczmarz method.

\begin{algorithm}
	\caption{Greedy randomized extended Kaczmarz method}
	\label{algo-3-2}
	\begin{algorithmic}[1]
        \REQUIRE $A, b,z_0= b$\\
        \ENSURE $x_{k+1}$, $z_{k+1}$\\
        \STATE
       For $k=0, 1, 2,\cdots $, until convergence, do:\\
       \STATE        Compute\\
        \begin{equation}\label{3-6}
         \epsilon  _{k} = \frac{1}{2}\left(\frac{1}{{\parallel b - Ax_{k}{\parallel _{2}^2}}}\mathop {\max }\limits_{\substack{i_k \in [m]}} \left\{ \frac{|{b}^{(i_{k})}-z_{k}^{(i_{k})}-A^{(i_{k})}x_{k}|^2} {\| A^{(i_{k})}\|_{2}^2 }\right\} + \frac{1}{{\parallel A\parallel _{F}^2}}\right),
        \end{equation}
        \begin{equation}\label{3-7}
         \bar{\epsilon} _{k} = \frac{1}{2}\left(\frac{1}{{\parallel  A^*z_{k}{\parallel _{2}^2}}}\mathop {\max }\limits_{\substack{j_k \in [n]}} \left\{ \frac{|A_{(j_{k})}^*z_k|^2}{\|A_{(j_{k})}\|_2^2}\right\} + \frac{1}{{\parallel A\parallel _{F}^2}}\right),
        \end{equation}
         \STATE
         Determine the index sets of positive integers\\
         \begin{equation}\label{3-8}
         \mathcal{U}_k = \left\{ {{i_k}\Big| {|{b}^{(i_{k})}-z_{k}^{(i_{k})}-A^{(i_{k})}x_{k}|^2 \ge {\epsilon_k}} \parallel b - 
         Ax_{k}{\parallel _{2}^2}}\| A^{(i_{k})}\|_{2}^2 \right\}.\end{equation}\\
         \begin{equation}\label{3-9}
         \bar{\mathcal{U}}_k = \left\{ {j_k}\Big| {|A^*_{(j_{k})}z_{k}|^2 \ge \bar{\epsilon}_k} \parallel 
         A^*{z_{k}}{\parallel _{2}^2}\| A_{(j_{k})}\|_{2}^2 \right\}.\end{equation}\\
         \STATE
   Compute the $i$-th element of the vector 
     $\tilde{\mathbf{r}} _{k_1}$ by  \\
     $$\tilde{\mathbf{r} } _{k_1} ^{(i)}=\left\{ \begin{array}{l}
    {{b}^{(i_{k})}-z_{k}^{(i_{k})}-A^{(i_{k})}x_{k}},\qquad  \mathrm{if} \quad i \in \mathcal{U}_{k}
     \\
    0.\qquad \qquad \quad \enspace \mathrm{otherwise}
    \end{array}
    \right.$$
        \STATE
        Select $i_{k} \in {\mathcal{U}_k}$ with probability $\mathbb{P}$(row =$i_{k}$)=$ \mid {{{\tilde {\mathbf{r} }}_{k_1}}^{(i_{k})}}\mid^2/\left \| {{{\tilde {\mathbf{r} }}}_{k_1}} \right \|_{2}^2 $.\\
         \STATE
   Compute the $j$-th element of the vector 
     $\tilde{\mathbf{r}} _{k_2}$ by  \\
     $$\tilde{\mathbf{r} } _{k_2} ^{(j)}=\left\{ \begin{array}{l}
    {A_{(j_{k})}^*z_k},\qquad  \mathrm{if} \quad j \in \bar{\mathcal{U}}_{k}
     \\
    0.\qquad \qquad \quad \enspace \mathrm{otherwise}
    \end{array}
    \right.$$
        \STATE
        Select $j_{k} \in {\bar{\mathcal{U}}_k}$ with probability $\mathbb{P}$(column =$j_{k}$)=$ \mid {{{\tilde {\mathbf{r} }}_{k_2}}^{(j_{k})}}\mid^2  /\left \| {{{\tilde {\mathbf{r} }}}_{k_2}} \right \|_{2}^2 $.\\
            \STATE
            Set $z_{k+1} =(I_m-\frac{A_{(j_k)}A_{(j_k)}^*}{\|A_{(j_k)}\|_2^2})z_{k}.$\\
            \STATE
            Set $x_{k+1} =x_k+\frac{{b}^{(i_{k})}-z_k^{(i_k)}-A^{(i_{k})}x_{k}} {\| A^{(i_{k})}\|_{2} ^{2} }(A^{(i_{k})})^{*}$.

	\end{algorithmic}  
\end{algorithm}

From Theorem \ref{The-1}, we establish the following result regarding the convergence rate
of Algorithm \ref{algo-3-2}.

\begin{theorem}\label{The-2}
{\rm Let $x_{\star}=A^{\dagger}b$ be the least-norm solution of Eq.\eqref{1-1}. Then the iterative sequence $\{x_{k}\}$ generated by Algorithm \ref{algo-3-2} converges to $x_{\star }$ for any initial vector $x_0$ in expectation. Moreover, the corresponding error norm in expectation yields
}
\begin{equation}\label{3-10}
\mathbb{E}\|x_{k}-x_{\star}\|_2^2\leq(\max\{\alpha,\beta\})^{\left \lfloor \frac{k}{2} \right \rfloor }\left(1+2\frac{\|A\|_F^2}{t}\right)\|x_{\star}\|_2^2,    
\end{equation}
\rm where $\alpha=1-\frac{1}{2}\left(\frac{\|A\|_F^2}{\tau_{max}}+1\right)\frac{\lambda_{min}(A^*A)}{\|A\|_F^2}$ and $\beta=1-\frac{1}{2}\left(\frac{\|A\|_F^2}{\tilde{\tau}_{max}}+1\right)\frac{\lambda_{min}(A^*A)}{\|A\|_F^2}$ with $\tau_{max}=\|A\|_F^2-\min\limits_{i\in[m]}\|A^{(i)}\|_2^2$, $\tilde{\tau}_{max}=\|A\|_F^2-\min\limits_{i\in[n]}\|A_{(i)}\|_2^2$ and $t=\min\limits_{i\in[m]}\|A^{(i)}\|_2^2$.
\end{theorem}
\begin{proof}
Similar to \cite{Zou}.
Define 
\begin{equation}\label{3-11}   
\tilde{x}_{k+1}=x_k+\frac{b_{\mathcal{R}(A)}^{(i_{k})}-A^{(i_{k})}x_{k}} {\| A^{(i_{k})}\|_{2} ^{2} }(A^{(i_{k})})^{*}.
\end{equation}
Note that $\tilde{x}_{k+1}$ is the projection of $x_{k}$ on the hyperplane $$\mathcal{H}_{i_k}:=\left\{x:\left \langle A^{(i_{k})},x \right \rangle =b_{\mathcal{R}(A)}^{(i_{k})}\right\}$$ by GRK. Then it follows that
\begin{equation}\label{3-12}
x_{k+1}=x_k+\frac{b_{\mathcal{R}(A)}^{(i_{k})}+w^{(i_{k})}-A^{(i_{k})}x_{k}}{\| A^{(i_{k})}\|_{2} ^{2}}(A^{(i_{k})})^{*},
\end{equation} 
where $x_{k+1}$ is the projection of $x_{k}$ on the hyperplane $$\mathcal{H}_{i_k}^{w}:=\left\{x:\left \langle A^{(i_{k})},x \right \rangle =b_{\mathcal{R}(A)}^{(i_{k})}+w^{(i_{k})}\right\}$$ by GRK.
Therefore,
$$
x_{k+1}-\tilde{x}_{k+1}=\frac{w^{(i_{k})}}{\| A^{(i_{k})}\|_{2} ^{2}}(A^{(i_{k})})^{*}
$$
and
$$
\tilde{x}_{k+1}-x_{\star}=\left(I-\frac{(A^{(i_{k})})^{*}A^{(i_{k})}}{ \|A^{(i_{k})}\|_{2} ^{2}}\right)(x_k-x_{\star}).
$$
which shows that $(x_{k+1}-\tilde{x}_{k+1})^*(\tilde{x}_{k+1}-x_{\star})=0$.
Furthermore, by orthogonality, it follows that
$$
\mathbb{E}_k\|x_{k+1}-x_{\star}\|_2^2=\mathbb{E}_k\|\tilde{x}_{k+1}-x_{\star}\|_2^2+\mathbb{E}_k\|x_{k+1}-\tilde{x}_{k+1}\|_2^2.
$$
For the first term of the right-hand side, it follows that
\begin{equation}\label{3-13}
\begin{aligned}
\mathbb{E}_k\|\tilde{x}_{k+1}-x_{\star}\|_2^2&=\|x_k-x_{\star}\|_2^2-\sum\limits_{{i_k} \in {\mathcal{U}_{k}}}\frac{|b_{\mathcal{R}(A)}^{(i_{k})}-A^{(i_{k})}x_k|^2}{\sum\limits_{{i} \in {\mathcal{U}_{k}}}|b_{\mathcal{R}(A)}^{(i)}-A^{(i)}x_k|^2} \frac{|b_{\mathcal{R}(A)}^{(i_{k})}-A^{(i_{k})}x_k|^2}{\|A^{(i_{k})}\|_{2}^{2}}\\
&\leq \left[1-\frac{1}{2}\left(\frac{\|A\|_F^2}{\tau_{max}}+1\right)\frac{\lambda_{min}(A^*A)}{\|A\|_F^2}\right]\|x_k-x_{\star}\|_2^2.
\end{aligned}
\end{equation}
For the second term of the right-hand side, it follows that
\begin{equation}\label{3-14}
\begin{aligned}
\|x_{k+1}-\tilde{x}_{k+1}\|_2^2=\frac{|w^{(i_{k})}|^2}{\| A^{(i_{k})}\|_{2} ^{2}}=\frac{|b_{\mathcal{R}(A)^{\bot}}^{(i_{k})}-z_k^{(i_{k})}|^2}{\| A^{(i_{k})}\|_{2} ^{2}}\leq \frac{\|b_{\mathcal{R}(A)^{\bot}}-z_{k}\|_2^2}{t},
\end{aligned}
\end{equation}
where $t=\min\limits_{i\in[m]}\|A^{(i)}\|_2^2$. As a result, we have
\begin{equation}\label{3-15}
\mathbb{E}_k\|x_{k+1}-x_{\star}\|_2^2\leq \alpha\|x_k-x_{\star}\|_2^2+\frac{1}{t}\|b_{\mathcal{R}(A)^{\bot}}-z_{k}\|_2^2.   
\end{equation}
Observe from Theorem \ref{The-1} that
\begin{equation}\label{3-16}
\mathbb{E}\|z_l-b_{\mathcal{R}(A)^{\bot}}\|_2^2\leq \beta^l\|b_{\mathcal{R}(A)}\|_2^2\leq\|b_{\mathcal{R}(A)}\|_2^2
\end{equation}
holds for $l \geq 0$. Define $k_p:=\left \lfloor k/2 \right \rfloor$ and $k_q:=k-k_p$, with $k$ being a nonnegative integer. Following by Eqs.\eqref{3-15} and \eqref{3-16} yields
\begin{equation}\label{3-17}
\begin{aligned}
\mathbb{E}\|x_{k_p}-x_{\star}\|_2^2&\leq \alpha\mathbb{E}\|x_{k_p-1}-x_{\star}\|_2^2+ \frac{1}{t}\mathbb{E}\|b_{\mathcal{R}(A)^{\bot}}-z_{k_p-1}\|_2^2   \\
&\leq \alpha\mathbb{E}\|x_{k_p-1}-x_{\star}\|_2^2+ \frac{\|b_{\mathcal{R}(A)}\|_2^2}{t}\\
&\leq \alpha^{k_p}\|x_{0}-x_{\star}\|_2^2+\sum_{l=0}^{k_p-1}\alpha^{l}\frac{\|b_{\mathcal{R}(A)}\|_2^2}{t}\\
&\leq\|x_{\star}\|_2^2+\sum_{l=0}^{\infty }\alpha^{l}\frac{\|b_{\mathcal{R}(A)}\|_2^2}{t}\\
&\leq\|x_{\star}\|_2^2+\frac{\|A\|_F^2}{\lambda_{min}(A^*A)}\frac{\|b_{\mathcal{R}(A)}\|_2^2}{t}.
\end{aligned}
\end{equation}
For any $l\geq 0$, it follows that
\begin{equation}\label{3-18}
\mathbb{E}\|z_{l+k_p}-b_{\mathcal{R}(A)^{\bot}}\|_2^2\leq \beta^{l+k_p}\|b_{\mathcal{R}(A)}\|_2^2\leq\beta^{k_p}\|b_{\mathcal{R}(A)}\|_2^2.
\end{equation}
Consequently, by Eqs.\eqref{3-17} and \eqref{3-18}, it follows that
$$
\begin{aligned}
\mathbb{E}\|x_{k}-x_{\star}\|_2^2&\leq \alpha^{k_q}
\mathbb{E}\|x_{k_p}-x_{\star}\|_2^2+\beta^{k_p}\sum_{l=0}^{k_q-1}\alpha^{l}\frac{\|b_{\mathcal{R}(A)}\|_2^2}{t}\\
&\leq \alpha^{k_q}\left(\|x_{\star}\|_2^2+\frac{\|A\|_F^2}{\lambda_{min}(A^*A)}\frac{\|b_{\mathcal{R}(A)}\|_2^2}{t}\right)+\beta^{k_p}\frac{\|A\|_F^2}{\lambda_{min}(A^*A)}\frac{\|b_{\mathcal{R}(A)}\|_2^2}{t}\\
&\leq \alpha^{k_q}\|x_{\star}\|_2^2+(\alpha^{k_q}+\beta^{k_p})\frac{\|A\|_F^2}{\lambda_{min}(A^*A)}\frac{\|b_{\mathcal{R}(A)}\|_2^2}{t}\\
&\leq\alpha^{k_q}\|x_{\star}\|_2^2+(\alpha^{k_q}+\beta^{k_p})\frac{\|A\|_F^2}{\lambda_{min}(A^*A)}\frac{\lambda_{min}(A^*A)}{t}\|x_{\star}\|_2^2\\
&\leq(\max\{\alpha,\beta\})^{k_p}\left(1+2\frac{\|A\|_F^2}{t}\right)\|x_{\star}\|_2^2.
\end{aligned}
$$
\end{proof}

Note that GREK has to construct the indicator sets $\mathcal{U}_k$, $\bar{\mathcal{U}}_k$, and compute probabilities during iterations, which requires much more storage and computational operations compared to methods without involving these formulas, especially in large data problems. To deal with these difficulties, Jiang and Wu \cite{Jiang} proposed a semi-randomized Kaczmarz method, which outperforms the GRK method. They showed that the row corresponding to the current largest homogeneous residual is selected as the working row for higher computational efficiency. Likewise, we present a semi-randomized extended Kaczmarz method (SREK), built upon selecting the row and column corresponding to the current largest homogeneous residuals, for solving least-squares
problem \eqref{1-1}. We first give the following semi-randomized orthogonal projection method, which efficiently computes the orthogonal projection of $b$ onto the column space of $A$.
\begin{algorithm}
	\caption{A semi-randomized orthogonal projection}
	\label{algo-3-3}
	\begin{algorithmic}[1]
        \REQUIRE $A, b$\\
        \ENSURE $z_{k+1}$\\
        \STATE
        Initialize $z_0= b$
        \STATE
       For $k=0, 1, 2,\cdots $, until convergence, do:\\
            \STATE Select $i_k\in\{1,2,\dots,n\} $ satisfying $\frac{|A_{(i_k)}z_k|}{\|A_{(i_k)}\|_2}=\max\limits_{1\leq i \leq n}\frac{|A_{(i)}z_k|}{\|A_{(i)}\|_2}.$
    \STATE
 Compute $z_{k+1} =(I_m-\frac{A_{(i_k)}A_{(i_k)}^*}{\|A_{(i_k)}\|_2^2})z_{k}.$
	\end{algorithmic}  
\end{algorithm}

An important result for guaranteeing the convergence of Algorithm \ref{algo-3-3}  is stated in the following theorem.
\begin{theorem}\label{The-3}
{\rm  Let $b$ be the initial guess of Algorithm \ref{algo-3-3}. Then the iterative sequence $\{z_{k}\}$ generated by Algorithm \ref{algo-3-3} converges to $b_{\mathcal{R}(A)^{\bot}}$ in expectation. Moreover, the corresponding error norm in expectation yields }
\end{theorem}
\begin{equation}\label{3-19}
\mathbb{E}_{k}\|z_k-b_{\mathcal{R}(A)^{\bot}}\|_2^2\leq \left(1-\frac{\lambda_{min}(A^*A)}{\tilde{\tau}_{max}}\right)^k\|b_{\mathcal{R}(A)}\|_2^2,   
\end{equation}
where $\tilde{\tau}_{max}=\|A\|_F^2-\min\limits_{i\in[n]}\|A_{(i)}\|_2^2$.
\begin{proof}
From Algorithm \ref{algo-3-3}, it can be stated that
\begin{equation}\label{3-20}
\begin{aligned}
\max_{1\leq i \leq n}\left \{ \frac{|A_{(i)}^*e_{k-1}|^2}{\|A_{(i)}\|_2^2} \right \}&=\max_{1\leq i \leq n}\left \{ \frac{|A_{(i)}^*z_{k-1}|^2}{\|A_{(i)}\|_2^2} \right \}\\
&=\frac{\max_{1\leq i \leq n}\left \{ \frac{|A_{(i)}^*z_{k-1}|^2}{\|A_{(i)}\|_2^2} \right \}}{\|A^*z_{k-1}\|_2^2}\|A^*z_{k-1}\|_2^2\\
&=\frac{\max_{1\leq i \leq n}\left \{ \frac{|A_{(i)}^*z_{k-1}|^2}{\|A_{(i)}\|_2^2} \right \}}{\sum_{i=1}^{n}\frac{|A_{i}^*z_{k-1}|^2}{\|A_{(i)}\|_2^2}\|A_{(i)}\|_2^2}\|A^*z_{k-1}\|_2^2\\
&=\frac{\max_{1\leq i \leq n}\left \{ \frac{|A_{(i)}^*z_{k-1}|^2}{\|A_{(i)}\|_2^2} \right \}}{\sum_{i=1,i\neq i_{k-1}}^{n}\frac{|A_{(i)}^*z_{k-1}|^2}{\|A_{(i)}\|_2^2}\|A_{(i)}\|_2^2}\|A^*z_{k-1}\|_2^2\\
&\geq \frac{\|A^*z_{k-1}\|_2^2}{\sum_{i=1,i\neq i_{k-1}}^{n}\|A_{(i)}\|_2^2}\\
&\geq \frac{\|A^*z_{k-1}\|_2^2}{\tilde{\tau}_{max}}\\
&\geq \frac{\lambda_{min}(A^*A)}{\tilde{\tau}_{max}}\|e_{k-1}\|_2^2.
\end{aligned}
\end{equation}
Observe from Algorithm \ref{algo-3-3} and Eq.\eqref{3-5} that
$$
e_k=P_{i_k}e_{k-1}.
$$
Consequently, from Eq.\eqref{3-20} and the above relations, it follows that
$$
\begin{aligned}
\mathbb{E}_{k}\|e_k\|_2^2&=\mathbb{E}_{k}\|P_{i_k}e_{k-1}\|_2^2\\
&=\mathbb{E}_{k}\left \langle P_{i_k}e_{k-1},P_{i_k}e_{k-1}  \right \rangle \\
&=\mathbb{E}_{k}\left \langle e_{k-1},P_{i_k}e_{k-1}  \right \rangle\\
&=\|e_{k-1}\|_2^2-\mathbb{E}_{k}\left \langle e_{k-1},\frac{A_{(i_k)}A_{(i_k)}^*}{\|A_{(i_k)}\|_2^2}e_{k-1}  \right \rangle\\
&=\|e_{k-1}\|_2^2-\max_{1\leq i \leq n}\left \{ \frac{|A_{i}^*e_{k-1}|^2}{\|A_{i}\|_2^2} \right \} \\
&\leq (1-\frac{\lambda_{min}(A^*A)}{\tilde{\tau}_{max}})\|e_{k-1}\|_2^2.
\end{aligned}
$$
The theorem follows immediately by induction.
\end{proof}

In what follows, we present the semi-randomized extended Kaczmarz method for solving Eq.\eqref{1-1}. Indeed, the proposed method can be regarded as a specific combination of the semi-randomized orthogonal projection method together with the semi-randomized Kaczmarz method.
\begin{algorithm}
	\caption{A semi-randomized extended Kaczmarz method}
	\label{algo-3-4}
	\begin{algorithmic}[1]
        \REQUIRE $A, b,z_0= b$\\
        \ENSURE $x_{k+1}$, $z_{k+1}$\\
        \STATE
       For $k=0, 1, 2,\cdots $, until convergence, do:\\
            \STATE Select $i_k\in\{1,2,\dots,m\} $ satisfying $$\frac{|{b}^{(i_{k})}-A^{(i_{k})}x_{k}|} {\| A^{(i_{k})}\|_{2} }=\max_{1\leq i \leq m}\frac{|{b}^{(i)}-A^{(i)}x_{k}|} {\| A^{(i)}\|_{2} }.$$
            \STATE Select $j_k\in\{1,2,\dots,n\} $ satisfying $$\frac{|A_{(j_k)}^*z_k|}{\|A_{(j_k)}\|_2}=\max_{1\leq j \leq n}\frac{|A_{(j)}^*z_k|}{\|A_{(j)}\|_2}.$$
            \STATE
            Set $z_{k+1} =(I_m-\frac{A_{(j_k)}A_{(j_k)}^*}{\|A_{(j_k)}\|_2^2})z_{k}.$\\
            \STATE
            Set $x_{k+1} =x_k+\frac{{b}^{(i_{k})}-z_k^{(i_k)}-A^{(i_{k})}x_{k}} {\| A^{(i_{k})}\|_{2} ^{2} }(A^{(i_{k})})^{*}$.

	\end{algorithmic}  
\end{algorithm}

An important result for guaranteeing the convergence of Algorithm \ref{algo-3-4}  is stated in the following theorem.
\begin{theorem}\label{The-4}
{\rm Let $x_{\star}=A^{\dagger}b$ be the  least-norm solution of the least-squares
problem \eqref{1-1}. Then the iterative sequence $\{x_{k}\}$ generated by Algorithm \ref{algo-3-4} converges to $x_{\star }$ for any initial vector $x_0$ in expectation. Moreover, the corresponding error norm in expectation yields
}
\begin{equation}\label{3-21}
\mathbb{E}\|x_{k}-x_{\star}\|_2^2\leq(\max\{\hat{\alpha},\hat{\beta}\})^{\left \lfloor \frac{k}{2} \right \rfloor }\left(1+2\frac{\tau_{max}}{t}\right)\|x_{\star}\|_2^2,    
\end{equation}
\rm where $\hat{\alpha}=1-\frac{\lambda_{min}(A^*A)}{\tau_{max}}$ and $\hat{\beta}=1-\frac{\lambda_{min}(A^*A)}{\tilde{\tau}_{max}}$ with $\tau_{max}=\|A\|_F^2-\min\limits_{i\in[m]}\|A^{(i)}\|_2^2$, $\tilde{\tau}_{max}=\|A\|_F^2-\min\limits_{i\in[n]}\|A_{(i)}\|_2^2$ and $t=\min\limits_{i\in[m]}\|A^{(i)}\|_2^2$.
\end{theorem}
\begin{proof}

Similar to Theorem \ref{The-2}, we have
\begin{equation}\label{3-22}
\mathbb{E}_k\|x_{k+1}-x_{\star}\|_2^2=\mathbb{E}_k\|\tilde{x}_{k+1}-x_{\star}\|_2^2+\mathbb{E}_k\|x_{k+1}-\tilde{x}_{k+1}\|_2^2.
\end{equation}
For the first term of the right-hand side, it follows that
\begin{equation}\label{3-23}
\begin{aligned}
\mathbb{E}_k\|\tilde{x}_{k+1}-x_{\star}\|_2^2&=\|x_k-x_{\star}\|_2^2-\max_{i_k\in[m]}\left\{\frac{|b_{\mathcal{R}(A)}^{(i_{k})}-A^{(i_{k})}x_k|^2}{\|A^{(i_{k})}\|_{2}^{2}}\right\}\\
&\leq (1-\frac{\lambda_{min}(A^*A)}{\tau_{max}})\|x_k-x_{\star}\|_2^2.
\end{aligned}
\end{equation}
For the second term of the right-hand side of Eq.\eqref{3-22}, it follows that
\begin{equation}\label{3-24}
\begin{aligned}
\mathbb{E}_k\|x_{k+1}-\tilde{x}_{k+1}\|_2^2=\frac{|w^{(i_{k})}|^2}{\| A^{(i_{k})}\|_{2} ^{2}}=\frac{|b_{\mathcal{R}(A)^{\bot}}^{(i_{k})}-z_k^{(i_{k})}|^2}{\| A^{(i_{k})}\|_{2} ^{2}}\leq\frac{\|b_{\mathcal{R}(A)^{\bot}}-z_{k}\|_2^2}{t},
\end{aligned}
\end{equation}
where $t=\min\limits_{i\in[m]}\|A^{(i)}\|_2^2$. Substituting \eqref{3-23} and \eqref{3-24} into \eqref{3-22} yields
\begin{equation}\label{3-25}
\mathbb{E}_k\|x_{k+1}-x_{\star}\|_2^2\leq \hat{\alpha}\|x_k-x_{\star}\|_2^2+\frac{\|b_{\mathcal{R}(A)^{\bot}}-z_{k}\|_2^2}{t}.
\end{equation}
According to Theorem \ref{The-3}, it can be stated that
\begin{equation}\label{3-26}
\mathbb{E}\|z_l-b_{\mathcal{R}(A)^{\bot}}\|_2^2\leq \hat{\beta}^l\|b_{\mathcal{R}(A)}\|_2^2\leq\|b_{\mathcal{R}(A)}\|_2^2
\end{equation}
holds for $l \geq 0$.
Define $k_p:=\left \lfloor k/2 \right \rfloor$ and $k_q:=k-k_p$, with $k$ being a nonnegative integer. It follows from Eqs.\eqref{3-25} and \eqref{3-26} that
\begin{equation}\label{3-27}
\begin{aligned}
\mathbb{E}\|x_{k_p}-x_{\star}\|_2^2&\leq \hat{\alpha}\mathbb{E}\|x_{k_p-1}-x_{\star}\|_2^2+ \frac{1}{t}\mathbb{E}\|b_{\mathcal{R}(A)^{\bot}}-z_{k_p-1}\|_2^2   \\
&\leq \hat{\alpha}\mathbb{E}\|x_{k_p-1}-x_{\star}\|_2^2+ \frac{\|b_{\mathcal{R}(A)}\|_2^2}{t}\\
&\leq \hat{\alpha}^{k_p}\|x_{0}-x_{\star}\|_2^2+\sum_{l=0}^{k_p-1}\hat{\alpha}^{l}\frac{\|b_{\mathcal{R}(A)}\|_2^2}{t}\\
&\leq\|x_{\star}\|_2^2+\sum_{l=0}^{\infty }\hat{\alpha}^{l}\frac{\|b_{\mathcal{R}(A)}\|_2^2}{t\|_2^2}\\
&=\|x_{\star}\|_2^2+\frac{\tau_{max}}{\lambda_{min}(A^*A)}\frac{\|b_{\mathcal{R}(A)}\|_2^2}{t}.\\
\end{aligned}
\end{equation}
Therefore,
\begin{equation}\label{3-28}
\mathbb{E}\|z_{l+k_p}-b_{\mathcal{R}(A)^{\bot}}\|_2^2\leq \hat{\beta}^{l+k_p}\|b_{\mathcal{R}(A)}\|_2^2\leq\hat{\beta}^{k_p}\|b_{\mathcal{R}(A)}\|_2^2 .
\end{equation}
Consequently, by Eqs.\eqref{3-27} and \eqref{3-28}, it follows that
$$
\begin{aligned}
\mathbb{E}\|x_{k}-x_{\star}\|_2^2&\leq \hat{\alpha}^{k_q}
\mathbb{E}\|x_{k_p}-x_{\star}\|_2^2+\hat{\beta}^{k_p}\sum_{l=0}^{k_q-1}\hat{\alpha}^{l}\frac{\|b_{\mathcal{R}(A)}\|_2^2}{t}\\
&\leq \hat{\alpha}^{k_q}\left(\|x_{\star}\|_2^2+\frac{\tau_{max}}{\lambda_{min}(A^*A)}\frac{\|b_{\mathcal{R}(A)}\|_2^2}{t}\right)+\hat{\beta}^{k_p}\frac{\tau_{max}}{\lambda_{min}(A^*A)}\frac{\|b_{\mathcal{R}(A)}\|_2^2}{t}\\
&\leq \hat{\alpha}^{k_q}\|x_{\star}\|_2^2+(\hat{\alpha}^{k_q}+\hat{\beta}^{k_p})\frac{\tau_{max}}{\lambda_{min}(A^*A)}\frac{\|b_{\mathcal{R}(A)}\|_2^2}{t}\\
&\leq\hat{\alpha}^{k_q}\|x_{\star}\|_2^2+(\hat{\alpha}^{k_q}+\hat{\beta}^{k_p})\frac{\tau_{max}}{\lambda_{min}(A^*A)}\frac{\lambda_{min}(A^*A)}{t}\|x_{\star}\|_2^2\\
&\leq(\max\{\hat{\alpha},\hat{\beta}\})^{k_p}\left(1+2\frac{\tau_{max}}{t}\right)\|x_{\star}\|_2^2.
\end{aligned}
$$
\end{proof}
\subsection{Two-dimensional randomized extended Kaczmarz method with simple random sampling}\label{subsection-3-2}
The two-subspace randomized extended Kaczmarz \cite{WWT} is an excellent solver for least-squares problems. As stated earlier, each iteration of the Kaczmarz method is a typical one-dimensional projection method. Note that an approximate solution given by a suitable two-dimensional subspace may converge faster than one generated by the subspace $\mathcal K_{0}$. Let $\mathcal{K}_1= \mathrm{span}\left \{(A^{(i_{k_1})})^*, (A^{(i_{k_2})})^* \right \}$, $\mathcal{K}_2= \mathrm{span}\left \{A_{(j_{k_1})}, A_{(j_{k_2})} \right \}$ and $\mathcal{L}_1=\mathrm{span}\left \{e _{i_{k_1}}, e  _{i_{k_2}}  \right \}$, $\mathcal{L}_2=\mathrm{span}\left \{e _{j_{k_1}}, e  _{j_{k_2}}  \right \}$ be two search spaces and two constraint spaces, respectively, where $e _{i_{k_1}}, e _{i_{k_2}} $ and $e _{j_{k_1}}, e _{j_{k_2}} $ are the $i_{k_1}$-th, $i_{k_2}$-th and $j_{k_1}$-th, $j_{k_2}$-th standard basis vectors, respectively.  We find that the approximate solutions $x_{k+1}$ and $z_{k+1}$ yield the following conditions
\begin{equation}\label{3-29}
\mbox{Find}\ x_{k+1}\in  x_{k}+\mathcal K_{1},\ \mbox{such that}\ b-{A}x_{k+1} \perp \mathcal L_{1},
\end{equation}
\begin{equation}\label{3-30}
\mbox{Find}\ z_{k+1}\in  z_{k}+\mathcal K_{2},\ \mbox{such that}\ {A}^*z_{k+1} \perp \mathcal L_{2}. 
\end{equation}
This shows that $x_{k+1}$ and $z_{k+1}$ can be expressed as
\begin{equation}\label{3-31}
x_{k+1} =x_{k} + \gamma_k(A^{(i_{k_1})})^{*} + \lambda_k(A^{(i_{k_2})})^{*},
\end{equation}
\begin{equation}\label{3-32}
z_{k+1} =z_{k} + \tilde{\gamma}_kA_{(i_{k_1})} + \tilde{\lambda}_kA_{(i_{k_2})},
\end{equation}
respectively,
where $\gamma_k$, $\tilde{\gamma}_k$, $\lambda_k$, and $\tilde{\lambda}_k$ are scalars. 
Here, if $A^{(i_{k_1})}$ is parallel to $A^{(i_{k_2})}$, then Eq.\eqref{3-31} reduces to
$$
x_{k+1} =x_k+\frac{{b}^{(i_{k_1})}-A^{(i_{k_1})}x_{k}} {\| A^{(i_{k_1})}\|_{2} ^{2} }(A^{(i_{k_1})})^{*},
$$
otherwise, it follows from Eq.\eqref{3-29} that
$$\left\{ \begin{array}{l}
 \langle 
b-Ax_{k+1} ,e _{i_{k_1}}  \rangle =0,\\
 \langle 
b-Ax_{k+1} ,e _{i_{k_2}}   \rangle =0,
 \end{array}
\right.$$
that is,
$$
\left\{ \begin{array}{l}
b _{i_{k_1}}-A^{(i_{k_1})}x_k-A^{(i_{k_1})}(A^{(i_{k_1})})^*\gamma_k-A^{(i_{k_1})}(A^{(i_{k_2})})^*\lambda_k=0,\\
b _{i_{k_2}}-A^{(i_{k_2})}x_k-A^{(i_{k_2})}(A^{(i_{k_1})})^*\gamma_k-A^{(i_{k_2})}(A^{(i_{k_2})})^*\lambda_k=0.
\end{array}
\right.$$
Therefore, the parameters $\gamma_k$ and $\lambda_k$ 
are given by
$$\begin{aligned}
\gamma_k&=\frac{\|A^{(i_{k_2})}\|_2^2({b}^{(i_{k_1})}-A^{(i_{k_1})}x_k)-A^{(i_{k_1})}(A^{(i_{k_2})})^*({b}^{(i_{k_2})}-A^{(i_{k_2})}x_k)}{\|A^{(i_{k_1})}\|_2^2\|A^{(i_{k_2})}\|_2^2-|A^{(i_{k_1})}(A^{(i_{k_2})})^*|^2},\\
\lambda_k&=\frac{\|A^{(i_{k_1})}\|_2^2({b}^{(i_{k_2})}-A^{(i_{k_2})}x_k)-A^{(i_{k_2})}(A^{(i_{k_1})})^*({b}^{(i_{k_1})}-A^{(i_{k_1})}x_k)}{\|A^{(i_{k_1})}\|_2^2\|A^{(i_{k_2})}\|_2^2-|A^{(i_{k_1})}(A^{(i_{k_2})})^*|^2}.
\end{aligned}$$
Similarly, if $A_{(j_{k_1})}$ is parallel to $A_{(j_{k_2})}$, then Eq.\eqref{3-32} reduces to
$$
z_{k+1} =(I_m-\frac{A_{(j_{k_1})}A_{(j_{k_1})}^*}{\|A_{(j_{k_1})}\|_2^2})z_{k},
$$
otherwise, 
$$\begin{aligned}
\tilde{\gamma}_k&=\frac{A_{(j_{k_1})}^*A_{(j_{k_2})}A_{(j_{k_2})}^*z_k-\|A_{(j_{k_2})}\|_2^2A_{(j_{k_1})}^*z_k}{\|A_{(j_{k_1})}\|_2^2\|A_{(j_{k_2})}\|_2^2-|A_{(j_{k_1})}^*A_{(j_{k_2})}|^2},\\
\tilde{\lambda}_k&=\frac{A_{(j_{k_2})}^*A_{(j_{k_1})}A_{(j_{k_1})}^*z_k-\|A_{(j_{k_1})}\|_2^2A_{(j_{k_2})}^*z_k}{\|A_{(j_{k_1})}\|_2^2\|A_{(j_{k_2})}\|_2^2-|A_{(j_{k_1})}^*A_{(j_{k_2})}|^2}.
\end{aligned}$$

    Indeed, TREK may be time-consuming for big data problems, since we have to compute the probability of the whole data, such that select two active rows and columns at each iteration. To address this issue, we performed the TREK with simple random sampling to save storage and computational operations. As the simplest sampling method, simple random sampling \cite{Olken} only requires a portion of the advanced knowledge with a single random selection. By this selection method, all the individuals have an equal probability of being selected, which ensures the unbiased of the dataset. Hence, we take a small portion of rows and columns as samples and then select working rows and columns from the obtained samples. The main advantage is that we only need to compute the probability corresponding to the selected simple sampling set, rather than the whole dataset. This idea stems from Chebyshev's (weak) law of large numbers.

\begin{theorem}\cite{Carlton}\label{The-5}
\rm Denote by $\{z_1,\cdots,z_n,\cdots\}$, a series of random variables, by $\mathbb{E}(z_k)$, the expectation, and by $\mathbb{D}(z_k)$, the variance of $z_k$. Then, when $\frac{1}{n^2}\sum_{k=1}^{n}\mathbb{D}(z_k)$\\$\rightarrow 0$, the following relation 
\begin{equation}\label{3-33}
\lim_{n \to \infty} P\left \{ \left|\frac{1}{n}\sum _{k=1}^{n}z_k-\frac{1}{n}\sum _{k=1}^{n}\mathbb{E} (z_k) \right|<\varepsilon \right \} =1    
\end{equation}
holds for any small positive number $\varepsilon$.
\end{theorem}

Observe from Theorem \ref{The-5} that, if the sample size is large enough, then the sample mean will be close to the population mean, i.e., a small part estimating the whole is reasonable. Putting this fact together with TREK, a two-dimensional randomized extended Kaczmarz method with simple random sampling for Eq.\eqref{1-1}, is as follows.

\vspace{1em}
\begin{breakablealgorithm}
	\caption{Two-dimensional randomized extended Kaczmarz method with simple random sampling}
	\label{algo-3-5}
	\begin{algorithmic}[1]
        \REQUIRE $A, b,z_0= b$\\
        \ENSURE $x_{k+1}$, $z_{k+1}$\\
        \STATE
       For $k=0, 1, 2,\cdots $, until convergence, do:\\
            \STATE
            Generate two indicator sets $\Gamma_{k}$, $\tilde{\Gamma}_{k}$ i.e., choosing $l m$ rows and $ln$ columns of $A$ by using the simple random sampling, where $0<l<1$.
            \STATE Select $i_{k_1}\in\Gamma_k $ with probability Pr(row=$i_{k_1}$)=$\frac{\|A^{(i_{k_1})}\|_2^2}{\sum_{i\in\Gamma_k}\|A^{(i)}\|_2^2}$.
            \STATE Select $i_{k_2}\in\Gamma_k $ with probability Pr(row=$i_{k_2}$)=$\frac{\|A^{(i_{k_2})}\|_2^2}{\sum_{j\in\Gamma_k}\|A^{(j)}\|_2^2}$.

            \STATE Select $j_{k_1}\in\tilde{\Gamma}_{k} $ with probability Pr(column=$j_{k_1}$)=$\frac{\|A_{(j_{k_1})}\|_2^2}{\sum_{i\in\tilde{\Gamma}_{k}}\|A_{(i)}\|_2^2}$.
            
            \STATE Select $j_{k_2}\in\tilde{\Gamma}_{k} $ with probability Pr(column=$j_{k_2}$)=$\frac{\|A_{(j_{k_2})}\|_2^2}{\sum_{j\in\tilde{\Gamma}_{k}}\|A_{(j)}\|_2^2}$.
            
 \STATE  If $A^{(i_{k_1})}$ is parallel to $A^{(i_{k_2})}$, then we compute
$$
x_{k+1} =x_k+\frac{{b}^{(i_{k_1})}-z_k^{(i_{k_1})}-A^{(i_{k_1})}x_{k}} {\| A^{(i_{k_1})}\|_{2} ^{2} }(A^{(i_{k_1})})^{*},
$$
otherwise, compute
$$\begin{aligned}
\gamma_k&=\frac{\|A^{(i_{k_2})}\|_2^2r_k^{(i_{k_1})}-A^{(i_{k_1})}(A^{(i_{k_2})})^*r_k^{(i_{k_2})}}{\|A^{(i_{k_1})}\|_2^2\|A^{(i_{k_2})}\|_2^2-|A^{(i_{k_1})}(A^{(i_{k_2})})^*|^2},\\
\lambda_k&=\frac{\|A^{(i_{k_1})}\|_2^2r_k^{(i_{k_2})}-A^{(i_{k_2})}(A^{(i_{k_1})})^*r_k^{(i_{k_1})}}{\|A^{(i_{k_1})}\|_2^2\|A^{(i_{k_2})}\|_2^2-|A^{(i_{k_1})}(A^{(i_{k_2})})^*|^2},
\end{aligned}$$
with $r_k^{(i_{k_s})}={b}^{(i_{k_s})}-z_{k}^{(i_{k_s})}-A^{(i_{k_s})}x_k$, $s=1,2.$
 Update $x_{k+1} =x_{k} + \gamma_k(A^{(i_{k_1})})^{*} + \lambda_k(A^{(i_{k_2})}) ^{*}$.
  \STATE  If $A_{(j_{k_1})}$ is parallel to $A_{(j_{k_2})}$, then we compute
$$
z_{k+1} =(I_m-\frac{A_{(j_{k_1})}A_{(j_{k_1})}^*}{\|A_{(j_{k_1})}\|_2^2})z_{k},
$$
otherwise, compute
$$\begin{aligned}
\tilde{\gamma}_k&=\frac{A_{(j_{k_1})}^*A_{(j_{k_2})}A_{(j_{k_2})}^*z_k-\|A_{(j_{k_2})}\|_2^2A_{(j_{k_1})}^*z_k}{\|A_{(j_{k_1})}\|_2^2\|A_{(j_{k_2})}\|_2^2-|A_{(j_{k_1})}^*A_{(j_{k_2})}|^2},\\
\tilde{\lambda}_k&=\frac{A_{(j_{k_2})}^*A_{(j_{k_1})}A_{(j_{k_1})}^*z_k-\|A_{(j_{k_1})}\|_2^2A_{(j_{k_2})}^*z_k}{\|A_{(j_{k_1})}\|_2^2\|A_{(j_{k_2})}\|_2^2-|A_{(j_{k_1})}^*A_{(j_{k_2})}|^2}.
\end{aligned}$$
 Update $z_{k+1} =z_{k} + \tilde{\gamma}_kA_{(j_{k_1})} + \tilde{\lambda}_kA_{(j_{k_2})}$.
	\end{algorithmic}  
\end{breakablealgorithm}
\vspace{1em}

Actually, if the linear system is consistent, then $z_k$ is not required, such that Algorithm \ref{algo-3-5} will reduce to the following two-dimensional randomized Kaczmarz method with simple random sampling (TRKS).
\vspace{1em}
\begin{breakablealgorithm}
	\caption{Two-dimensional randomized Kaczmarz method with simple random sampling}
	\label{algo-3-6}
	\begin{algorithmic}[1]
        \REQUIRE $A, b$\\
        \ENSURE $x_{k+1}$\\
        \STATE
       For $k=0, 1, 2,\cdots $, until convergence, do:\\
            \STATE
            Generate an indicator set $\Gamma_k$, i.e., choosing $l m$ rows of $A$ by using the simple random sampling, where $0<l<1$.
            \STATE Select $i_{k_1}\in\Gamma_k $ with probability Pr(row=$i_{k_1}$)=$\frac{\|A^{(i_{k_1})}\|_2^2}{\sum_{i\in\Gamma_k}\|A^{(i)}\|_2^2}$.
            \STATE Select $i_{k_2}\in\Gamma_k $ with probability Pr(row=$i_{k_2}$)=$\frac{\|A^{(i_{k_2})}\|_2^2}{\sum_{j\in\Gamma_k}\|A^{(j)}\|_2^2}$.
        
 \STATE  If $A^{(i_{k_1})}$ is parallel to $A^{(i_{k_2})}$, then we compute
$$
x_{k+1} =x_k+\frac{{b}^{(i_{k})}-A^{(i_{k})}x_{k}} {\| A^{(i_{k})}\|_{2} ^{2} }(A^{(i_{k})})^{*},
$$
otherwise, compute
$$\begin{aligned}
\gamma_k&=\frac{\|A^{(i_{k_2})}\|_2^2({b}^{(i_{k_1})}-A^{(i_{k_1})}x_k)-A^{(i_{k_1})}(A^{(i_{k_2})})^*({b}^{(i_{k_2})}-A^{(i_{k_2})}x_k)}{\|A^{(i_{k_1})}\|_2^2\|A^{(i_{k_2})}\|_2^2-|A^{(i_{k_1})}(A^{(i_{k_2})})^*|^2},\\
\lambda_k&=\frac{\|A^{(i_{k_1})}\|_2^2({b}^{(i_{k_2})}-A^{(i_{k_2})}x_k)-A^{(i_{k_2})}(A^{(i_{k_1})})^*({b}^{(i_{k_1})}-A^{(i_{k_1})}x_k)}{\|A^{(i_{k_1})}\|_2^2\|A^{(i_{k_2})}\|_2^2-|A^{(i_{k_1})}(A^{(i_{k_2})})^*|^2}.
\end{aligned}$$
 Update $x_{k+1} =x_{k} + \gamma_k(A^{(i_{k_1})})^{*} + \lambda_k(A^{(i_{k_2})}) ^{*}$.

	\end{algorithmic}  
\end{breakablealgorithm}
\vspace{1em}

To analyze the convergence of TREKS, we first provide the following lemma, and refer to Appendix \ref{App-1} for its proof.

\begin{lemma}\label{lem-1}
{\rm Let $x_{\star}=A^{\dagger}b$ be the least-norm solution of consistent linear systems $Ax=b$. Then the iterative sequence $\{x_{k}\}$ generated by Algorithm \ref{algo-3-6} converges to $x_{\star }$ for any initial vector $x_0$ in expectation. Moreover, the corresponding error norm in expectation yields
\begin{equation}\label{3-34}
\begin{aligned}
\mathbb{E}_k\|x_{k+1}-x_{\star}\|_2^2&\leq\left[\left(1-\frac{1-\varepsilon_k}{1+\tilde{\varepsilon}_k}\frac{\lambda_{min}(A^*A)}{\tau_{max}}\right)\left(1-\frac{1-\varepsilon_{k_1}}{1+\tilde{\varepsilon}_{k_1}}\frac{\lambda_{min}(A^*A)}{\|A\|_F^2}\right)\right.\\
&\quad\left.-\frac{1-\varepsilon_{k_2}}{1+\tilde{\varepsilon}_{k_2}}\frac{D\lambda_{min}(A^*A)}{\|A\|_F^2}\frac{\tau_{min}}{\tau_{max}}\right]\|x_k-x_{\star}\|_2^2.
\end{aligned}
\end{equation}   
where $D=\min\left\{\frac{\delta^2(1-\delta)}{1+\delta},\frac{\Delta^2(1-\Delta)}{1+\Delta}\right\}$ with $\delta$ and $\Delta$ being defined as in \eqref{2-2}, and $\tau_{min}=\|A\|_F^2-\max\limits_{i\in[m]}\|A^{(i)}\|_2^2$ , $\tau_{max}=\|A\|_F^2-\min\limits_{i\in[m]}\|A^{(i)}\|_2^2$}.
\end{lemma}

From Lemma \ref{lem-1}, we give the following theorem for guaranteeing the convergence of Algorithm TREKS.
\begin{theorem}\label{The-6}
{\rm Let $x_{\star}=A^{\dagger}b$ be the least-norm solution of the least-squares
problem \eqref{1-1}. Then the iterative sequence $\{x_{k}\}$ generated by Algorithm \ref{algo-3-5} converges to $x_{\star }$ for any initial vector $x_0$ in expectation. Moreover, the corresponding error norm in expectation yields
\begin{equation}\label{3-35}
\mathbb{E}\|x_{k}-x_{\star}\|_2^2\leq(\max\{\tilde{\alpha},\tilde{\beta}\})^{\left \lfloor \frac{k}{2} \right \rfloor }\left(1+\frac{1+\hat{\varepsilon }_k}{1-\check{\varepsilon }_k }\frac{4}{(1-\triangle)(1-\tilde{\alpha})}\frac{\tau_{max}}{\tau_{min}}\frac{\lambda_{min}(A^*A)}{\|A\|_F^2}\right)\|x_{\star}\|_2^2,
\end{equation}
where 
$$
\tilde{\alpha}=\left(1-\frac{1-\varepsilon_k}{1+\tilde{\varepsilon}_k}\frac{\lambda_{min}(A^*A)}{\tau_{max}}\right)\left(1-\frac{1-\varepsilon_{k_1}}{1+\tilde{\varepsilon}_{k_1}}\frac{\lambda_{min}(A^*A)}{\|A\|_F^2}\right)-\frac{1-\varepsilon_{k_2}}{1+\tilde{\varepsilon}_{k_2}}\frac{D\lambda_{min}(A^*A)\tau_{min}}{\|A\|_F^2\tau_{max}}$$
and
$$
\tilde{\beta}=\left(1-\frac{1-\overline{\varepsilon_k}}{1+\overline{\tilde{\varepsilon}_k }}\frac{\lambda_{min}(A^*A)}{\tilde{\tau}_{max}}\right)\left(1-\frac{1-\overline{\varepsilon_{k_1}}}{1+\overline{\tilde{\varepsilon}_{k_1} }}\frac{\lambda_{min}(A^*A)}{\|A\|_F^2}\right)-\frac{1-\overline{\varepsilon_{k_2}}}{1+\overline{\tilde{\varepsilon }_{k_2}}}\frac{\tilde{D}\lambda_{min}(A^*A)\tilde{\tau}_{min}}{\|A\|_F^2\tilde{\tau}_{max}}$$}
\end{theorem}
and $D=\min\left\{\frac{\delta^2(1-\delta)}{1+\delta},\frac{\Delta^2(1-\Delta)}{1+\Delta}\right\}$, $\tilde{D}=\min\left\{\frac{\tilde{\delta}^2(1-\tilde{\delta})}{1+\tilde{\delta}},\frac{\tilde{\Delta}^2(1-\tilde{\Delta})}{1+\tilde{\Delta}}\right\}$ with $\delta$, $\tilde{\delta}$ and $\Delta$, $\tilde{\Delta}$ being defined as in Eqs. \eqref{2-2} and \eqref{2-3}, and $\tau_{min}=\|A\|_F^2-\max\limits_{i\in[m]}\|A^{(i)}\|_2^2$ , $\tilde{\tau}_{min}=\|A\|_F^2-\max\limits_{i\in[n]}\|A_{(i)}\|_2^2$ , $\tau_{max}=\|A\|_F^2-\min\limits_{i\in[m]}\|A^{(i)}\|_2^2$ , $\tilde{\tau}_{max}=\|A\|_F^2-\min\limits_{i\in[n]}\|A_{(i)}\|_2^2$.
\begin{proof}
Similar to Theorem \ref{The-2}.
Define 
\begin{equation} \label{3-36}  
\begin{aligned}
\tilde{x}_{k+1}&=x_k+\frac{\|A^{(i_{k_2})}\|_2^2({b}^{(i_{k_1})}-A^{(i_{k_1})}x_k)-A^{(i_{k_1})}(A^{(i_{k_2})})^*({b}^{(i_{k_2})}-A^{(i_{k_2})}x_k)}{\|A^{(i_{k_1})}\|_2^2\|A^{(i_{k_2})}\|_2^2-|A^{(i_{k_1})}(A^{(i_{k_2})})^*|^2}(A^{(i_{k_1})}) ^{*}\\
&\quad+\frac{\|A^{(i_{k_1})}\|_2^2({b}^{(i_{k_2})}-A^{(i_{k_2})}x_k)-A^{(i_{k_2})}(A^{(i_{k_1})})^*({b}^{(i_{k_1})}-A^{(i_{k_1})}x_k)}{\|A^{(i_{k_1})}\|_2^2\|A^{(i_{k_2})}\|_2^2-|A^{(i_{k_1})}(A^{(i_{k_2})})^*|^2}(A^{(i_{k_2})}) ^{*}.
\end{aligned}
\end{equation}
As shown in Algorithm \ref{algo-3-6}, $\tilde{x}_{k+1}$ is the projection of $x_{k}$ on the intersection of two hyperplanes $\mathcal{H}_{i_{k_1}}:=\left\{x:\left \langle A^{(i_{k_1})},x \right \rangle =b_{\mathcal{R}(A)}^{(i_{k_1})}\right\}$ and $\mathcal{H}_{i_{k_2}}:=\left\{x:\left \langle A^{(i_{k_2})},x \right \rangle =b_{\mathcal{R}(A)}^{(i_{k_2})}\right\}$. Moreover,
\begin{equation}\label{3-37}   
x_{k+1}=x_k+\bar{\gamma}_k(A^{(i_{k_1})}) ^{*}+\bar{\lambda}_k(A^{(i_{k_2})}) ^{*},
\end{equation}
where $\bar{\gamma}_k=\frac{\|A^{(i_{k_2})}\|_2^2(b_{\mathcal{R}(A)}^{(i_{k_1})}+w^{(i_{k_1})}-A^{(i_{k_1})}x_k)-A^{(i_{k_1})}(A^{(i_{k_2})})^*({b}^{(i_{k_2})}b_{\mathcal{R}(A)}^{(i_{k_2})}+w^{(i_{k_2})}-A^{(i_{k_2})}x_k)}{\|A^{(i_{k_1})}\|_2^2\|A^{(i_{k_2})}\|_2^2-|A^{(i_{k_1})}(A^{(i_{k_2})})^*|^2}$\\
and ${\bar{\lambda}_k=\frac{\|A^{(i_{k_1})}\|_2^2(b_{\mathcal{R}(A)}^{(i_{k_2})}+w^{(i_{k_2})}-A^{(i_{k_2})}x_k)-A^{(i_{k_2})}(A^{(i_{k_1})})^*(b_{\mathcal{R}(A)}^{(i_{k_1})}+w^{(i_{k_1})}-A^{(i_{k_1})}x_k)}{\|A^{(i_{k_1})}\|_2^2\|A^{(i_{k_2})}\|_2^2-|A^{(i_{k_1})}(A^{(i_{k_2})})^*|^2}}$.\\
Then $x_{k+1}$ is the projection of $x_{k}$ on the intersection of two hyperplanes $\mathcal{H}_{i_{k_1}}^{w}:=\left\{x:\left \langle A^{(i_{k_1})},x \right \rangle =b_{\mathcal{R}(A)}^{(i_{k_1})}+w^{(i_{k_1})}\right\}$ and $\mathcal{H}_{i_{k_2}}^{w}:=\left\{x:\left \langle A^{(i_{k_2})},x \right \rangle =b_{\mathcal{R}(A)}^{(i_{k_2})}+w^{(i_{k_2})}\right\}$.
It follows that
$$
\begin{aligned}
x_{k+1}-\tilde{x}_{k+1}&=\frac{\|A^{(i_{k_2})}\|_2^2w^{(i_{k_1})}-A^{(i_{k_1})}(A^{(i_{k_2})})^*w^{(i_{k_2})}}{\|A^{(i_{k_1})}\|_2^2\|A^{(i_{k_2})}\|_2^2-|A^{(i_{k_1})}(A^{(i_{k_2})})^*|^2}(A^{(i_{k_1})})^*\\
&\quad+\frac{\|A^{(i_{k_1})}\|_2^2w^{(i_{k_2})}-A^{(i_{k_2})}(A^{(i_{k_1})})^*w^{(i_{k_1})}}{\|A^{(i_{k_1})}\|_2^2\|A^{(i_{k_2})}\|_2^2-|A^{(i_{k_1})}(A^{(i_{k_2})})^*|^2}(A^{(i_{k_2})})^*
\end{aligned}
$$
and
$$
\begin{aligned}
\tilde{x}_{k+1}-x_{\star}&=\left(I-\frac{\|A^{(i_{k_2})}\|_2^2(A^{(i_{k_1})})^{*}A^{(i_{k_1})}+(A^{(i_{k_1})})^{*}A^{(i_{k_1})}(A^{(i_{k_2})})^{*}A^{(i_{k_2})}}{ \|A^{(i_{k_1})}\|_2^2\|A^{(i_{k_2})}\|_2^2-|A^{(i_{k_1})}(A^{(i_{k_2})})^*|^2}\right.\\
&\quad\left.-\frac{\|A^{(i_{k_1})}\|_2^2(A^{(i_{k_2})})^{*}A^{(i_{k_2})}+(A^{(i_{k_2})})^{*}A^{(i_{k_2})}(A^{(i_{k_1})})^{*}A^{(i_{k_1})}}{ \|A^{(i_{k_1})}\|_2^2\|A^{(i_{k_2})}\|_2^2-|A^{(i_{k_1})}(A^{(i_{k_2})})^*|^2}\right)(x_k-x_{\star}).
\end{aligned}
$$
From the above results, it is trivial to see that $(x_{k+1}-\tilde{x}_{k+1})^*(\tilde{x}_{k+1}-x_{\star})=0$. This shows that
\begin{equation}\label{3-38}
\mathbb{E}_k\|x_{k+1}-x_{\star}\|_2^2=\mathbb{E}_k\|\tilde{x}_{k+1}-x_{\star}\|_2^2+\mathbb{E}_k\|x_{k+1}-\tilde{x}_{k+1}\|_2^2.
\end{equation}
Following Chebyshev’s (weak) law of large numbers, if 
$lm$ is sufficiently large and there are two scalars $0<\hat{\varepsilon}_k,\check{\varepsilon}_k\ll1$, then 
\begin{equation}\label{3-39}
\begin{aligned}
&\quad\frac{\left(\sum_{\substack{i_{k_1}\in\Gamma_k }}\|A^{(i_{k_1})}\|_2^2\sum_{\substack{i_{k_2}\in\Gamma_k\\i_{k_2}\neq i_{k_1}}}\|A^{(i_{k_2})}\|_2^2\right)\|x_{k+1}-\tilde{x}_{k+1}\|_2^2}{(lm)^2-lm}\\
&=\frac{\left(\sum\limits_{\substack{i_{k_1}\in[m] }}\|A^{(i_{k_1})}\|_2^2\sum\limits_{\substack{i_{k_2}\in[m]\\i_{k_2}\neq i_{k_1}}}\|A^{(i_{k_2})}\|_2^2\right)\|x_{k+1}-\tilde{x}_{k+1}\|_2^2}{(m)^2-m}(1\pm\hat{\varepsilon }_k)
\end{aligned}
\end{equation}
and
\begin{equation}\label{3-40}
\begin{aligned}
&\quad\frac{\sum_{\substack{i_{k_1}\in\Gamma_k }}\|A^{(i_{k_1})}\|_2^2\sum_{\substack{i_{k_2}\in\Gamma_k\\i_{k_2}\neq i_{k_1}}}\|A^{(i_{k_2})}\|_2^2}{(lm)^2-lm}\\
&=\frac{\sum_{\substack{i_{k_1}\in[m] }}\|A^{(i_{k_1})}\|_2^2\sum\limits_{\substack{i_{k_2}=1\\i_{k_2}\neq i_{k_1}}}^{m}\|A^{(i_{k_2})}\|_2^2}{m^2-m}(1\pm\check{\varepsilon }_k )\\
&=\frac{\|A\|_F^2(\|A\|_F^2-\|A^{(i_{k_1})}\|_2^2)}{m^2-m}(1\pm\check{\varepsilon }_k) . 
\end{aligned}
\end{equation}
Then, by Eqs.\eqref{add-1}, \eqref{3-39}, and \eqref{3-40}, it follows that
\begin{equation}\label{3-41}
\begin{aligned}
\mathbb{E}_k\|x_{k+1}-\tilde{x}_{k+1}\|_2^2&=\tilde{P}_{k_1}\tilde{P}_{k_2}\|x_{k+1}-\tilde{x}_{k+1}\|_2^2\\
&\leq\frac{(1\pm\hat{\varepsilon }_k)}{(1\pm\check{\varepsilon }_k )}\frac{1}{\|A\|_F^2\tau_{min}}\sum\limits_{\substack{i_{k_1}\in[m] }}\|A^{(i_{k_1})}\|_2^2\sum\limits_{\substack{i_{k_2}\in[m]\\i_{k_2}\neq i_{k_1}}}\|A^{(i_{k_2})}\|_2^2\\
&\quad\|x_{k+1}-\tilde{x}_{k+1}\|_2^2\\
&\leq\frac{1+\hat{\varepsilon }_k}{1-\check{\varepsilon }_k }\frac{1}{\|A\|_F^2\tau_{min}}\sum\limits_{\substack{i_{k_1}\in[m] }}\sum\limits_{\substack{i_{k_2}\in[m]\\i_{k_2}\neq i_{k_1}}}\|A^{(i_{k_1})}\|_2^2\|A^{(i_{k_2})}\|_2^2\\
&\left(|s_{k_2}|^2+\left|\frac{1}{\|u_k\|_2}s_{k_1}-\frac{|\mu_k|}{\|u_k\|_2}s_{k_2}\right|^2\right)\\
&
\end{aligned}
\end{equation}
with $s_{k_r}=\frac{b_{\mathcal{R}(A)^{\bot}}^{(i_{k_r})}-z_k^{(i_{k_r})}}{\|A^{(i_{k_r})}\|_2}$, $r=1,2$. Define $\xi=\frac{1}{\|u_k\|_2}$ and $\eta=\frac{|\mu_k|}{\|u_k\|_2}$, for $p,q\in[m]$ and $p\neq q$, it follows that

\begin{equation}\label{3-42}
\begin{aligned}
&\sum\limits_{{i_{k_1}} \in [m]}\sum\limits_{{i_{k_2}} \in [m],i_{k_2}\neq i_{k_1}}\|A^{(i_{k_1})}\|_2^2\|A^{(i_{k_2})}\|_2^2\left|\frac{1}{\|u_k\|_2}s_{k_1}-\frac{|\mu_k|}{\|u_k\|_2}s_{k_2}\right|^2\\
&=\sum\limits_{p<q}\|A^{(p)}\|_2^2\|A^{(q)}\|_2^2\left(|\xi s_{p}-\eta s_{q}|^2+|\xi s_{q}-\eta s_{p}|^2 \right)\\
&\leq\sum\limits_{p<q}\|A^{(p)}\|_2^2\|A^{(q)}\|_2^2(\xi+\eta)^2(s_{p}^2+s_{q}^2)\\
\end{aligned}
\end{equation}
$$\begin{aligned}
&\leq\frac{1+\Delta}{1-\Delta}\sum\limits_{p<q}\left(\|A^{(p)}\|_2^2\left|b_{\mathcal{R}(A)^{\bot}}^{(q)}-z_k^{(q)}\right|^2+\|A^{(q)}\|_2^2\left|b_{\mathcal{R}(A)^{\bot}}^{(p)}-z_k^{(p)}\right|^2\right)\\
&=\frac{1+\Delta}{1-\Delta}\sum\limits_{p\in [m]}\sum\limits_{q\in [m],q\neq p}\|A^{(q)}\|_2^2\left|b_{\mathcal{R}(A)^{\bot}}^{(p)}-z_k^{(p)}\right|^2\\
&\leq\frac{1+\Delta}{1-\Delta}\tau_{max}\|b_{\mathcal{R}(A)^{\bot}}-z_k\|_2^2.
\end{aligned}
$$
The first inequality follows by the fact that for any $\xi$,$\eta$,$a$,$b\in\mathbb{C}$,
$$|\xi a-\eta b|^2+|\xi b-\bar{\eta} a|^2\leq(|\xi|+|\eta|)^2(|a|^2+|b|^2).$$ 
The second inequality follows from the fact that
$$
(|\xi|+|\eta|)^2=(\frac{1+|\mu_k|}{\|u_k\|_2})^2=\frac{1+|\mu_k|}{1-|\mu_k|}\leq\frac{1+\triangle }{1-\triangle }.
$$
Besides, it holds that
\begin{equation}\label{3-43}
\sum\limits_{{i_{k_1}} \in [m]}\sum\limits_{{i_{k_2}} \in [m],i_{k_2}\neq i_{k_1}}\|A^{(i_{k_1})}\|_2^2\|A^{(i_{k_2})}\|_2^2|s_{k_2}|^2\leq\tau_{max}\|b_{\mathcal{R}(A)^{\bot}}-z_k\|_2^2
\end{equation}
Therefore, following by Eqs.\eqref{3-42} and \eqref{3-43} yields
\begin{equation}\label{3-44}
\mathbb{E}_k\|x_{k+1}-\tilde{x}_{k+1}\|_2^2\leq\frac{1+\hat{\varepsilon }_k}{1-\check{\varepsilon }_k }\frac{2}{(1-\triangle)\|A\|_F^2}\frac{\tau_{max}}{\tau_{min}}\|b_{\mathcal{R}(A)^{\bot}}-z_{k}\|_2^2   
\end{equation}
Since $\tilde{x}_{k+1}$ can be obtained by TRKS  for the consistent linear system $Ax=b_{\mathcal{R}(A)}$, from Lemma \ref{lem-1} and $x_k\in\mathcal{R}(A^*)$, we have
\begin{equation}\label{3-45}
\mathbb{E}_k\|\tilde{x}_{k+1}-x_{\star}\|_2^2\leq\tilde{\alpha}\|x_k-x_{\star}\|_2^2.
\end{equation}
Substituting \eqref{3-44} and \eqref{3-45} into Eq.\eqref{3-38} yields
\begin{equation}\label{3-46}
\mathbb{E}_k\|x_{k+1}-x_{\star}\|_2^2\leq \tilde{\alpha}\|x_k-x_{\star}\|_2^2+\frac{1+\hat{\varepsilon }_k}{1-\check{\varepsilon }_k }\frac{2}{(1-\triangle)\|A\|_F^2}\frac{\tau_{max}}{\tau_{min}}\|b_{\mathcal{R}(A)^{\bot}}-z_{k}\|_2^2. 
\end{equation}
One can see that $z_k$ generated by TREKS is equivalent to $z_k$ generated by TRKS for a consistent linear system $A^*z=0$. Thus, it follows from
$z_0-b_{\mathcal{R}(A)^{\bot}}=b_{\mathcal{R}(A)}\in\mathcal{R}(A)$ and Lemma \ref{lem-1} that
$$
\mathbb{E}_k\|b_{\mathcal{R}(A)^{\bot}}-z_{k+1}\|_2^2\leq\tilde{\beta}\|x_k-x_{\star}\|_2^2.
$$
For any $l\geq 0$, it follows that
\begin{equation}\label{3-47}
\mathbb{E}\|z_l-b_{\mathcal{R}(A)^{\bot}}\|_2^2\leq \tilde{\beta}^l\|b_{\mathcal{R}(A)}\|_2^2\leq\|b_{\mathcal{R}(A)}\|_2^2.
\end{equation}
Similar to Theorem \ref{The-2}, for any nonnegative integer $k$, define $k_p:=\left \lfloor k/2 \right \rfloor$ and $k_q:=k-k_p$. According to \eqref{3-46} and \eqref{3-47}, we have
\begin{equation}\label{3-48}
\begin{aligned}
\mathbb{E}\|x_{k_p}-x_{\star}\|_2^2&\leq \tilde{\alpha}\mathbb{E}\|x_{k_p-1}-x_{\star}\|_2^2+ \frac{1+\hat{\varepsilon }_k}{1-\check{\varepsilon }_k }\frac{2}{(1-\triangle)\|A\|_F^2}\frac{\tau_{max}}{\tau_{min}}\|b_{\mathcal{R}(A)^{\bot}}-z_{k_{p}-1}\|_2^2\\
&\leq\tilde{\alpha}^{k_p}\|x_{0}-x_{\star}\|_2^2+\frac{1+\hat{\varepsilon }_k}{1-\check{\varepsilon }_k }\frac{2}{(1-\triangle)\|A\|_F^2}\frac{\tau_{max}}{\tau_{min}}\sum_{l=0}^{k_p-1}\tilde{\alpha}^{l}\|b_{\mathcal{R}(A)}\|_2^2\\
&\leq\tilde{\alpha}^{k_p}\|x_{\star}\|_2^2+\frac{1+\hat{\varepsilon }_k}{1-\check{\varepsilon }_k }\frac{2}{(1-\triangle)\|A\|_F^2}\frac{\tau_{max}}{\tau_{min}}\sum_{l=0}^{k_p-1}\tilde{\alpha}^{l}\|b_{\mathcal{R}(A)}\|_2^2\\
&\leq\tilde{\alpha}^{k_p}\|x_{\star}\|_2^2+\frac{1+\hat{\varepsilon }_k}{1-\check{\varepsilon }_k }\frac{2}{(1-\triangle)\|A\|_F^2}\frac{\tau_{max}}{\tau_{min}}\sum_{l=0}^{\infty }\tilde{\alpha}^{l}\|b_{\mathcal{R}(A)}\|_2^2\\
&=\tilde{\alpha}^{k_p}\|x_{\star}\|_2^2+\frac{1+\hat{\varepsilon }_k}{1-\check{\varepsilon }_k }\frac{2}{(1-\triangle)(1-\alpha)}\frac{\tau_{max}}{\tau_{min}}\frac{\|b_{\mathcal{R}(A)}\|_2^2}{\|A\|_F^2}
\end{aligned}  
\end{equation}
Moreover,
\begin{equation}\label{3-49}
\mathbb{E}\|z_{l+k_p}-b_{\mathcal{R}(A)^{\bot}}\|_2^2\leq \tilde{\beta}^{l+k_p}\|b_{\mathcal{R}(A)}\|_2^2\leq\tilde{\beta}^{k_p}\|b_{\mathcal{R}(A)}\|_2^2
\end{equation}
holds for $l\geq 0$.
Consequently, by Eqs.\eqref{3-48}, \eqref{3-49} and the fact $b_{\mathcal{R}(A)}=Ax_{\star}$, it follows that
\begin{equation}\label{3-50}
\begin{aligned}
\mathbb{E}\|x_{k}-x_{\star}\|_2^2&\leq \tilde{\alpha}^{k_q}
\mathbb{E}\|x_{k_p}-x_{\star}\|_2^2+\frac{1+\hat{\varepsilon }_k}{1-\check{\varepsilon }_k }\frac{2}{(1-\triangle)\|A\|_F^2}\frac{\tau_{max}}{\tau_{min}}\tilde{\beta}^{k_p}\sum_{l=0}^{k_q-1}\tilde{\alpha}^{l}\|b_{\mathcal{R}(A)}\|_2^2\\
&\leq\tilde{\alpha}^{k_q}\left( \tilde{\alpha}^{k_p}\|x_{\star}\|_2^2+\frac{1+\hat{\varepsilon }_k}{1-\check{\varepsilon }_k }\frac{2}{(1-\triangle)(1-\tilde{\alpha})}\frac{\tau_{max}}{\tau_{min}}\frac{\|b_{\mathcal{R}(A)}\|_2^2}{\|A\|_F^2}\right)\\
&\quad+\tilde{\beta}^{k_p}\frac{1+\hat{\varepsilon }_k}{1-\check{\varepsilon }_k }\frac{2}{(1-\triangle)(1-\tilde{\alpha})}\frac{\tau_{max}}{\tau_{min}}\frac{\|b_{\mathcal{R}(A)}\|_2^2}{\|A\|_F^2}\\
&=\tilde{\alpha}^{k}\|x_{\star}\|_2^2+(\tilde{\alpha}^{k_q}+\tilde{\beta}^{k_p})\frac{1+\hat{\varepsilon }_k}{1-\check{\varepsilon }_k }\frac{2}{(1-\triangle)(1-\tilde{\alpha})}\frac{\tau_{max}}{\tau_{min}}\frac{\|b_{\mathcal{R}(A)}\|_2^2}{\|A\|_F^2}\\
&\leq\tilde{\alpha}^{k}\|x_{\star}\|_2^2+(\tilde{\alpha}^{k_q}+\tilde{\beta}^{k_p})\frac{1+\hat{\varepsilon }_k}{1-\check{\varepsilon }_k }\frac{2}{(1-\triangle)(1-\tilde{\alpha})}\frac{\tau_{max}}{\tau_{min}}\frac{\lambda_{min}(A^*A)}{\|A\|_F^2}\|x_{\star}\|_2^2\\
&\leq(\max\{\tilde{\alpha},\tilde{\beta}\})^{k_p}\left(1+\frac{1+\hat{\varepsilon }_k}{1-\check{\varepsilon }_k }\frac{4}{(1-\triangle)(1-\tilde{\alpha})}\frac{\tau_{max}}{\tau_{min}}\frac{\lambda_{min}(A^*A)}{\|A\|_F^2}\right)\|x_{\star}\|_2^2.
\end{aligned}
\end{equation}
The proof is completed.
\end{proof}

\subsection{Two-dimensional greedy randomized extended Kaczmarz method}\label{subsection-3-3}
As well known that the crucial key of randomized Kaczmarz-type methods is to select the active row effectively. Similar to the GREK method, we apply the greedy strategy to the two-dimensional extended Kaczmarz method, that is, at each iteration, grab the rows and columns corresponding to the two larger elements in the corresponding residual vector, respectively. Putting this result with Algorithm \ref{algo-3-5} gives the following algorithm.
\vspace{1em}
\begin{breakablealgorithm}
	\caption{Two-dimensional greedy randomized extended Kaczmarz method}
	\label{algo-3-7}
	\begin{algorithmic}[1]
        \REQUIRE $A, b,z_0= b$\\
        \ENSURE $x_{k+1}$, $z_{k+1}$\\
        \STATE
       For $k=0, 1, 2,\cdots $, until convergence, do:\\
       \STATE        Compute\\
        \begin{equation}\label{3-51}
         \epsilon  _{k} = \frac{1}{2}\left(\frac{1}{{\parallel b - Ax_{k}{\parallel _{2}^2}}}\mathop {\max }\limits_{\substack{i_k \in [m]}} \left\{ \frac{|{b}^{(i_{k})}-z_{k}^{(i_{k})}-A^{(i_{k})}x_{k}|^2} {\| A^{(i_{k})}\|_{2}^2 }\right\} + \frac{1}{{\parallel A\parallel _{F}^2}}\right),
        \end{equation}
        \begin{equation}\label{3-52}
         \bar{\epsilon} _{k} = \frac{1}{2}\left(\frac{1}{{\parallel  A^*z_{k}{\parallel _{2}^2}}}\mathop {\max }\limits_{\substack{j_k \in [n]}} \left\{ \frac{|A_{(j_{k})}^*z_k|^2}{\|A_{(j_{k})}\|_2^2}\right\} + \frac{1}{{\parallel A\parallel _{F}^2}}\right),
        \end{equation}
         \STATE
         Determine the index sets of positive integers\\
         \begin{equation}\label{3-53}
         \mathcal{U}_k = \left\{ {{i_k}\Big| {|{b}^{(i_{k})}-z_{k}^{(i_{k})}-A^{(i_{k})}x_{k}|^2 \ge {\epsilon_k}} \parallel b - 
         Ax_{k}{\parallel _{2}^2}}\| A^{(i_{k})}\|_{2}^2 \right\}.\end{equation}\\
         \begin{equation}\label{3-54}
         \bar{\mathcal{U}}_k = \left\{ {j_k}\Big| {|A^*_{(j_{k})}z_{k}|^2 \ge \bar{\epsilon}_k} \parallel 
         A^*{z_{k}}{\parallel _{2}^2}\| A_{(j_{k})}\|_{2}^2 \right\}.\end{equation}\\
         \STATE
   Compute the $i$-th element of the vector 
     $\tilde{\mathbf{r}} _{k_1}$ by  \\
     $$\tilde{\mathbf{r} } _{k_1} ^{(i)}=\left\{ \begin{array}{l}
    {{b}^{(i_{k})}-z_{k}^{(i_{k})}-A^{(i_{k})}x_{k}},\qquad  \mathrm{if} \quad i \in \mathcal{U}_{k}
     \\
    0.\qquad \qquad \quad \enspace \mathrm{otherwise}
    \end{array}
    \right.$$
        \STATE
        Select $i_{k_1} \in {\mathcal{U}_k}$ with probability $\mathbb{P}$(row =$i_{k_1}$)=$ \mid {{{\tilde {\mathbf{r} }}_{k_1}}^{(i_{k_1})}}\mid^2/\left \| {{{\tilde {\mathbf{r} }}}_{k_1}} \right \|_{2}^2 $.\\
        \STATE   
        Select $i_{k_2} \in {\mathcal{U}_k}$ with probability $\mathbb{P}$(row =$i_{k_2}$)=$ \mid {{{\tilde {\mathbf{r} }}_{k_1}}^{(i_{k_2})}}\mid^2  /\left \| {{{\tilde {\mathbf{r} }}}_{k_1}} \right \|_{2}^2 $.\\
         \STATE
   Compute the $j$-th element of the vector 
     $\tilde{\mathbf{r}} _{k_2}$ by  \\
     $$\tilde{\mathbf{r} } _{k_2} ^{(j)}=\left\{ \begin{array}{l}
    {A_{(j_{k})}^*z_k},\qquad  \mathrm{if} \quad j \in \bar{\mathcal{U}}_{k}
     \\
    0.\qquad \qquad \quad \enspace \mathrm{otherwise}
    \end{array}
    \right.$$
        \STATE
        Select $j_{k_1} \in {\bar{\mathcal{U}}_k}$ with probability $\mathbb{P}$(column =$j_{k_1}$)=$ \mid {{{\tilde {\mathbf{r} }}_{k_2}}^{(j_{k_1})}}\mid^2  /\left \| {{{\tilde {\mathbf{r} }}}_{k_2}} \right \|_{2}^2 $.\\
        \STATE   
        Select $j_{k_2} \in {\bar{\mathcal{U}}_k}$ with probability $\mathbb{P}$(column=$j_{k_2}$)=$ \mid {{{\tilde {\mathbf{r} }}_{k_2}}^{(j_{k_2})}}\mid^2  /\left \| {{{\tilde {\mathbf{r} }}}_{k_2}} \right \|_{2}^2 $.\\
     \STATE Update $x_{k+1}$ as the line $7$ of Algorithm \ref{algo-3-5}.

     \STATE Update $z_{k+1}$ as the line $8$ of Algorithm \ref{algo-3-5}.

	\end{algorithmic}  
\end{breakablealgorithm}
\vspace{1em}

Note that, if the linear system is consistent, then $z_k$ is not required, such that Algorithm \ref{algo-3-7} will reduce to the following two-dimensional greedy randomized Kaczmarz (TGRK) method.
\vspace{1em}
\begin{breakablealgorithm}
	\caption{Two-dimensional greedy randomized Kaczmarz method}
	\label{algo-3-8}
	\begin{algorithmic}[1]
        \REQUIRE $A, b$\\
        \ENSURE $x_{k+1}$\\
        \STATE
       For $k=0, 1, 2,\cdots $, until convergence, do:\\
       \STATE        Compute\\
        \begin{equation}\label{3-55}
         \epsilon  _{k} = \frac{1}{2}\left(\frac{1}{{\parallel b - Ax_{k}{\parallel _{2}^2}}}\mathop {\max }\limits_{\substack{i_k \in [m]}} \left\{ \frac{|{b}^{(i_{k})}-A^{(i_{k})}x_{k}|^2} {\| A^{(i_{k})}\|_{2}^2 }\right\} + \frac{1}{{\parallel A\parallel _{F}^2}}\right),
        \end{equation}
         \STATE
         Determine the index sets of positive integers\\
         \begin{equation}\label{3-56}
         \mathcal{U}_k = \left\{ {{i_k}\Big| {|{b}^{(i_{k})}-A^{(i_{k})}x_{k}|^2 \ge {\epsilon_k}} \parallel b - 
         Ax_{k}{\parallel _{2}^2}}\| A^{(i_{k})}\|_{2}^2 \right\}.\end{equation}\\
         \STATE
   Compute the $i$-th element of the vector 
     $\tilde{\mathbf{r}} _{k_1}$ by  \\
     $$\tilde{\mathbf{r} } _{k_1} ^{(i)}=\left\{ \begin{array}{l}
    {{b}^{(i_{k})}-A^{(i_{k})}x_{k}},\quad  \mathrm{if} \quad i \in \mathcal{U}_{k}
     \\
    0.\qquad \qquad \qquad \enspace \mathrm{otherwise}
    \end{array}
    \right.$$
        \STATE
        Select $i_{k_1} \in {\mathcal{U}_k}$ with probability $\mathbb{P}$(row =$i_{k_1}$)=$ \mid {{{\tilde {\mathbf{r} }}_{k_1}}^{(i_{k_1})}}\mid^2/\left \| {{{\tilde {\mathbf{r} }}}_{k_1}} \right \|_{2}^2 $.\\
        \STATE   
        Select $i_{k_2} \in {\mathcal{U}_k}$ with probability $\mathbb{P}$(row =$i_{k_2}$)=$ \mid {{{\tilde {\mathbf{r} }}_{k_1}}^{(i_{k_2})}}\mid^2  /\left \| {{{\tilde {\mathbf{r} }}}_{k_1}} \right \|_{2}^2 $.\\
     \STATE Update $x_{k+1}$ as the line $5$ of Algorithm \ref{algo-3-6}.

	\end{algorithmic}  
\end{breakablealgorithm}
\vspace{1em}

To analyze the convergence of TGREK, we first provide the following lemma, and refer to Appendix \ref{App-2} for its proof.

\begin{lemma}\label{lem-2}
{\rm Let $x_{\star}=A^{\dagger}b$ be the solution of consistent linear systems $Ax=b$. Then the iterative sequence $\{x_{k}\}$ generated by Algorithm \ref{algo-3-8} converges to $x_{\star }$ for any initial vector $x_0$ in expectation. Moreover, the corresponding error norm in expectation yields}
\begin{equation}\label{3-57}
\mathbb{E}_k\|x_{k+1}-x_{\star}\|_2^2\leq \left[1-\frac{1}{1+\Delta}\left(\frac{\|A\|_F^2}{\tau_{max}}+1\right)\frac{\lambda_{min}(A^*A)}{\|A\|_F^2}\right]\|x_k-x_{\star}\|_2^2.    
\end{equation}
\end{lemma}

From Lemma \ref{lem-2}, we give the following theorem for guaranteeing the convergence of Algorithm TGREK.

\begin{theorem}\label{The-7}
{\rm Let $x_{\star}=A^{\dagger}b$ be the  least-norm solution of the least-squares
problem \eqref{1-1}. Then the iterative sequence $\{x_{k}\}$ generated by Algorithm \ref{algo-3-7} converges to $x_{\star }$ for any initial vector $x_0$ in expectation. Moreover, the corresponding error norm in expectation yields

\begin{equation}\label{3-58}
\mathbb{E}\|x_{k}-x_{\star}\|_2^2\leq(\max\{\alpha_1,\beta_1\})^{\left \lfloor \frac{k}{2} \right \rfloor }\left(1+\frac{2\lambda_{min}(A^*A)}{t\tau_{min}(1-\alpha_1)}\left(\frac{1+\triangle }{1-\triangle }+\|A\|_F^2\right)\right)\|x_{\star}\|_2^2
\end{equation}
where $$\alpha_1=1-\frac{1}{1+\Delta}\left(\frac{\|A\|_F^2}{\tau_{max}}+1\right)\frac{\lambda_{min}(A^*A)}{\|A\|_F^2}$$
and
$$\beta_1=1-\frac{1}{1+\tilde{\Delta}}\left(\frac{\|A\|_F^2}{\tilde{\tau}_{max}}+1\right)\frac{\lambda_{min}(A^*A)}{\|A\|_F^2}.$$}
\end{theorem}
\begin{proof}
According to Eq.\eqref{3-38}, it can be stated that
\begin{equation}\label{3-59}
\mathbb{E}_k\|x_{k+1}-x_{\star}\|_2^2=\mathbb{E}_k\|\tilde{x}_{k+1}-x_{\star}\|_2^2+\mathbb{E}_k\|x_{k+1}-\tilde{x}_{k+1}\|_2^2.
\end{equation}
For the first term of the right-hand side, from Lemma \ref{lem-2} and $x_k\in\mathcal{R}(A^*)$, we have
\begin{equation}\label{3-60}
\mathbb{E}_k\|\tilde{x}_{k+1}-x_{\star}\|_2^2\leq\alpha_1\|x_k-x_{\star}\|_2^2,
\end{equation}
where $\tilde{x}_{k+1}$ is obtained by one TGRK iteration on the vector $x_k$ for the consistent linear system $Ax=b_{\mathcal{R}(A)}$. 
For the second term of the right-hand side of Eq.\eqref{3-59}, from Eq.\eqref{3-41}, we have
\begin{equation}\label{3-61}
\begin{aligned}
\|x_{k+1}-\tilde{x}_{k+1}\|_2^2
&\leq|s_{k_2}|^2+\left|\frac{1}{\|u_k\|_2}s_{k_1}-\frac{|\mu_k|}{\|u_k\|_2}s_{k_2}\right|^2
\end{aligned}
\end{equation}
with $s_{k_l}=\frac{b_{\mathcal{R}(A)^{\bot}}^{(i_{k_l})}-z_k^{(i_{k_l})}}{\|A^{(i_{k_l})}\|_2}$, $l=1,2$.  Moreover, for any $i\in[m]$, we have $\frac{|w^{(i)}|^2}{\| A^{(i)}\|_{2} ^{2}}=\frac{|b_{\mathcal{R}(A)^{\bot}}^{(i)}-z_k^{(i)}|^2}{\| A^{(i)}\|_{2} ^{2}}\leq\frac{1}{\| A^{(i)}\|_2^2}\|b_{\mathcal{R}(A)^{\bot}}-z_{k}\|_2^2\leq \frac{\|b_{\mathcal{R}(A)^{\bot}}-z_{k}\|_2^2}{t}$, where $t=\min\limits_{i\in[m]}\|A^{(i)}\|_2^2$.  
It is trivial to see that
$$
(|\xi|+|\eta|)^2=(\frac{1+|\mu_k|}{\|u_k\|_2})^2=\frac{1+|\mu_k|}{1-|\mu_k|}\leq\frac{1+\triangle }{1-\triangle},
$$
where $\xi=\frac{1}{\|u_k\|_2}$ and $\eta=\frac{|\mu_k|}{\|u_k\|_2}$. Then it follows that
\begin{equation}\label{3-62}
\begin{aligned}
\left|\frac{1}{\|u_k\|_2}s_{k_1}-\frac{|\mu_k|}{\|u_k\|_2}s_{k_2}\right|^2&\leq|\xi s_{k_1}|^2+|\eta s_{k_2}|^2+2|\xi s_{k_1}||\eta s_{k_2}|\\
&\leq |\xi|^2\frac{\|W\|_2^2}{t}+|\eta|^2\frac{\|W\|_2^2}{t}+2|\xi \eta||s_{k_2}s_{k_1}|\\
&\leq(|\xi|+|\eta|)^2\frac{\|W\|_2^2}{t}\\
&\leq\frac{1+\triangle }{1-\triangle }\frac{\|W\|_2^2}{t}
\end{aligned}
\end{equation}
and
\begin{equation}\label{3-63}
|s_{k_2}|^2\leq\frac{\|W\|_2^2}{t},
\end{equation}
where $W=b_{\mathcal{R}(A)^{\bot}}-z_{k}$.
Therefore, by Eqs. \eqref{3-62} and \eqref{3-63}, it follows that
\begin{equation}\label{3-64}
\mathbb{E}_k\|x_{k+1}-\tilde{x}_{k+1}\|_2^2\leq\frac{1}{t}\left(\frac{1+\triangle }{1-\triangle }+1\right)\|b_{\mathcal{R}(A)^{\bot}}-z_{k}\|_2^2.   
\end{equation}
Consequently, from the relations \eqref{3-60} and \eqref{3-64}, it follows that
\begin{equation}\label{3-65}
\mathbb{E}_k\|x_{k+1}-x_{\star}\|_2^2\leq \alpha_1\|x_k-x_{\star}\|_2^2+\frac{1}{t}\left(\frac{1+\triangle }{1-\triangle }+1\right)\|b_{\mathcal{R}(A)^{\bot}}-z_{k}\|_2^2.   
\end{equation}
In the TGREK method, the iteration for $z$ is equivalent to applying TGRK to solve a consistent linear system $A^*z=0$, where $z$ is the $m$-dimensional unknown vector and the least-norm solution is the zero vector. Thus, it follows from
$z_0-b_{\mathcal{R}(A)^{\bot}}=b_{\mathcal{R}(A)}\in\mathcal{R}(A)$ and Lemma \ref{lem-2} that
$$
\mathbb{E}_k\|b_{\mathcal{R}(A)^{\bot}}-z_{k+1}\|_2^2\leq\beta_1\|x_k-x_{\star}\|_2^2.
$$
For any $l\geq 0$, we have
\begin{equation}\label{3-66}
\mathbb{E}\|z_l-b_{\mathcal{R}(A)^{\bot}}\|_2^2\leq \beta_1^l\|b_{\mathcal{R}(A)}\|_2^2\leq\|b_{\mathcal{R}(A)}\|_2^2
\end{equation}
Similar to Theorem \ref{The-2}, for any nonnegative integer $k$, define $k_p:=\left \lfloor k/2 \right \rfloor$ and $k_q:=k-k_p$. 
Then it follows from inequalities \eqref{3-65} and \eqref{3-66} that
\begin{equation}\label{3-67}
\begin{aligned}
\mathbb{E}\|x_{k_p}-x_{\star}\|_2^2&\leq \alpha_1\mathbb{E}\|x_{k_p-1}-x_{\star}\|_2^2+ \frac{1}{t}\left(\frac{1+\triangle }{1-\triangle }+1\right)\|b_{\mathcal{R}(A)^{\bot}}-z_{k_{p}-1}\|_2^2\\
&\leq\alpha_1^{k_p}\|x_{0}-x_{\star}\|_2^2+\frac{1}{t}\left(\frac{1+\triangle }{1-\triangle }+1\right)\sum_{l=0}^{k_p-1}\alpha_1^{l}\|b_{\mathcal{R}(A)}\|_2^2\\
&\leq\alpha_1^{k_p}\|x_{\star}\|_2^2+\frac{1}{t}\left(\frac{1+\triangle }{1-\triangle }+1\right)\sum_{l=0}^{k_p-1}\alpha_1^{l}\|b_{\mathcal{R}(A)}\|_2^2\\
&\leq\alpha_1^{k_p}\|x_{\star}\|_2^2+\frac{1}{t}\left(\frac{1+\triangle }{1-\triangle }+1\right)\sum_{l=0}^{\infty }\alpha_1^{l}\|b_{\mathcal{R}(A)}\|_2^2\\
&=\alpha_1^{k_p}\|x_{\star}\|_2^2+\frac{1}{t(1-\alpha_1)}\left(\frac{1+\triangle }{1-\triangle }+1\right)\|b_{\mathcal{R}(A)}\|_2^2
\end{aligned}
\end{equation}
For any $l\geq 0$ we have
\begin{equation}\label{3-68}
\mathbb{E}\|z_{l+k_p}-b_{\mathcal{R}(A)^{\bot}}\|_2^2\leq \beta_1^{l+k_p}\|b_{\mathcal{R}(A)}\|_2^2\leq\beta_1^{k_p}\|b_{\mathcal{R}(A)}\|_2^2
\end{equation}
Then, by \eqref{3-67}, \eqref{3-68} and the fact $b_{\mathcal{R}(A)}=Ax_{\star}$, for any nonnegative integer $k$, it follows that
\begin{equation}\label{3-69}
\begin{aligned}
\mathbb{E}\|x_{k}-x_{\star}\|_2^2&\leq \alpha_1^{k_q}
\mathbb{E}\|x_{k_p}-x_{\star}\|_2^2+\frac{1}{t}\left(\frac{1+\triangle }{1-\triangle }+1\right)\beta_1^{k_p}\sum_{l=0}^{k_q-1}\tilde{\alpha}^{l}\|b_{\mathcal{R}(A)}\|_2^2\\
&\leq\alpha_1^{k_q}\left( \alpha_1^{k_p}\|x_{\star}\|_2^2+\frac{1}{t(1-\alpha_1)}\left(\frac{1+\triangle }{1-\triangle }+1\right)\|b_{\mathcal{R}(A)}\|_2^2\right)\\
&\quad+\beta_1^{k_p}\frac{1}{t(1-\tilde{\alpha})}\left(\frac{1+\triangle }{1-\triangle }+1\right)\|b_{\mathcal{R}(A)}\|_2^2\\
&=\alpha_1^{k}\|x_{\star}\|_2^2+(\alpha_1^{k_q}+\beta_1^{k_p})\frac{1}{t(1-\alpha_1)}\left(\frac{1+\triangle }{1-\triangle }+1\right)\|b_{\mathcal{R}(A)}\|_2^2\\
&\leq\alpha_1^{k}\|x_{\star}\|_2^2+(\alpha_1^{k_q}+\beta_1^{k_p})\frac{\lambda_{min}(A^*A)}{t(1-\alpha_1)}\left(\frac{1+\triangle }{1-\triangle }+1\right)\|x_{\star}\|_2^2\\
&\leq(\max\{\alpha_1,\beta_1\})^{k_p}\left(1+\frac{2\lambda_{min}(A^*A)}{t(1-\alpha_1)}\left(\frac{1+\triangle }{1-\triangle }+1\right)\right)\|x_{\star}\|_2^2.
\end{aligned}
\end{equation}
The proof is completed.

\end{proof}

\subsection{Two-dimensional semi-randomized extended Kaczmarz method}\label{subsection-3-4}
As seen, the TGREK solver is still constructing the index sets and computing probabilities during iterations, which is time-consuming. To overcome this issue, we propose a two-dimensional semi-random extended Kaczmarz (TSREK) method, built upon selecting the two active rows and columns corresponding to the current largest and the second largest homogeneous residuals, for solving the least-squares problem \eqref{1-1}.
\vspace{1em}
\begin{breakablealgorithm}
	\caption{Two-dimensional semi-randomized extended Kaczmarz method}
	\label{algo-3-9}
	\begin{algorithmic}[1]
        \REQUIRE $A, b,z_0= b$\\
        \ENSURE $x_{k+1}$, $z_{k+1}$\\
        \STATE
       For $k=0, 1, 2,\cdots $, until convergence, do:\\
            \STATE Select $i_{k_1}\in\{1,2,\dots,m\} $ satisfying $$\frac{|{b}^{(i_{k_1})}-z_{k}^{(i_{k_1})}-A^{(i_{k_1})}x_{k}|} {\| A^{(i_{k_1})}\|_{2} }=\max_{1\leq i \leq m}\frac{|{b}^{(i)}-z_{k}^{(i)}-A^{(i)}x_{k}|} {\| A^{(i)}\|_{2} }.$$
            \STATE Select $i_{k_2}\in\{1,2,\dots,m\} $ satisfying $$\frac{|{b}^{(i_{k_2})}-z_{k}^{(i_{k_2})}-A^{(i_{k_2})}x_{k}|} {\| A^{(i_{k_2})}\|_{2} }=\max_{1\leq i \leq m,i\neq i_{k_1}}\frac{|{b}^{(i)}-z_{k}^{(i)}-A^{(i)}x_{k}|} {\| A^{(i)}\|_{2} }.$$
            \STATE Select $j_{k_1}\in\{1,2,\dots,n\} $ satisfying $$\frac{|A_{(j_{k_1})}^*z_k|}{\|A_{(j_{k_1})}\|_2}=\max_{1\leq j \leq n}\frac{|A_{(j)}^*z_k|}{\|A_{(j)}\|_2}.$$
            \STATE Select $j_{k_2}\in\{1,2,\dots,n\} $ satisfying $$\frac{|A_{(j_{k_2})}^*z_k|}{\|A_{(j_{k_2})}\|_2}=\max_{1\leq j \leq n,j\neq j_{k_1}}\frac{|A_{(j)}^*z_k|}{\|A_{(j)}\|_2}.$$
     \STATE Update $x_{k+1}$ as the line $7$ of Algorithm \ref{algo-3-5}.

     \STATE Update $z_{k+1}$ as the line $8$ of Algorithm \ref{algo-3-5}.

	\end{algorithmic}  
\end{breakablealgorithm}
\vspace{1em}

Note that, if the linear system is consistent, then $z_k$ is not required, such that Algorithm \ref{algo-3-9} will reduce to the following two-dimensional semi-randomized Kaczmarz (TSRK) method.
\vspace{1em}
\begin{breakablealgorithm}
	\caption{Two-dimensional semi-randomized Kaczmarz method}
	\label{algo-3-10}
	\begin{algorithmic}[1]
        \REQUIRE $A, b$\\
        \ENSURE $x_{k+1}$\\
        \STATE
       For $k=0, 1, 2,\cdots $, until convergence, do:\\
            \STATE Select $i_{k_1}\in\{1,2,\dots,m\} $ satisfying $$\frac{|{b}^{(i_{k_1})}-A^{(i_{k_1})}x_{k}|} {\| A^{(i_{k_1})}\|_{2} }=\max_{1\leq i \leq m}\frac{|{b}^{(i)}-A^{(i)}x_{k}|} {\| A^{(i)}\|_{2} }.$$
            \STATE Select $i_{k_2}\in\{1,2,\dots,m\} $ satisfying $$\frac{|{b}^{(i_{k_2})}-A^{(i_{k_2})}x_{k}|} {\| A^{(i_{k_2})}\|_{2} }=\max_{1\leq i \leq m,i\neq i_{k_1}}\frac{|{b}^{(i)}-A^{(i)}x_{k}|} {\| A^{(i)}\|_{2} }.$$
     \STATE Update $x_{k+1}$ as the line $5$ of Algorithm \ref{algo-3-6}.

	\end{algorithmic}  
\end{breakablealgorithm}
\vspace{1em}

To analyze the convergence of TSREK, we first provide the following lemma, and refer to Appendix \ref{App-3} for its proof.
\begin{lemma}\label{lem-3}
{\rm Let $x_{\star}=A^{\dagger}b$ be the solution of consistent linear systems $Ax=b$. Then the iterative sequence $\{x_{k}\}$ generated by Algorithm \ref{algo-3-10} converges to $x_{\star }$ for any initial vector $x_0$ in expectation. Moreover, the corresponding error norm in expectation yields
\begin{equation}\label{3-70}
\mathbb{E}_k\|x_{k+1}-x_{\star}\|_2^2\leq \left(1-\frac{\lambda_{min}(A^*A)}{\tau_{max}}-\frac{\lambda_{min}(A^*A)}{\tau_{max}}\frac{\omega}{c^2(1-\delta^2)}\right)\|x_k-x_{\star}\|_2^2.    
\end{equation}
where $c$ and $\omega$ are defined in \eqref{5-1}.}
\end{lemma}

From Lemma \ref{lem-3}, we give the following theorem for guaranteeing the convergence of Algorithm TSREK.

\begin{theorem}\label{The-8}
{\rm Let $x_{\star}=A^{\dagger}b$ be the least-norm solution of the least-squares
problem \eqref{1-1}. Then the iterative sequence $\{x_{k}\}$ generated by Algorithm \ref{algo-3-9} converges to $x_{\star }$ for any initial vector $x_0$ in expectation. Moreover, the corresponding error norm in expectation yields
\begin{equation}\label{3-71}
\mathbb{E}\|x_{k}-x_{\star}\|_2^2\leq(\max\{\tilde{\alpha}_1,\tilde{\beta}_1\})^{\left \lfloor \frac{k}{2} \right \rfloor }\left(1+\frac{2\lambda_{min}(A^*A)}{t(1-\tilde{\alpha}_1)}\left(\frac{1+\triangle }{1-\triangle }+1\right)\right)\|x_{\star}\|_2^2
\end{equation}
where $$\tilde{\alpha}_1=1-\frac{\lambda_{min}(A^*A)}{\tau_{max}}-\frac{\lambda_{min}(A^*A)}{\tau_{max}}\frac{\omega}{c^2(1-\delta^2)}$$
and
$$\tilde{\beta}_1=1-\frac{\lambda_{min}(A^*A)}{\tilde{\tau}_{max}}-\frac{\lambda_{min}(A^*A)}{\tilde{\tau}_{max}}\frac{\tilde{\omega}}{\tilde{c}^2(1-\tilde{\delta}^2)}$$}
\end{theorem}
\begin{proof}
The proof is similar to Theorem \ref{The-7}, and hence we omit it.
\end{proof}

From Algorithm \ref{algo-3-9}, we can obtain that TSREK still requires the whole coefficient matrix to select working rows and columns, even though it avoids computing probabilities and constructing index sets. This is unfavorable for big data problems. Therefore, we introduce a two-dimensional semi-randomized extended Kaczmarz method with simple random sampling (TSREKS), which is powerful for big data problems.

\vspace{1em}
\begin{breakablealgorithm}
	\caption{Two-dimensional semi-randomized extended Kaczmarz method with simple random sampling}
	\label{algo-3-11}
	\begin{algorithmic}[1]
        \REQUIRE $A, b,z_0= b$\\
        \ENSURE $x_{k+1}$, $z_{k+1}$\\
        \STATE
       For $k=0, 1, 2,\cdots $, until convergence, do:\\
            \STATE
            Generate two indicator sets $\Phi_k$ and $\tilde{\Phi}_k$, i.e., choosing $\eta m$ rows and $\eta n$ columns of $A$ by using the simple random sampling, where $0<\eta<1$.\\
            \STATE Select $i_{k_1}\in\Phi_k $ satisfying $$\frac{|{b}^{(i_{k_1})}-z_{k}^{(i_{k_1})}-A^{(i_{k_1})}x_{k}|} {\| A^{(i_{k_1})}\|_{2} }=\max_{i\in\Phi_k}\frac{|{b}^{(i)}-z_{k}^{(i)}-A^{(i)}x_{k}|} {\| A^{(i)}\|_{2} }.$$
            \STATE Select $i_{k_2}\in\Phi_k $ satisfying $$\frac{|{b}^{(i_{k_2})}-z_{k}^{(i_{k_2})}-A^{(i_{k_2})}x_{k}|} {\| A^{(i_{k_2})}\|_{2} }=\max_{i\in\Phi_k,i\neq i_{k_1}}\frac{|{b}^{(i)}-z_{k}^{(i)}-A^{(i)}x_{k}|} {\| A^{(i)}\|_{2} }.$$
            \STATE Select $j_{k_1}\in\tilde{\Phi}_k $ satisfying $$\frac{|A_{(j_{k_1})}^*z_k|}{\|A_{(j_{k_1})}\|_2}=\max_{j\in\tilde{\Phi}_k}\frac{|A_{(j)}^*z_k|}{\|A_{(j)}\|_2}.$$
            \STATE Select $j_{k_2}\in\tilde{\Phi}_k $ satisfying $$\frac{|A_{(j_{k_2})}^*z_k|}{\|A_{(j_{k_2})}\|_2}=\max_{j\in\tilde{\Phi}_k,j\neq j_{k_1}}\frac{|A_{(j)}^*z_k|}{\|A_{(j)}\|_2}.$$
     \STATE Update $x_{k+1}$ as the line $7$ of Algorithm \ref{algo-3-5}.

     \STATE Update $z_{k+1}$ as the line $8$ of Algorithm \ref{algo-3-5}.

	\end{algorithmic}  
\end{breakablealgorithm}
\vspace{1em}

Note that, if the linear system is consistent, then $z_k$ is not required, such that Algorithm \ref{algo-3-11} will reduce to the following two-dimensional semi-randomized extended Kaczmarz method with simple random sampling (TSRKS).

\begin{algorithm}
	\caption{Two-dimensional semi-randomized
    Kaczmarz method with simple random sampling}
	\label{algo-3-12}
	\begin{algorithmic}[1]
        \REQUIRE $A, b$\\
        \ENSURE $x_{k+1}$\\
        \STATE
       For $k=0, 1, 2,\cdots $, until convergence, do:\\
            \STATE
            Generate an indicator set $\Phi_k$, i.e., choosing $\eta m$ rows of $A$ by using the simple random sampling, where $0<\eta<1$.\\
            \STATE Select $i_{k_1}\in\Phi_k $ satisfying $$\frac{|{b}^{(i_{k_1})}-A^{(i_{k_1})}x_{k}|} {\| A^{(i_{k_1})}\|_{2} }=\max_{i\in\Phi_k}\frac{|{b}^{(i)}-A^{(i)}x_{k}|} {\| A^{(i)}\|_{2} }.$$
            \STATE Select $i_{k_2}\in\Phi_k $ satisfying $$\frac{|{b}^{(i_{k_2})}-A^{(i_{k_2})}x_{k}|} {\| A^{(i_{k_2})}\|_{2} }=\max_{i\in\Phi_k,i\neq i_{k_1}}\frac{|{b}^{(i)}-A^{(i)}x_{k}|} {\| A^{(i)}\|_{2} }.$$

     \STATE Update $x_{k+1}$ as the line $5$ of Algorithm \ref{algo-3-6}.

	\end{algorithmic}  
\end{algorithm}

To analyze the convergence of TSREKS, we first provide the following lemma, and refer to Appendix \ref{App-4} for its proof.

\begin{lemma}\label{lem-4}
{\rm Let $x_{\star}=A^{\dagger}b$ be the solution of consistent linear systems $Ax=b$. Then the iterative sequence $\{x_{k}\}$ generated by Algorithm \ref{algo-3-12} converges to $x_{\star }$ for any initial vector $x_0$ in expectation. Moreover, the corresponding error norm in expectation yields}
\begin{equation}\label{3-72}
\begin{aligned}
&\mathbb{E}_k\|x_{k+1}-x_{\star}\|_2^2\\
&\quad\leq \left(1-\frac{1-\varepsilon_{k_3}}{1+\tilde{\varepsilon}_{k_3}}\frac{\lambda_{min}(A^*A)}{\tau_{max}}-\frac{1-\varepsilon_{k_3}}{1+\tilde{\varepsilon}_{k_3}}\frac{\lambda_{min}(A^*A)}{\tau_{max}}\frac{\omega_1}{c_1^2(1-\delta^2)}\right)\|x_k-x_{\star}\|_2^2.   
\end{aligned}
\end{equation}
\end{lemma}

From Lemma \ref{lem-4}, we give the following theorem for guaranteeing the convergence of Algorithm TSREKS.

\begin{theorem}\label{The-9}
{\rm Let $x_{\star}=A^{\dagger}b$ be the  least-norm solution of the least-squares
problem \eqref{1-1}. Then the iterative sequence $\{x_{k}\}$ generated by Algorithm \ref{algo-3-11} converges to $x_{\star }$ for any initial vector $x_0$ in expectation. Moreover, the corresponding error norm in expectation yields

\begin{equation}\label{3-73}
\mathbb{E}\|x_{k}-x_{\star}\|_2^2\leq(\max\{\tilde{\alpha}_2,\tilde{\beta}_2\})^{\left \lfloor \frac{k}{2} \right \rfloor }\left(1+\frac{2\lambda_{min}(A^*A)}{t(1-\tilde{\alpha}_2)}\left(\frac{1+\triangle }{1-\triangle }+1\right)\right)\|x_{\star}\|_2^2
\end{equation}
where $$\tilde{\alpha}_2=1-\frac{1-\varepsilon_{k_3}}{1+\tilde{\varepsilon}_{k_3}}\frac{\lambda_{min}(A^*A)}{\tau_{max}}-\frac{1-\varepsilon_{k_3}}{1+\tilde{\varepsilon}_{k_3}}\frac{\lambda_{min}(A^*A)}{\tau_{max}}\frac{\omega_1}{c_1^2(1-\delta^2)}$$
and
$$\tilde{\beta}_2=1-\frac{1-\overline{\varepsilon_{k_3}}}{1+\overline{\tilde{\varepsilon}_{k_3}}}\frac{\lambda_{min}(A^*A)}{\tilde{\tau}_{max}}-\frac{1-\overline{\varepsilon_{k_3}}}{1+\overline{\tilde{\varepsilon}_{k_3}}}\frac{\lambda_{min}(A^*A)}{\tilde{\tau}_{max}}\frac{\tilde{\omega}_1}{\tilde{c}_1^2(1-\tilde{\delta}^2)}$$}
\end{theorem}
\begin{proof}
The proof is similar to Theorem \ref{The-7}, hence we omit it.

\end{proof}

\section{Numerical Experiments}\label{section-4}
\label{sec:experiments}

In this section, we record some numerical results to illustrate the feasibility and validity of the proposed methods compared with some state-of-the-art randomized Kaczmarz extended methods for Eq.\eqref{1-1}. In the first part, we will numerically compare the proposed methods with REK \cite{Zou} and TREK \cite{WWT} when applied to both over- and under-determined random data. In the second part, we present additional numerical results on the X-ray computed tomography problem to further illustrate the effectiveness of the proposed method in comparison with the two solvers mentioned above. Denote by ``IT", the number of iterations, and by ``CPU", the elapsed computing time in seconds. For each method, we report the mean computing time in seconds and the mean number of iterations based on their average values of $5$ repeated tests. Additionally, all experiments have been carried out in MATLAB
2021b on a personal computer with Inter(R) Core(TM) i5-13500HX @2.5GHz and 16.00 GB memory. Unless otherwise stated, we check for convergence every
$ \min(m,n)$ iterations, the stopping
criteria of all tested methods is the residual norm at the current iterate $x_{k}$ satisfying both
$$
\frac{\|b-z_k-Ax_k\|_2}{\|A\|_F\|x_k\|_2}\leq10^{-5} \qquad \mathrm{and}\qquad\frac{\|A^*z_k\|_2}{\|A\|_F^2\|x_k\|_2}\leq10^{-5},$$ and its relative error norm $$\mathrm{RSE}:=  \frac{\| x_{k}-x_{\star} \| _{2}^2}{\|x_{\star}\|_2^2}.$$
Moreover, the initial guess for all tested
algorithm is chosen to be a zero vector. In Algorithm \ref{algo-3-5} and Algorithm \ref{algo-3-11}, we randomly select $\left \lceil l \cdot m \right \rceil$ and $\left \lceil \eta \cdot m \right \rceil$ integers from $m$ integers with equal probability to construct the indicator set $\Gamma_k=randperm(m,round(l \cdot m))$, and $ \Phi_k=randperm(m,round(\eta \cdot m))$. 
All test methods are as follows:\\
$\star$ REK: Randomized extended Kaczmarz method \cite{Zou}.\\
$\star$ TREK: Two-subspace randomized extended Kaczmarz method \cite{WWT}.\\
$\star$ TREKS: Two-dimensional randomized extended Kaczmarz method with simple random sampling.\\
$\star$ GREK: Greedy randomized extended Kaczmarz method.\\
$\star$ TGREK: Two-dimensional greedy randomized extended Kaczmarz method.\\
$\star$ SREK: Semi-randomized extended Kaczmarz method.\\
$\star$ TSREK: Two-dimensional semi-randomized extended Kaczmarz method.\\
$\star$ TSREKS: Two-dimensional semi-randomized extended Kaczmarz method with simple random sampling.\\

\begin{example}\label{Example-1}
\rm Consider the least-squares problem \eqref{1-1} with its coefficient matrix given by a random Gaussian matrix $A=randn(m,n)$, where the symbol $randn(\cdot)$ denotes a MATLAB built-in function, which randomly generates some synthetic data. Moreover, we randomly generate a solution $x_\star$ by the MATLAB function $randn$, with $x_\star=randn(n,1)$, such that the right-hand side vector $b$ is set to be $Ax_\star+r$, where $r\in $ null$(A^*)$ is a nonzero vector generated by the MATLAB function $null$. For TREKS and TSREKS, the parameters $l$ and $\eta$ are chosen to be $0.01$.

\begin{small}
\begin{table}[!htbp]
    \centering
    \caption{Numerical results  of Example \ref{Example-1}  \label{Table-1}}
    \setlength{\tabcolsep}{0.7mm}
    \begin{tabular}{|l|c|c|c|c|c|c|}
    \hline
        \multicolumn{2}{|c|}{$ m \times n$}  & $4000\times1000$ & $6000\times1500$ & $10000\times2500$ & $15000\times4000$ & $20000\times6000$\\ \hline
         & IT   & 35n & 33n & 31n & 32n  & 34n\\ 
         REK & CPU  & 4.0940 & 9.6833& 28.0329& 49.7197 &110.2743\\ 
        ~ & RSE  & 1.87e-7 & 2.88e-7& 4.18e-7 & 8.44e-7 & 1.30e-6\\ \hline
        
         & IT  & 17n & 17n & 16n  & 16n & 17n\\ 
       TREK& CPU & 3.1228&  9.3480& 20.8600 & 39.7708 & 98.9346\\ 
        ~ & RSE  & 1.02e-7 & 1.46e-7& 3.98e-7 & 4.96e-7 &1.46e-6\\ \hline

         & IT  & 17n & 16n & 16n  & 16n & 17n\\ 
        TREKS & CPU &  2.4251&  7.2530& 17.8421 & 33.5764 &77.4778\\
        ~ & RSE  & 1.35e-7 & 3.38e-7& 3.56e-7 & 6.78e-7 & \textbf{1.09e-6}\\ \hline

         & IT  & 8n & 7n & 7n & 7n & 8n\\ 
        GREK& CPU  & 1.4187 &  3.3920& 8.8301 & 19.7287 & 45.8266\\ 
        ~ & RSE  & 7.69e-6 & 5.82e-6& 4.37e-6 & 3.51e-6 &3.02e-6\\ \hline
        
        & IT  & 4n &4n & 4n&  4n &4n \\ 
       TGREK  & CPU  & 1.1552 & 2.3649 & 7.3129& 15.2386 & 32.2548\\ 
        ~ & RSE & 1.12e-7 & 7.82e-8& 1.31e-7 & 2.27e-7 &1.28e-6 \\ \hline

        & IT & 7n & 7n & 7n & 7n & 8n \\ 
        SREK & CPU  & 0.7706 &  2.3429& 6.4247 & 10.0677 & 23.9957\\
        ~ & RSE  & 1.13e-5 & 7.63e-6& 4.99e-6 & 4.44e-6 &2.76e-6 \\ \hline
        
         & IT  & 4n &4n & 4n&  4n &4n \\ 
        TSREK & CPU  & \textbf{0.6675} &  1.3428 & 4.5088 & 7.3736 & 15.5008\\ 
        ~ & RSE  & \textbf{9.70e-8} & \textbf{5.48e-8}& \textbf{9.81e-8}& \textbf{1.93e-7} & 1.17e-6\\ \hline

         & IT  & 5n &4n & 4n&  4n &4n \\ 
        TSREKS & CPU & 0.7421 & \textbf{1.1901} & \textbf{3.8598} & \textbf{6.4278} & \textbf{12.7261}\\ 
        ~ & RSE  & 1.14e-7 & 2.28e-7& 1.64e-7 & 2.91e-7 &1.36e-6\\ \hline

    \end{tabular}
\end{table}
\end{small}

\begin{small}
\begin{table}[!htbp]
    \centering
    \caption{Numerical results  of Example \ref{Example-1} with  $l,\eta=0.01$  \label{Table-2}}
    \setlength{\tabcolsep}{0.7mm}
    \begin{tabular}{|l|c|c|c|c|c|c|}
    \hline
        \multicolumn{2}{|c|}{$ m \times n$} & $1000\times3000$ & $1500\times5000$ & $2500\times8000$ & $3000\times10000$ & $5000\times20000$\\ \hline
         & IT & 44m & 38m & 36m  & 36m & 39m\\ 
        REK & CPU & 3.9053 & 11.8842 & 24.5849 & 35.7705 & 97.4236\\ 
        ~ & RSE & 2.95e-7 & 3.07e-7 & 7.06e-7 & 6.56e-7 & 1.35e-6\\ \hline
        
         & IT & 22m & 19m & 19m  & 18m & 20m\\ 
        TREK & CPU & 3.3454 & 10.5116 & 22.0182 & 30.6844 & 95.9596\\ 
        ~ & RSE & 2.20e-7 & 3.16e-7 & 4.25e-7 & 7.00e-7 & 9.37e-7\\ \hline

         & IT & 22m & 19m & 19m  & 18m & 20m\\ 
        TREKS & CPU & 3.1204 & 9.4046 & 20.9360 & 27.9019 & 87.4578\\ 
        ~ & RSE & 2.05e-7 & 2.44e-7 & 5.88e-7 & 6.88e-7 & 9.43e-7\\ \hline

         & IT & 11m & 9m & 9m & 9m & 10m\\ 
        GREK & CPU & 1.2461 & 3.7211 & 8.6506 & 13.3862 & 40.6111\\ 
        ~ & RSE & 9.47e-7 & 1.88e-6 & 1.10e-6 & 7.21e-7 & 7.80e-7\\ \hline

         & IT & 6m & 5m & 5m & 5m & 5m \\ 
        TGREK & CPU & 1.0513 & 3.0891 & 7.4636 & 10.4755 & 27.7316\\ 
        ~ & RSE & 6.01e-8 & \textbf{6.88e-8} & 1.36e-7 & 1.18e-7 & 5.92e-7\\ \hline

         & IT & 11m & 9m & 9m & 9m & 10m\\ 
        SREK & CPU & 0.9740 & 2.8513 & 6.1995 & 6.4247 & 21.6812\\ 
        ~ & RSE & 7.36e-7 & 1.34e-6 & 1.51e-6 & 4.99e-6 & 8.71e-7\\ \hline

         & IT & 5m & 5m & 5m & 5m & 5m\\ 
        TSREK & CPU & \textbf{0.5155} & 1.2682 & 3.0330 & 4.1234 & 9.2025\\ 
        ~ & RSE & 3.13e-7 & 8.65e-8 & \textbf{1.20e-7} & \textbf{6.33e-8} & \textbf{5.57e-7}\\ \hline

         & IT & 7m & 5m & 5m & 5m & 5m\\ 
        TSREKS & CPU & 0.6565 & \textbf{0.9940} & \textbf{2.6230} & \textbf{3.5606} & \textbf{7.7544}\\ 
        ~ & RSE & \textbf{4.20e-8} & 2.24e-7 & 1.58e-7 & 1.07e-7 & 5.79e-7\\ \hline
    \end{tabular}
\end{table}
\end{small}
Starting with the zero vector, we implement all methods to solve the least-squares problem \eqref{1-1}, and then record the obtained numerical results in Tables \ref{Table-1} and \ref{Table-2}. A rough conclusion drawn from these tables is that, under the same strategy, the iterations and computing time corresponding to two-dimensional extended Kaczmarz-type methods are commonly less than that of traditional one-dimensional methods, which perform more competitively. The reason is that, in per iteration, two working rows and columns possess more valid information compared with a single row and column. Specifically, one can see that TREKS performs better than REK and TREK, while GREK and TGREK require much fewer iterations and CPU time than TREKS in all cases. Indeed, the convergence rate of TGREK is about twice that of 
 TREKS, and it achieves higher accuracy compared with GREK, even though the speed-up of the former compared with GREK is at most $1.46$ for the case $5000\times 20000$. Moreover, we find that the SREK, TSREK, and TSREKS solvers outperform the GREK and TGREK solvers in terms of CPU time. The speed-up of TSREK compared with TGREK is at least $2$, with SREK at least $1.15$ for the case $4000\times1000$, but the former is superior to SREK in accuracy. It also appears that the TSREK and TSREKS are the best two solvers among all other methods in terms of CPU time. The reason is that, at each iteration, both of them only select the two active rows and columns corresponding to the current largest and the second largest homogeneous residuals, and the latter requires a small part of the rows of the coefficient matrix, which saves computational operations and storage. 

\end{example}

\begin{example}\label{Example-2}
\rm Consider solving least-squares
problems\eqref{1-1} with its coefficient matrix coming from the University of Florida Sparse Matrix Collection \cite{Davis}. The chosen matrices are mainly used for statistical mathematical, combinatorial, and linear programming problems. The right-hand side vectors are generated similarly in Example \ref{Example-1}. Due to the prohibitively large size of the matrix, the vector $r$ cannot be directly computed using the MATLAB function $null$. Instead, it is generated through a projection process, that is, $r=\tilde{r}-AA^{\dagger}\tilde{r}$, where the vector $\tilde{r}$ is generated randomly with the MATLAB function randn and $A^{\dagger}\tilde{r}$ is generated by the MATLAB function $lsqminnorm$. For TREKS and TSREKS, the parameters $l$ and $\eta$ are chosen to be $0.01$.

\begin{small}
\begin{table}[!htbp]
    \centering
    \caption{Numerical results  of Example \ref{Example-2}  \label{Table-3}}
    \setlength{\tabcolsep}{0.3mm}
    \begin{tabular}{|l|c|c|c|c|c|c|}
    \hline
        \multicolumn{2}{|c|}{Matrix} & Franz10 & relat7 & GL7d13 & lp\_nug20 & testbig \\ \hline
        \multicolumn{2}{|c|}{$ m \times n$} & $19588\times4164$ & $21924\times1045$ & $47271\times8899$ & $15240\times72600$ & $17613\times31223$\\ \hline
        \multicolumn{2}{|c|}{Density}  & 0.12\% & 0.36\% & 0.09\% & 0.03\% & 0.01\% \\ \hline
        \multicolumn{2}{|c|}{Cond$\left ( A \right ) $ } & 1.27e+16 & Inf & Inf & 5.37e+16 & 6.69e+2 \\ \hline

        REK & IT & 24n & 168n  & 76n  & 16m & 163m\\ 
        ~ & CPU & 24.2922 & 31.3828 & 853.6224 & 341.6211 & 2102.1692\\ 
        ~ & RSE & 4.95e-7 & 2.38e-6 & 6.44e-5 & 1.12e-6 & 8.35e-5\\ \hline

        TREK & IT & 12n & 73n  & 32n & 8m & 82m\\ 
        ~ & CPU & 23.1345 & 26.9842 & 603.9348 & 371.6056 & 2193.4794\\ 
        ~ & RSE & 2.95e-7 & 2.07e-6 & 1.09e-5 & 1.38e-6 & 6.07e-5\\ \hline

        TREKS & IT & 12n & 84n  & 26n & 8m & 81m\\ 
        ~ & CPU & 17.5021 & 18.8409 & 550.0942 & 335.9771 & 1699.3326\\ 
        ~ & RSE & 4.54e-7 & 2.29e-6 & 2.18e-5 & 1.31e-6 & 5.88e-5\\ \hline

        GREK & IT & 3n & 24n & 5n & 3m & 18m\\ 
        ~ & CPU & 6.3689 & 9.8453 & 94.3604 & 96.5894 & 349.7251\\ 
        ~ & RSE & 8.42e-6 & 3.19e-6 & 1.47e-6 & 2.39e-7 & 1.04e-5\\ \hline

        TGREK & IT & 2n & 14n & 3n & 2m & 10m\\ 
        ~ & CPU & 4.2869 & 6.8950 & 53.5814 & 65.1760 & 216.0054\\ 
        ~ & RSE & 5.44e-9 & 2.55e-6 & 1.93e-7 & 7.22e-10 & 1.53e-5\\ \hline

        SREK & IT & 3n & 28n & 5n & 3m & 18m\\ 
        ~ & CPU & 2.8134 & 5.6693 & 47.2849 & 61.5645 & 186.6214\\ 
        ~ & RSE & 3.33e-7 & 4.09e-6 & 1.84e-7 & 3.24e-6 & 1.14e-5\\ \hline

        TSREK & IT & 2n & 15n & 3n & 2m & 10m\\ 
        ~ & CPU & 1.7946 & 3.3496 & 21.7502 & 37.7981 & 83.0902\\ 
        ~ & RSE & \textbf{2.43e-9} & \textbf{2.05e-6} & \textbf{1.40e-7} & \textbf{5.94e-10} & 1.04e-5\\ \hline

        TSREKS & IT & 2n & 12n & 3n & 2m & 10m\\ 
        ~ & CPU & \textbf{1.5192} & \textbf{2.0421} & \textbf{19.8472} & \textbf{30.1789} & \textbf{71.0731}\\ 
        ~ & RSE & 3.50e-9 & 2.79e-6 & 1.47e-7 & 7.15e-10 & \textbf{5.60e-6}\\ \hline
    \end{tabular}
\end{table}
\end{small}

We run the tested methods on some practical problems, and then list the numerical results in Table \ref{Table-3}. As shown in this table, a rough conclusion is that the TREKS, GREK, SREK, TGREK, TSREK, and TSRKES solvers have comparable performance to the REK and TREK solvers. Particularly, the speed-up of TGREK compared with GREK is at least $1.43$ in the case of relat7, and at most $1.76$ in the case of GL7d13, and TSREK compared with SREK is at least $1.57$ in the case of Franz10, and at most $2.25$ in the case of testbig. Moreover, one can also see that 
TSREK and TSREKS solvers are superior to the other solvers in terms of the number of iterations and CPU time. Indeed, TSREKS is about $2$-$10$ times faster than others. This reflects that the method with the simple random sampling strategy is more favorable 
for big data problems. 
\end{example}

\begin{example}\label{Example-3}
\rm Consider solving the X-ray computed tomography problem and the seismic tomography in geophysics.
The Matlab package AIR Tools II \cite{Hansen} is used to simulate the above problems. First, the performance of the proposed algorithm is tested in Curved Fan-Beam Tomography. We generated a test image of size 64 by 64 based on the $fancurvedtomo$ function in the AIR Tools II package and recorded the measurements of the angles 0:1.5:179.5. The number of sensors per projection line is 100 and the distance $d$ from the source to the center of rotation is set to 80. The Euclidean condition number of $A$ is 573.3796. The right-hand side vector is
$b=Ax_{\star}+r$, where the unique solution of the least-squares problem \eqref{1-1} is obtained by reshaping the $64\times64$ Shepp-Logan medical phantom and $r$ is a nonzero vector in the null space generated by the MATLAB function $null$. For TREKS and TSREKS, the parameters $l$ and $\eta$ are chosen to be $0.01$.

In Table \ref{Table-4}, we list the number of iteration steps, the computing times, and the relative solution errors for all tested methods. It appears from this table that the TREKS and GREK work better than REK and TREK, while SREK, TGREK, TSREK, and TSREKS are superior to the other methods in terms of computing time. More precisely, with the simple random sampling, the speed-up of TREKS compared with TREK is about $2$, and TSREKS compared with TSREK is about $1.5$, which is more favorable for big data problems. Moreover, the TGREK and TSREK are about two times faster than GREK and SREK, respectively. This illustrates that the two-dimensional extended Kaczamrz-type method works better than the one-dimensional ones. The restored Shepp-Logan medical phantom of size $64\times64$ given by the tested methods is displayed in Fig.\ref{Fig:1}.

\begin{table}[!htbp]
    \centering
    \caption{Numerical results  of Example \ref{Example-3} for $10080\times4096$ matrix $A$   \label{Table-4}}
    \begin{tabular}{|l|c|c|c|c|}
    \hline
        \multicolumn{2}{|c|}{Algorithm} & IT & CPU &RSE \\ \hline

        \multicolumn{2}{|c|}{REK} &  3287n  & 9354.5201 & 1.20e-3\\ 
        
        \multicolumn{2}{|c|}{TREK}  &  1634n & 9137.1344 & 1.34e-3 \\

        \multicolumn{2}{|c|}{TREKS}  &  1011n & 5597.3971 & 6.93e-4 \\

        \multicolumn{2}{|c|}{GREK}  &  372n & 1275.4051 & 8.46e-4\\

        \multicolumn{2}{|c|}{TGREK} & 215n & 707.5904 &9.90e-4  \\

        \multicolumn{2}{|c|}{SREK}  &  430n & 811.4277 & 2.15e-3\\

        \multicolumn{2}{|c|}{TSREK}  & 160n & 352.9157 &4.54e-4  \\

        \multicolumn{2}{|c|}{TSREKS}  & 119n & \textbf{247.2895} &\textbf{2.26e-4}\\ \hline

    \end{tabular}
\end{table}

 \begin{figure}[htbp!]
\begin{center}
	\begin{minipage}[c]{1.0\textwidth}
		\includegraphics[width=5in]{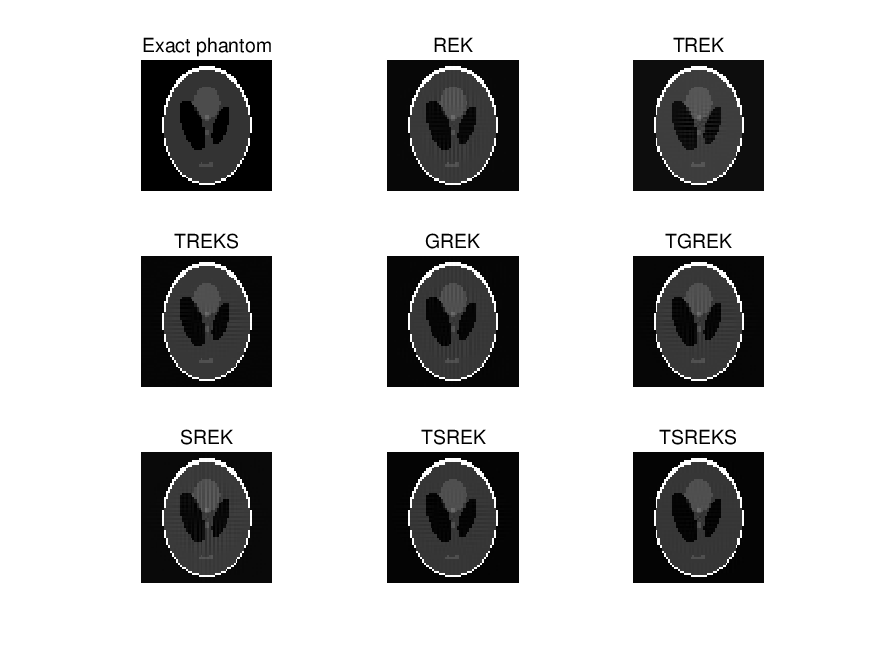}
	\end{minipage}
\end{center}
	\vspace{-3em}\caption{Pictures of the Shepp-Logan medical phantom \label{Fig:1}}
\end{figure}

To further illustrate the efficiency of the proposed methods, we apply them to solve the seismic wave tomography problem. The goal is to use the observed data $b$ to invert the properties in the square domain of size $[0,20]\times[0,20]$ subsurface medium. There are $18$ equally spaced sources and $560$ equally spaced receivers on the domain boundary. Here, a testing image of size $20\times 20$ based on the $seismictomo$ function in the AIR Tools II package is given, and the Euclidean condition number of $A$ is $448.46$. The right-hand side vector is $b=Ax_{\star}+r$, where the unique solution of the least-squares problem \eqref{1-1} is obtained by reshaping the $20\times20$ tectonic phantom and $r$ is a nonzero vector in the null space generated by the MATLAB function $null$. For TREKS and TSREKS, the parameters $l$ and $\eta$ are chosen to be $0.1$.

In Table \ref{Table-5}, we list the number of iteration steps, the computing times, and the relative solution errors for all tested methods. As was shown, TREKS shows better performance than REK and TREK in terms of computing time, while GREK, SREK, TGREK, TSREK, and TSREKS converge faster than the former three. In particular, with the simple random sampling, the speed-up of TREKS compared with TREK is about $1.3$, and TSREKS compared with TSREK is about $1.4$, which is more favorable for big data problems. Indeed, the TGREK and TSREK are about $1.2$ times faster and achieve higher accuracy than GREK and SREK, respectively. This illustrates that the two-dimensional extended Kaczamrz-type methods work more actively than the one-dimensional ones. The restored tectonic phantoms of size $20\times20$ given by the tested methods are displayed in Fig.\ref{Fig:2}.

\begin{table}[!htbp]
    \centering
    \caption{Numerical results of Example \ref{Example-3} for $10080\times400$ matrix $A$   \label{Table-5}}
    \begin{tabular}{|l|c|c|c|c|}
    \hline
        \multicolumn{2}{|c|}{Algorithm} & IT & CPU &RSE \\ \hline

        \multicolumn{2}{|c|}{REK} &  30885n  & 1701.8708 & 2.39e-4\\ 
        
        \multicolumn{2}{|c|}{TREK}  &  14001n & 1388.6912 & 2.61e-4 \\

        \multicolumn{2}{|c|}{TREKS}  &  13366n & 1051.6015 & 2.64e-4 \\

        \multicolumn{2}{|c|}{GREK}  &  3460n & 347.0520 & 3.32e-4\\

        \multicolumn{2}{|c|}{TGREK} & 1513n & 200.9663 &\textbf{2.38e-4}  \\

        \multicolumn{2}{|c|}{SREK}  &  2993n & 159.4135 & 4.56e-4\\

        \multicolumn{2}{|c|}{TSREK}  & 1999n & 130.6461 &2.66e-4  \\

        \multicolumn{2}{|c|}{TSREKS}  & 1234n & \textbf{91.7292} &3.45e-4\\ \hline

    \end{tabular}
\end{table}

 \begin{figure}[htbp!]
\begin{center}
	\begin{minipage}[c]{1.0\textwidth}
		\includegraphics[width=5in]{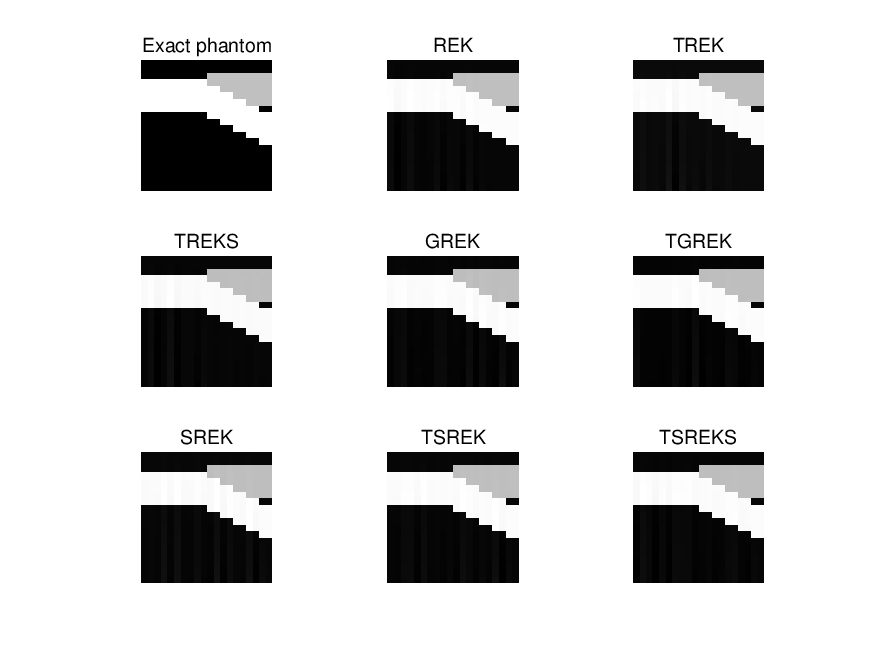}
	\end{minipage}
\end{center}
	\vspace{-3em}\caption{Pictures of the tectonic phantom \label{Fig:2}}
\end{figure}

\end{example}
\section{Conclusions}
\label{section-5}
In this paper, we proposed several fast two-dimensional extended Kaczmarz-type methods to solve the least-squares problems \eqref{1-1}. The main contributions are listed as follows.\\
\begin{itemize}
    \item  A greedy randomized extended Kaczmarz method, which incorporates the greedy randomized Kaczmarz iteration with a greedy randomized orthogonal projection iteration, is proposed. To improve its convergence rate, we derive the SREK method, by selecting the row and column corresponding to the current largest homogeneous residuals at each iteration.
    \item From a purely algebraic point of view, an alternative to TREK is concisely proposed, which performs more straightforwardly and easier than the method proposed in \cite{WWT}. We then develop a two-dimensional randomized extended Kaczmarz method with simple random sampling (TREKS), which only utilizes a small portion of rows and columns of the coefficient matrix, for solving Eq.\eqref{1-1}.
    \item A novel two-dimensional greedy randomized extended Kaczmarz method, by grasping two larger entries of the residual vector at each iteration, is proposed. The convergence analysis of which is also newly established. By selecting the rows and columns corresponding to the current largest and the second largest homogeneous residuals, a two-dimensional semi-randomized extended Kaczmarz method is derived. The method does not need to construct index sets with working rows and columns, nor compute probabilities. To further save the arithmetic operations of TSREK, we present a modified TSREK method with simple random sampling that only requires computing some elements of the residual vectors corresponding to the simple sampling set. The proposed method is comparative for big data problems.

    \item  The obtained numerical results, including when the proposed solvers for some randomly generated data and the X-ray computed tomography
    problem, illustrate the feasibility and validity of the proposed solvers compared to many state-of-the-art randomized extended Kaczmarz methods.
\end{itemize}

\bigskip
\noindent\textbf{Data Availability Statement:} {The data used to support the findings of this study are included within this article.}

\noindent\textbf{Conflict of Interest:} {The authors declare no conflict of interest.}

\appendix
	\section{
\textbf{Proof of Lemma \ref{lem-1}}}\label{App-1}

From the TRKS method, it can be stated that
$$x_{k+1}=x_k+\frac{b^{(i_{k_2})}-A^{(i_{k_2})}x_k}{\|A^{(i_{k_2})}\|_2^2}(A^{(i_{k_2})})^*+\left(\frac{b^{(i_{k_1})}}{\|A^{(i_{k_1})}\|_2}-\bar{\mu} _k\frac{b^{(i_{k_2})}}{\|A^{(i_{k_2})}\|_2}-u_k^*x_k\right)\frac{u_k}{\|u_k\|_2^2},$$
where 
\begin{equation}\label{5-1}
\mu_k=\frac{(A^{(i_{k_2})})(A^{(i_{k_1})})^*}{\|A^{(i_{k_2})}\|_2\|A^{(i_{k_1})}\|_2} 
\end{equation}
and 
\begin{equation}\label{5-2}
u_k=\frac{(A^{(i_{k_1})})^*}{\|A^{(i_{k_1})}\|_2}-\mu_k\frac{(A^{(i_{k_2})})^*}{\|A^{(i_{k_2})}\|_2}
\end{equation} with $\|u_k\|_2^2=1-|\mu_k|^2$.
Then, it follows that
\begin{equation}\label{5-3}
\begin{aligned}
x_{k+1}-x_{\star}&=x_k-x_{\star}+\frac{b^{(i_{k_2})}-A^{(i_{k_2})}x_k}{\|A^{(i_{k_2})}\|_2^2}(A^{(i_{k_2})})^*\\
&\quad+\left(\frac{b^{(i_{k_1})}}{\|A^{(i_{k_1})}\|_2}-\bar{\mu} _k\frac{b^{(i_{k_2})}}{\|A^{(i_{k_2})}\|_2}-u_k^*x_k\right)\frac{u_k}{\|u_k\|_2^2}\\
&=\left(I_n-\frac{(A^{(i_{k_2})})^*A^{(i_{k_2})}}{\|A^{(i_{k_2})}\|_2}\right)(x_k-x_{\star})\\
&\quad+\left(\frac{A^{(i_{k_1})}x_{\star}}{\|A^{(i_{k_1})}\|_2}-\bar{\mu} _k\frac{A^{(i_{k_2})}x_{\star}}{\|A^{(i_{k_2})}\|_2}-u_k^*x_k\right)\frac{u_k}{\|u_k\|_2^2}\\
&=\left(I_n-\frac{(A^{(i_{k_2})})^*A^{(i_{k_2})}}{\|A^{(i_{k_2})}\|_2}\right)(x_k-x_{\star})-u_k^*(x_k-x_{\star})\frac{u_k}{\|u_k\|_2^2}\\
&=\left(I_n-\frac{(A^{(i_{k_2})})^*A^{(i_{k_2})}}{\|A^{(i_{k_2})}\|_2}-\frac{u_ku_k^*}{\|u_k\|_2^2}\right)(x_k-x_{\star}).
\end{aligned}
\end{equation}
Denote $\tilde{y}_k$ and $\tilde{x}_{k+1}$ as the vectors obtained by utilizing one RK method on $x_k$ and $\tilde{y}_k$ with the working rows $A^{(i_{k_1})}$ and $A^{(i_{k_2})}$ for the consistent linear system $Ax=b$, respectively. That is,
$$
\tilde{y}_k=x_k+\frac{{b}^{(i_{k_1})}-A^{(i_{k_1})}x_{k}} {\| A^{(i_{k_1})}\|_{2} ^{2} }(A^{(i_{k_1})})^{*}
$$
and
$$
\tilde{x}_{k+1}=\tilde{y}_k+\frac{{b}^{(i_{k_2})}-A^{(i_{k_2})}\tilde{y}_k} {\| A^{(i_{k_2})}\|_{2} ^{2} }(A^{(i_{k_2})})^{*}.
$$
It follows from $u_k$ that
$$
\begin{aligned}
\tilde{x}_{k+1}-x_{\star}&=\tilde{y}_k-x_{\star}+\frac{{b}^{(i_{k_2})}-A^{(i_{k_2})}\tilde{y}_k} {\| A^{(i_{k_2})}\|_{2} ^{2} }(A^{(i_{k_2})})^{*}\\
&=\left(I_n-\frac{(A^{(i_{k_2})})^{*}A^{(i_{k_2})}}{\| A^{(i_{k_2})}\|_{2} ^{2}}\right)(\tilde{y}_k-x_{\star})\\
&=\left(I_n-\frac{(A^{(i_{k_2})})^{*}A^{(i_{k_2})}}{\| A^{(i_{k_2})}\|_{2} ^{2}}\right)\left(I_n-\frac{(A^{(i_{k_1})})^{*}A^{(i_{k_1})}}{\| A^{(i_{k_1})}\|_{2} ^{2}}\right)(x_k-x_{\star})\\
&=\left(I_n-\frac{(A^{(i_{k_2})})^{*}A^{(i_{k_2})}}{\| A^{(i_{k_2})}\|_{2} ^{2}}-\left(I_n-\frac{(A^{(i_{k_2})})^{*}A^{(i_{k_2})}}{\| A^{(i_{k_2})}\|_{2} ^{2}}\right)\frac{(A^{(i_{k_1})})^{*}A^{(i_{k_1})}}{\| A^{(i_{k_1})}\|_{2} ^{2}}\right)\\
&\quad(x_k-x_{\star})\\
&=\left(I_n-\frac{(A^{(i_{k_2})})^{*}A^{(i_{k_2})}}{\| A^{(i_{k_2})}\|_{2} ^{2}}-\frac{u_kA^{(i_{k_1})}}{\| A^{(i_{k_1})}\|_{2}} \right)(x_k-x_{\star})\\
&=\left(I_n-\frac{(A^{(i_{k_2})})^{*}A^{(i_{k_2})}}{\| A^{(i_{k_2})}\|_{2} ^{2}}-\frac{u_ku_k^*}{\| u_k\|_{2}^2} \right)(x_k-x_{\star})\\
&\quad+\left(\frac{u_ku_k^*}{\| u_k\|_{2}^2}-\frac{u_kA^{(i_{k_1})}}{\| A^{(i_{k_1})}\|_{2}} \right)(x_k-x_{\star}).
\end{aligned}
$$
It is easy to see that $\left(I_n-\frac{(A^{(i_{k_2})})^{*}A^{(i_{k_2})}}{\| A^{(i_{k_2})}\|_{2} ^{2}}-\frac{u_ku_k^*}{\| u_k\|_{2}^2} \right)^*\left(\frac{u_ku_k^*}{\| u_k\|_{2}^2}-\frac{u_kA^{(i_{k_1})}}{\| A^{(i_{k_1})}\|_{2}} \right)=0.$
Furthermore, through the orthogonality, we have
$$
\begin{aligned}
\|\tilde{x}_{k+1}-x_{\star}\|_2^2&=\left\|\left(I_n-\frac{(A^{(i_{k_2})})^{*}A^{(i_{k_2})}}{\| A^{(i_{k_2})}\|_{2} ^{2}}-\frac{u_ku_k^*}{\| u_k\|_{2}^2} \right)(x_k-x_{\star})\right\|_2^2\\
&\quad+\left\|\left(\frac{u_ku_k^*}{\| u_k\|_{2}^2}-\frac{u_kA^{(i_{k_1})}}{\| A^{(i_{k_1})}\|_{2}} \right)(x_k-x_{\star})\right\|_2^2
\end{aligned}
$$
From Eq.\eqref{5-3},
$$
\|x_{k+1}-x_{\star}\|_2^2=\|\tilde{x}_{k+1}-x_{\star}\|_2^2-\left|\frac{|\mu_k|^2}{\|u_k\|_2}\frac{A^{(i_{k_1})}(x_k-x_{\star})}{\| A^{(i_{k_1})}\|_{2}}-\frac{\bar{\mu}_k}{\|u_k\|_2}\frac{A^{(i_{k_2})}(x_k-x_{\star})}{\| A^{(i_{k_2})}\|_{2}}\right|^2.
$$
By taking the conditional expectation on both sides of the equation, we have
\begin{equation}\label{5-4}
\begin{aligned}
&\quad\mathbb{E}_k\|x_{k+1}-x_{\star}\|_2^2\\
&=\mathbb{E}_k\|\tilde{x}_{k+1}-x_{\star}\|_2^2-\mathbb{E}_k\left|\frac{|\mu_k|^2}{\|u_k\|_2}\frac{A^{(i_{k_1})}(x_k-x_{\star})}{\| A^{(i_{k_1})}\|_{2}}-\frac{\bar{\mu}_k}{\|u_k\|_2}\frac{A^{(i_{k_2})}(x_k-x_{\star})}{\| A^{(i_{k_2})}\|_{2}}\right|^2.
\end{aligned}
\end{equation}
For the first term of the right-hand side, define \\

$$Q_{i_{k_1}}=I_n-\frac{(A^{(i_{k_1})})^{*}A^{(i_{k_1})}}{\| A^{(i_{k_1})}\|_{2} ^{2}}\quad \mathrm{and}  \quad Q_{i_{k_2}}=I_n-\frac{(A^{(i_{k_2})})^{*}A^{(i_{k_2})}}{\| A^{(i_{k_2})}\|_{2} ^{2}},$$
the expected probability
\begin{equation}\label{add-1}
\tilde{P}_{k_1}=\sum_{\substack{i_{k_1}\in\Gamma_k }}\frac{\|A^{(i_{k_1})}\|_2^2}{\sum_{\substack{i\in\Gamma_k }}\|A^{(i)}\|_2^2}\quad\mathrm{and}\quad\tilde{P}_{k_2}=\sum_{\substack{i_{k_2}\in\Gamma_k\\i_{k_2}\neq i_{k_1}}}\frac{\|A^{(i_{k_2})}\|_2^2}{\sum_{\substack{j\in\Gamma_k\\j\neq i_{k_1}}}\|A^{(j)}\|_2^2}.
\end{equation}
From Chebyshev’s (weak) law of large numbers, if 
$lm$ is sufficiently large and there are six scalars $0<\varepsilon_k,\varepsilon_{k_1},\varepsilon_{k_2},\tilde{\varepsilon}_k,\tilde{\varepsilon}_{k_1},\tilde{\varepsilon}_{k_2}\ll1$, then it follows that
\begin{equation}\label{5-5}
\frac{\sum_{\substack{i_{k_2}\in\Gamma_k\\i_{k_2}\neq i_{k_1}}}|A^{(i_{k_2})}Q_{i_{k_1}}(x_k-x_{\star})|^2}{lm-1}=\frac{\sum\limits_{\substack{i=1\\i\neq i_{k_1}}}^{m}|A^{(i)}Q_{i_{k_1}}(x_k-x_{\star})|^2}{m-1}(1\pm\varepsilon_k ),
\end{equation}
\begin{equation}\label{5-6}
\frac{\sum_{\substack{j\in\Gamma_k\\j\neq i_{k_1}}}\|A^{(j)}\|_2^2}{lm-1}  =\frac{\sum\limits_{\substack{j=1\\j\neq i_{k_1}}}^{m}\|A^{(j)}\|_2^2}{m-1}(1\pm\tilde{\varepsilon}_k ),
\end{equation}

\begin{equation}\label{add-2}
\frac{\sum_{\substack{i_{k_1}\in\Gamma_k}}|Q_{i_{k_1}}(x_k-x_{\star})|^2}{lm}=\frac{\sum\limits_{\substack{i_{k_1}=1}}^{m}|Q_{i_{k_1}}(x_k-x_{\star})|^2}{m}(1\pm\varepsilon_{k_1} ),
\end{equation}
\begin{equation}\label{add-3}
\frac{\sum_{\substack{i\in\Gamma_k}}\|A^{(i)}\|_2^2}{lm}  =\frac{\sum\limits_{\substack{i=1}}^{m}\|A^{(i)}\|_2^2}{m}(1\pm\tilde{\varepsilon}_{k_1} ),
\end{equation}

\begin{equation}\label{5-7}
\resizebox{1\hsize}{!}{$
\begin{aligned}
&\quad\frac{\sum_{\substack{i_{k_1}\in\Gamma_k }}\|A^{(i_{k_1})}\|_2^2\sum_{\substack{i_{k_2}\in\Gamma_k\\i_{k_2}\neq i_{k_1}}}\|A^{(i_{k_2})}\|_2^2\left|\frac{|\mu_k|^2}{\|u_k\|_2}\frac{A^{(i_{k_1})}(x_k-x_{\star})}{\| A^{(i_{k_1})}\|_{2}}-\frac{\bar{\mu}_k}{\|u_k\|_2}\frac{A^{(i_{k_2})}(x_k-x_{\star})}{\| A^{(i_{k_2})}\|_{2}}\right|^2}{(lm)^2-lm}\\
&=\frac{\sum\limits_{\substack{i_{k_1}\in[m] }}\|A^{(i_{k_1})}\|_2^2\sum\limits_{\substack{i_{k_2}\in[m]\\i_{k_2}\neq i_{k_1}}}\|A^{(i_{k_2})}\|_2^2\left|\frac{|\mu_k|^2}{\|u_k\|_2}\frac{A^{(i_{k_1})}(x_k-x_{\star})}{\| A^{(i_{k_1})}\|_{2}}-\frac{\bar{\mu}_k}{\|u_k\|_2}\frac{A^{(i_{k_2})}(x_k-x_{\star})}{\| A^{(i_{k_2})}\|_{2}}\right|^2}{(m)^2-m}(1\pm{\varepsilon }_{k_2}),
\end{aligned}
$}
\end{equation}
and
\begin{equation}\label{5-8}
\resizebox{1\hsize}{!}{$
\begin{aligned}
&\quad\frac{\sum_{\substack{i_{k_1}\in\Gamma_k }}\|A^{(i_{k_1})}\|_2^2\sum_{\substack{i_{k_2}\in\Gamma_k\\i_{k_2}\neq i_{k_1}}}\|A^{(i_{k_2})}\|_2^2}{(lm)^2-lm}=\frac{\sum_{\substack{i_{k_1}\in[m] }}\|A^{(i_{k_1})}\|_2^2\sum\limits_{\substack{i_{k_2}=1\\i_{k_2}\neq i_{k_1}}}^{m}\|A^{(i_{k_2})}\|_2^2}{m^2-m}(1\pm\tilde{\varepsilon}_{k_2} )\\
&=\frac{\|A\|_F^2(\|A\|_F^2-\|A^{(i_{k_1})}\|_2^2)}{m^2-m}(1\pm\tilde{\varepsilon}_{k_2} ).  
\end{aligned}$}
\end{equation}
Moreover, if $x_k-x_{\star}\in$ Row$(A)=\mathcal{R}(A^*)$ and $Q_{i_{k_1}}(x_k-x_{\star})\in \mathcal{R}(A^*)$, 
then 
$$
\begin{aligned}
&\quad\mathbb{E}_k\|\tilde{x}_{k+1}-x_{\star}\|_2^2\\&=\sum_{\substack{i_{k_1}=1}}^{m}\frac{\|A^{(i_{k_1})}\|_2^2}{\|A\|_F^2}\sum_{\substack{i_{k_2}=1\\i_{k_2}\neq i_{k_1}}}^{m}\frac{\|A^{(i_{k_2})}\|_2^2}{\|A\|_F^2-\|A^{(i_{k_1})}\|_2^2}(x_k-x_{\star})^*Q_{i_{k_1}}Q_{i_{k_2}}Q_{i_{k_1}}(x_k-x_{\star})\\
&=\sum_{\substack{i_{k_1}\in\Gamma_k }}\frac{\|A^{(i_{k_1})}\|_2^2}{\|A\|_F^2}\sum_{\substack{i_{k_2}\in\Gamma_k\\i_{k_2}\neq i_{k_1}}}\frac{\|A^{(i_{k_2})}\|_2^2}{\|A\|_F^2-\|A^{(i_{k_1})}\|_2^2}\\
\end{aligned}
$$
$$
\begin{aligned}
&\quad\left(\|Q_{i_{k_1}}(x_k-x_{\star})\|_2^2-\frac{|A^{(i_{k_2})}Q_{i_{k_1}}(x_k-x_{\star})|^2}{\|A^{(i_{k_2})}\|_2^2}\right) \\
&=\tilde{P}_{k_1}\left(\|Q_{i_{k_1}}(x_k-x_{\star})\|_2^2-\frac{\sum_{\substack{i_{k_2}\in\Gamma_k\\i_{k_2}\neq i_{k_1}}}|A^{(i_{k_2})}Q_{i_{k_1}}(x_k-x_{\star})|^2}{\sum_{\substack{j\in\Gamma_k\\j\neq i_{k_1}}}\|A^{(j)}\|_2^2}\right)\\
&=\tilde{P}_{k_1}\left(\|Q_{i_{k_1}}(x_k-x_{\star})\|_2^2-\left(\frac{1\pm\varepsilon_k}{1\pm\tilde{\varepsilon}_k }\right)\frac{\sum\limits_{\substack{i=1\\i\neq i_{k_1}}}^{m}|A^{(i)}Q_{i_{k_1}}(x_k-x_{\star})|^2}{\sum\limits_{\substack{j=1\\j\neq i_{k_1}}}^{m}\|A^{(j)}\|_2^2}\right)\\
&\leq\tilde{P}_{k_1}\left(\|Q_{i_{k_1}}(x_k-x_{\star})\|_2^2-\left(\frac{1-\varepsilon_k}{1+\tilde{\varepsilon}_k}\right)\frac{\|AQ_{i_{k_1}}(x_k-x_{\star})\|_F^2}{\tau_{max}}\right)\\
&\leq\tilde{P}_{k_1}\left(\|Q_{i_{k_1}}(x_k-x_{\star})\|_2^2-\left(\frac{1-\varepsilon_k}{1+\tilde{\varepsilon}_k}\right)\frac{\lambda_{min}(A^*A)}{\tau_{max}}\|Q_{i_{k_1}}(x_k-x_{\star})\|_2^2\right)\\
&=\left(1-\frac{1-\varepsilon_k}{1+\tilde{\varepsilon}_k}\frac{\lambda_{min}(A^*A)}{\tau_{max}}\right)\tilde{P}_{k_1}\left(\|x_k-x_{\star}\|_2^2-\frac{|A^{(i_{k_1})}(x_k-x_{\star})|^2}{\|A^{(i_{k_1})}\|_2^2}\right)\\
&\leq\left(1-\frac{1-\varepsilon_k}{1+\tilde{\varepsilon}_k}\frac{\lambda_{min}(A^*A)}{\tau_{max}}\right)\left(1-\frac{1-\varepsilon_{k_1}}{1+\tilde{\varepsilon}_{k_1}}\frac{\lambda_{min}(A^*A)}{\|A\|_F^2}\right)\|x_k-x_{\star}\|_2^2.
\end{aligned}
$$
Therefore, 
\begin{equation}\label{5-9}
\mathbb{E}_k\|\tilde{x}_{k+1}-x_{\star}\|_2^2\leq\left(1-\frac{1-\varepsilon_k}{1+\tilde{\varepsilon}_k}\frac{\lambda_{min}(A^*A)}{\tau_{max}}\right)\left(1-\frac{1-\varepsilon_{k_1}}{1+\tilde{\varepsilon}_{k_1}}\frac{\lambda_{min}(A^*A)}{\|A\|_F^2}\right)\|x_k-x_{\star}\|_2^2.
\end{equation}

Next, we analyze the second term of the right side of Eq.\eqref{5-4}.
Define
$\theta=\frac{|\mu_k|^2}{\|u_k\|_2}$, $\varsigma=\frac{\bar{\mu}_k}{\|u_k\|_2}$, $\phi _{i_{k_1}}=\frac{A^{(i_{k_1})}(x_k-x_{\star})}{\| A^{(i_{k_1})}\|_{2}}$ and $\phi _{i_{k_2}}=\frac{A^{(i_{k_2})}(x_k-x_{\star})}{\| A^{(i_{k_2})}\|_{2}}$, for $p,q\in[m]$ and $p\neq q$, it follows that

\begin{equation}\label{6-10}
\begin{aligned}
&\sum\limits_{{i_{k_1}} \in [m]}\sum\limits_{{i_{k_2}} \in [m],i_{k_2}\neq i_{k_1}}\|A^{(i_{k_1})}\|_2^2\|A^{(i_{k_2})}\|_2^2\\
&\quad\left|\frac{|\mu_k|^2}{\|u_k\|_2}\frac{A^{(i_{k_1})}(x_k-x_{\star})}{\| A^{(i_{k_1})}\|_{2}}-\frac{\bar{\mu}_k}{\|u_k\|_2}\frac{A^{(i_{k_2})}(x_k-x_{\star})}{\| A^{(i_{k_2})}\|_{2}}\right|^2\\
&=\sum\limits_{p<q}\|A^{(p)}\|_2^2\|A^{(q)}\|_2^2\left(|\theta \phi_{p}-\varsigma \phi_{q}|^2+|\varsigma \phi_{q}-\theta \phi_{p}|^2 \right)\\
&\geq\sum\limits_{p<q}\|A^{(p)}\|_2^2\|A^{(q)}\|_2^2(|\theta|-|\varsigma|)^2(\phi_{p}^2+\phi_{q}^2)\\
&\geq D\sum\limits_{p<q}\left(\|A^{(p)}\|_2^2\left|A^{(q)}(x_k-x_{\star})\right|^2+\|A^{(q)}\|_2^2\left|A^{(p)}(x_k-x_{\star})\right|^2\right)\\
&=D\sum_{p\in[m]}\sum\limits_{q\in [m],q\neq p}\|A^{(q)}\|_2^2\left|A^{(p)}(x_k-x_{\star})\right|^2\\
&\geq D\tau_{min}\|A(x_k-x_{\star})\|_2^2.
\end{aligned}
\end{equation}
The first inequality follows by the fact that 
$$|\theta a-\varsigma b|^2+|\theta b-\bar{\varsigma} a|^2\geq(|\theta|-|\varsigma|)^2(|a|^2+|b|^2)$$ 
holds for any $\theta$, $\varsigma$, $a$, and $b\in\mathbb{C}$.
The second inequality follows from the fact that
$$
(|\theta|-|\varsigma|)^2=\left(\frac{|\mu_k|^2-|\bar{\mu}_k|}{\|u_k\|_2}\right)^2=\frac{|\mu_k|^2(1-|\mu_k|)}{1+|\mu_k|}\geq D.
$$
Therefore, it follows that
\begin{equation}\label{5-10}
\begin{aligned}
&\mathbb{E}_k\left|\frac{|\mu_k|^2}{\|u_k\|_2}\frac{A^{(i_{k_1})}(x_k-x_{\star})}{\| A^{(i_{k_1})}\|_{2}}-\frac{\bar{\mu}_k}{\|u_k\|_2}\frac{A^{(i_{k_2})}(x_k-x_{\star})}{\| A^{(i_{k_2})}\|_{2}}\right|^2\\
&=\tilde{P}_{k_1}\tilde{P}_{k_2}\left|\frac{|\mu_k|^2}{\|u_k\|_2}\frac{A^{(i_{k_1})}(x_k-x_{\star})}{\| A^{(i_{k_1})}\|_{2}}-\frac{\bar{\mu}_k}{\|u_k\|_2}\frac{A^{(i_{k_2})}(x_k-x_{\star})}{\| A^{(i_{k_2})}\|_{2}}\right|^2\\
&=\frac{1\pm\varepsilon_{k_1}}{1\pm\tilde{\varepsilon}_{k_1}}\frac{\sum\limits_{\substack{i_{k_1}\in[m] }}\|A^{(i_{k_1})}\|_2^2\sum\limits_{\substack{i_{k_2}\in[m]\\i_{k_2}\neq i_{k_1}}}\|A^{(i_{k_2})}\|_2^2}{\|A\|_F^2(\|A\|_F^2-\|A^{(i_{k_1})}\|_2^2)}\\
&\quad\left|\frac{|\mu_k|^2}{\|u_k\|_2}\frac{A^{(i_{k_1})}(x_k-x_{\star})}{\| A^{(i_{k_1})}\|_{2}}-\frac{\bar{\mu}_k}{\|u_k\|_2}\frac{A^{(i_{k_2})}(x_k-x_{\star})}{\| A^{(i_{k_2})}\|_{2}}\right|^2\\
&\geq \frac{1-\varepsilon_{k_2}}{1+\tilde{\varepsilon}_{k_2}}\frac{D\tau_{min}\|A(x_k-x_{\star})\|_2^2}{\|A\|_F^2\tau_{max}}\\
&\geq\frac{1-\varepsilon_{k_2}}{1+\tilde{\varepsilon}_{k_2}}\frac{D\lambda_{min}(A^*A)}{\|A\|_F^2}\frac{\tau_{min}}{\tau_{max}}\|x_k-x_{\star}\|_2^2.
\end{aligned}
\end{equation}

Substituting \eqref{5-9} and \eqref{5-10} into \eqref{5-4}, we have
$$
\begin{aligned}
&\quad\mathbb{E}_k\|x_{k+1}-x_{\star}\|_2^2\\
&\leq\left[\left(1-\frac{1-\varepsilon_k}{1+\tilde{\varepsilon}_k}\frac{\lambda_{min}(A^*A)}{\tau_{max}}\right)\left(1-\frac{1-\varepsilon_{k_1}}{1+\tilde{\varepsilon}_{k_1}}\frac{\lambda_{min}(A^*A)}{\|A\|_F^2}\right)\right.\\
&\quad\left.-\frac{1-\varepsilon_{k_2}}{1+\tilde{\varepsilon}_{k_2}}\frac{D\lambda_{min}(A^*A)}{\|A\|_F^2}\frac{\tau_{min}}{\tau_{max}}\right]\|x_k-x_{\star}\|_2^2.
\end{aligned}
$$

\section{\textbf{Proof of Lemma \ref{lem-2}}}\label{App-2}
Define
$$\begin{aligned}\Theta_{k} &=A^{(i_{k_1})}(A^{(i_{k_2})})^* (A^{(i_{k_1})})^* A^{(i_{k_2})}-\|A^{(i_{k_2})}\|_2^2(A^{(i_{k_1})})^* A^{(i_{k_1})}\\
&\quad+A^{(i_{k_2})}(A^{(i_{k_1})})^* (A^{(i_{k_2})})^* A^{(i_{k_1})}-\|A^{(i_{k_1})}\|_2^2
(A^{(i_{k_2})})^* A^{(i_{k_2})},\\
\varphi_k&=\|A^{(i_{k_2})}\|_2^2|A^{(i_{k_1})}\overline{\mathbf{x}}_k|^2+\|A^{(i_{k_1})}\|_2^2|A^{(i_{k_2})}\overline{\mathbf{x}}_k|^2\\
&\quad-(A^{(i_{k_1})}\overline{\mathbf{x}}_k)^* A^{(i_{k_1})}(A^{(i_{k_2})})^* A^{(i_{k_2})}\overline{\mathbf{x}}_k-(A^{(i_{k_2})}\overline{\mathbf{x}}_k)^* A^{(i_{k_2})}(A^{(i_{k_1})})^* A^{(i_{k_1})}\overline{\mathbf{x}}_k,\end{aligned}$$
in which $\overline{\mathbf{x}}_k=x_k-x_{\star}$.
From Algorithm \ref{algo-3-8}, we have
\begin{equation}\label{5-11}
\begin{aligned}
\overline{\mathbf{x}}_{k+1}&=\overline{\mathbf{x}}_k+\frac{\Theta_k\overline{\mathbf{x}}_k}{\|A^{(i_{k_1})}\|_2^2\|A^{(i_{k_2})}\|_2^2-|A^{(i_{k_1})}(A^{(i_{k_2})})^*|^2}\\
&=\overline{\mathbf{x}}_k-P\overline{\mathbf{x}}_k\\
&=(I-P)\overline{\mathbf{x}}_k,   
\end{aligned}
\end{equation}
where $P=-\frac{\Theta_k}{\|A^{(i_{k_1})}\|_2^2\|A^{(i_{k_2})}\|_2^2-|A^{(i_{k_1})}(A^{(i_{k_2})})^*|^2}$. It is easy to verify that $P^2=P$ and $P^*=P$. This fact shows that $P$ is a projection matrix, and if $P$ is a projector, then so is $(I-P)$. From Eq.\eqref{5-11}, it follows that
\begin{equation}\label{5-12}
\begin{aligned}
\|\overline{\mathbf{x}}_{k+1}\|_2^2
&=\|(I-P)\overline{\mathbf{x}}_{k}\|_2^2\\
&=\langle(I-P)\overline{\mathbf{x}}_{k},(I-P)\overline{\mathbf{x}}_{k}\rangle\\
&=\overline{\mathbf{x}}_k^*(I-P)\overline{\mathbf{x}}_k\\
&=\overline{\mathbf{x}}_k^*\left(\overline{\mathbf{x}}_k-P\overline{\mathbf{x}}_k\right)\\
&=\|\overline{\mathbf{x}}_k\|_2^2-\langle P\overline{\mathbf{x}}_k,P\overline{\mathbf{x}}_k\rangle\\
&=\|\overline{\mathbf{x}}_k\|_2^2-\|
P\overline{\mathbf{x}}_k\|_2^2\\
&=\|\overline{\mathbf{x}}_k\|_2^2-\|x_{k+1}-x_{k}\|_2^2.
\end{aligned}  
\end{equation}
Define
$$\mathcal{P}_{k_1}=\sum\limits_{{i_{k_1}} \in {\mathcal{U}_k}}\frac{|{b}^{(i_{k_1})}-A^{(i_{k_1})}x_{k}|^2}{\sum\limits_{i\in \mathcal{U}_k}|{b}^{(i)}-A^{(i)}x_{k}|^2} \quad,\quad\mathcal{P}_{k_2}=\sum\limits_{{i_{k_2}} \in {\mathcal{U}_k},i_{k_2}\neq i_{k_1}}\frac{|{b}^{(i_{k_2})}-A^{(i_{k_2})}x_{k}|^2}{\sum\limits_{j\in \mathcal{U}_k,j\neq i_{k_1}}|{b}^{(j)}-A^{(j)}x_{k}|^2}.$$
Then, we have
\begin{equation}\label{5-13}
\begin{aligned}
\|\overline{\mathbf{x}}_{k+1}\|_2^2&=\|\overline{\mathbf{x}}_k\|_2^2-\|x_{k+1}-x_{k}\|_2^2\\
&=\|\overline{\mathbf{x}}_k\|_2^2-\frac{\varphi_k}{\|A^{(i_{k_1})}\|_2^2\|A^{(i_{k_2})}\|_2^2-|A^{(i_{k_1})}(A^{(i_{k_2})})^*|^2}\\
&\leq \|\overline{\mathbf{x}}_k\|_2^2-\left(|v_{k_1}|^2+\left|\frac{1}{\|u_k\|_2}v_{k_2}-\frac{|\mu_k|}{\|u_k\|_2}v_{k_1}\right|^2\right)\\
\end{aligned}
\end{equation}
where $v_{k_1}=\frac{|{b}^{(i_{k_1})}-A^{(i_{k_1})}x_{k}|}{\|A^{(i_{k_1})}\|_2}$ and $v_{k_2}=\frac{|{b}^{(i_{k_2})}-A^{(i_{k_2})}x_{k}|}{\|A^{(i_{k_2})}\|_2}$.
\begin{equation}\label{5-14}
\begin{aligned}
&\mathcal{P}_{k_1}\mathcal{P}_{k_2}\left|\frac{1}{\|u_k\|_2}v_{k_2}-\frac{|\mu_k|}{\|u_k\|_2}v_{k_1}\right|^2\\
&\geq\mathcal{P}_{k_2}\sum\limits_{{i_{k_1}} \in {\mathcal{U}_k}}\frac{\epsilon_{k}\|b-Ax\|_2^2\|A^{(i_{k_1})}\|_2^2\sum\limits_{i_{k_2}\in \mathcal{U}_k,i_{k_2}\neq i_{k_1}}\|A^{(i_{k_2})}\|_2^2}{\sum\limits_{i\in \mathcal{U}_k}|{b}^{(i)}-A^{(i)}x_{k}|^2\sum\limits_{i_{k_2}\in \mathcal{U}_k,i_{k_2}\neq i_{k_1}}\|A^{(i_{k_2})}\|_2^2}\left|\frac{1}{\|u_k\|_2}v_{k_2}-\frac{|\mu_k|}{\|u_k\|_2}v_{k_1}\right|^2\\
&=\frac{\epsilon_{k}\|b-Ax\|_2^2\mathcal{P}_{k_2}}{\sum\limits_{i\in \mathcal{U}_k}|{b}^{(i)}-A^{(i)}x_{k}|^2\sum\limits_{i_{k_2}\in \mathcal{U}_k,i_{k_2}\neq i_{k_1}}\|A^{(i_{k_2})}\|_2^2}\\
&\quad\sum\limits_{{i_{k_1}} \in {\mathcal{U}_k}}\sum\limits_{{i_{k_2}} \in {\mathcal{U}_k,i_{k_2}\neq i_{k_1}}}\|A^{(i_{k_1})}\|_2^2\|A^{(i_{k_2})}\|_2^2\left|\frac{1}{\|u_k\|_2}v_{k_2}-\frac{|\mu_k|}{\|u_k\|_2}v_{k_1}\right|^2\\
&\geq\frac{\epsilon_{k}\|b-Ax\|_2^2\mathcal{P}_{k_2}}{\sum\limits_{i\in \mathcal{U}_k}|{b}^{(i)}-A^{(i)}x_{k}|^2\sum\limits_{\substack{i_{k_2}\in \mathcal{U}_k\\i_{k_2}\neq i_{k_1}}}\|A^{(i_{k_2})}\|_2^2}\left(\frac{1-\Delta}{1+\Delta}\sum\limits_{\substack{q\in \mathcal{U}_k\\q\neq p}}\|A^{(q)}\|_2^2\sum\limits_{p\in \mathcal{U}_k}\left|A^{(p)}(x_k-x_*)\right|^2\right)\\
&=\frac{1-\Delta}{1+\Delta}\epsilon_{k}\|b-Ax\|_2^2\\
&\geq\frac{1-\Delta}{2(1+\Delta)}\left(\frac{\|A\|_F^2}{\tau_{max}}+1\right)\frac{\lambda_{min}(A^*A)}{\|A\|_F^2},
\end{aligned}
\end{equation}
the second inequality is because that for $p,q\in[m]$ and $p\neq q$, we can obtain
$$
\begin{aligned}
&\sum\limits_{{i_{k_1}} \in {\mathcal{U}_k}}\sum\limits_{{i_{k_2}} \in {\mathcal{U}_k},i_{k_2}\neq i_{k_1}}\|A^{(i_{k_1})}\|_2^2\|A^{(i_{k_2})}\|_2^2\left|\frac{1}{\|u_k\|_2}v_{k_2}-\frac{|\mu_k|}{\|u_k\|_2}v_{k_1}\right|^2\\
&=\sum\limits_{p<q}\|A^{(p)}\|_2^2\|A^{(q)}\|_2^2\left(|\theta v_{k_2}-\zeta v_{k_1}|^2+|\theta v_{k_1}-\zeta v_{k_2}|^2 \right)\\
&\geq\sum\limits_{p<q}\|A^{(p)}\|_2^2\|A^{(q)}\|_2^2(\theta-\zeta)^2(v_{k_1}^2+v_{k_2}^2)\\
&\geq\frac{1-\Delta}{1+\Delta}\sum\limits_{p<q}\left(\|A^{(p)}\|_2^2\left|A^{(q)}(x_k-x_*)\right|^2+\|A^{(q)}\|_2^2\left|A^{(p)}(x_k-x_*)\right|^2\right)\\
&=\frac{1-\Delta}{1+\Delta}\sum\limits_{p\in \mathcal{U}_k}\sum\limits_{q\in \mathcal{U}_k,q\neq p}\left(\|A^{(q)}\|_2^2\left|A^{(p)}(x_k-x_*)\right|^2\right)\\
&=\frac{1-\Delta}{1+\Delta}\sum\limits_{q\in \mathcal{U}_k,q\neq p}\|A^{(q)}\|_2^2\sum\limits_{p\in \mathcal{U}_k}\left|A^{(p)}(x_k-x_*)\right|^2,
\end{aligned}
$$
where $\theta=\frac{|\mu_k|}{\|u_k\|_2}$ and $\zeta=\frac{1}{\|u_k\|_2}$.
Finally, from Eqs. \eqref{5-13} and \eqref{5-14}, it follows that
\begin{equation}\label{5-15}
\begin{aligned}
&\mathbb{E}_{k}\|\overline{\mathbf{x}}_{k+1}\|  _{2}^2= \|\overline{\mathbf{x}}_k\| _{2}^2 - {\mathbb{E}_{k}}\parallel x_{k+1} - x_{k}{\parallel _{2}^2} \\
&\leq\|\overline{\mathbf{x}}_k\| _{2}^2-\mathcal{P}_{k_1}\mathcal{P}_{k_2}\left(|v_{k_1}|^2+\left|\frac{1}{\|u_k\|_2}v_{k_2}-\frac{|\mu_k|}{\|u_k\|_2}v_{k_1}\right|^2\right)\\
&\leq\left[1-\frac{1}{2}\left(\frac{\|A\|_F^2}{\tau_{max}}+1\right)\frac{\lambda_{min}(A^*A)}{\|A\|_F^2}-\frac{1-\Delta}{2(1+\Delta)}\left(\frac{\|A\|_F^2}{\tau_{max}}+1\right)\frac{\lambda_{min}(A^*A)}{\|A\|_F^2}\right]\|\overline{\mathbf{x}}_k\| _{2}^2\\
&=\left[1-\frac{1}{1+\Delta}\left(\frac{\|A\|_F^2}{\tau_{max}}+1\right)\frac{\lambda_{min}(A^*A)}{\|A\|_F^2}\right]\|\overline{\mathbf{x}}_k\| _{2}^2.
\end{aligned}
\end{equation}
From Algorithm \ref{algo-3-8}, it follows that
$$
\begin{aligned}
r_{k+1}&=b-Ax_{k+1}=b-A\left(x_k+\gamma_k(A^{(i_{k_1})})^{*} + \lambda_k(A^{(i_{k_2})}) ^{*}\right)\\
&=b-Ax_k-\gamma_kA(A^{(i_{k_1})})^{*}-\lambda_kA(A^{(i_{k_2})}) ^{*}\\
&=r_k-\gamma_kB_{(i_{k_1})}-\lambda_kB_{(i_{k_2})},
\end{aligned}
$$
with $B=AA^*$ and $B_{(l)}$ ($l=i_{k_1},i_{k_2}$) being the $l$-th column of $B$.

\section{\textbf{Proof of Lemma \ref{lem-3}}}\label{App-3}

From Algorithm \ref{algo-3-10}, Eqs. \eqref{5-11} and \eqref{5-12}, it follows that
\begin{equation}\label{5-16}
\|\overline{\mathbf{x}}_{k+1}\|_2^2=\|\overline{\mathbf{x}}_k\|_2^2-\|x_{k+1}-x_{k}\|_2^2.
\end{equation}
Then, we have
\begin{equation}\label{5-17}
\begin{aligned}
\|\overline{\mathbf{x}}_{k+1}\|_2^2&=\|\overline{\mathbf{x}}_k\|_2^2-\|x_{k+1}-x_{k}\|_2^2\\
&=\|\overline{\mathbf{x}}_k\|_2^2-\frac{\varphi_k}{\|A^{(i_{k_1})}\|_2^2\|A^{(i_{k_2})}\|_2^2-|A^{(i_{k_1})}(A^{(i_{k_2})})^*|^2}\\
&\leq \|\overline{\mathbf{x}}_k\|_2^2-\left(|v_{k_1}|^2+\left|\frac{1}{\|u_k\|_2}v_{k_2}-\frac{|\mu_k|}{\|u_k\|_2}v_{k_1}\right|^2\right)\\
&\leq\left(1-\frac{\lambda_{min}(A^*A)}{\tau_{max}}-\frac{\lambda_{min}(A^*A)}{\tau_{max}}\frac{\omega}{c^2(1-\delta^2)}\right)\|\overline{\mathbf{x}}_k\|_2^2,
\end{aligned}
\end{equation}
where $v_{k_1}=\frac{|{b}^{(i_{k_1})}-A^{(i_{k_1})}x_{k}|}{\|A^{(i_{k_1})}\|_2}$, $v_{k_2}=\frac{|{b}^{(i_{k_2})}-A^{(i_{k_2})}x_{k}|}{\|A^{(i_{k_2})}\|_2}$, and $v_{k_1}=cv_{k_2}$ with $c>1$.\\
Here, the last inequality is achieved by
\begin{equation}\label{5-18}
\begin{aligned}
\left|\frac{1}{\|u_k\|_2}v_{k_2}-\frac{|\mu_k|}{\|u_k\|_2}v_{k_1}\right|^2&=\left(\frac{1}{c\|u_k\|_2}-\frac{|\mu_k|}{\|u_k\|_2}\right)^2v_{k_1}^2\\
&=\frac{|1-c|\mu_k||^2}{c^2\|u_k\|_2^2}v_{k_1}^2\\
&\geq\frac{\omega}{c^2(1-\delta^2)}v_{k_1}^2,
\end{aligned}
\end{equation}
where $\omega=\min{|1-c|\mu_k||^2},$
and
\begin{equation}\label{5-19}
|v_{k_1}|^2\geq\frac{\lambda_{min}(A^*A)}{\tau_{max}}\|\overline{\mathbf{x}}_k\|_2^2.
\end{equation}
Finally, it follows from Eq.\eqref{5-17}
\begin{equation}\label{5-20}
\mathbb{E}_k\|x_{k+1}-x_{\star}\|_2^2\leq \left(1-\frac{\lambda_{min}(A^*A)}{\tau_{max}}-\frac{\lambda_{min}(A^*A)}{\tau_{max}}\frac{\omega}{c^2(1-\delta^2)}\right)\|x_k-x_{\star}\|_2^2.    
\end{equation}

\section{
\textbf{Proof of Lemma \ref{lem-4}}}\label{App-4}

From Chebyshev’s (weak) law of large numbers, if 
$\eta m$ is sufficiently large and there are two scalars $0<\varepsilon_{k_3},\tilde{\varepsilon}_{k_3}\ll1$, then it follows that

\begin{equation}\label{5-21}
\frac{\sum\limits_{\substack{i_{k_1}\in\Phi_k}}\|A^{(i_{k_1})}\|_2^2}{\eta m}  =\frac{\sum\limits_{\substack{i_{k_1}\in[m]\\i_{k_1}\neq i_{{k-1}_1}}}\|A^{(i_{k_1})}\|_2^2}{m-1}(1\pm\varepsilon_{k_3} ),
\end{equation}
\begin{equation}\label{5-22}
\begin{aligned}
&\quad\frac{\sum\limits_{i_{k_1}\in\Phi _k}|{b}^{(i_{k_1})}-A^{(i_{k_1})}x_{k}|^2}{\eta m}=\frac{\sum_{i_{k_1}\in[m] }|{b}^{(i_{k_1})}-A^{(i_{k_1})}x_{k}|^2}{m-1}(1\pm\tilde{\varepsilon}_{k_3} ) .
\end{aligned}
\end{equation}
and
\begin{equation}\label{5-23}
\begin{aligned}
\left|\frac{1}{\|u_k\|_2}v_{k_2}-\frac{|\mu_k|}{\|u_k\|_2}v_{k_1}\right|^2&=\left(\frac{1}{c_1\|u_k\|_2}-\frac{|\mu_k|}{\|u_k\|_2}\right)^2v_{k_1}^2\\
&=\frac{|1-c_1|\mu_k||^2}{c_1\|u_k\|_2^2}v_{k_1}^2\\
&\geq\frac{\omega_1}{c_1^2(1-\delta^2)}v_{k_1}^2,
\end{aligned}
\end{equation}
where $v_{k_1}=\frac{|{b}^{(i_{k_1})}-A^{(i_{k_1})}x_{k}|}{\|A^{(i_{k_1})}\|_2}$, $v_{k_2}=\frac{|{b}^{(i_{k_2})}-A^{(i_{k_2})}x_{k}|}{\|A^{(i_{k_2})}\|_2}$,  and $v_{k_1}=c_1v_{k_2}$ with $c_1>1$, $\omega_1=\min{|1-c_1|\mu_k||^2}.$

\begin{equation}\label{5-24}
\begin{aligned}
|v_{k_1}|^2&\geq\frac{\sum\limits_{i_{k_1}\in\Phi _k}|{b}^{(i_{k_1})}-A^{(i_{k_1})}x_{k}|^2}{\sum\limits_{\substack{i_{k_1}\in\Phi_k}}\|A^{(i_{k_1})}\|_2^2}\\
&=\frac{\frac{1}{\eta m}\sum\limits_{i_{k_1}\in\Phi _k}|{b}^{(i_{k_1})}-A^{(i_{k_1})}x_{k}|^2}{\frac{1}{\eta m}\sum\limits_{i_{k_1}\in\Phi _k}\|A^{(i_{k_1})}\|_2^2}\\
&=\left(\frac{1\pm\varepsilon_{k_3}}{1\pm\tilde{\varepsilon}_{k_3}}\right)\frac{\sum\limits_{i_{k_1}\in[m] }|{b}^{(i_{k_1})}-A^{(i_{k_1})}x_{k}|^2}{\sum\limits_{\substack{i_{k_1}\in[m]\\i_{k_1}\neq i_{{k-1}_1}}}\|A^{(i_{k_1})}\|_2^2}\\
&\geq\frac{1-\varepsilon_{k_3}}{1+\tilde{\varepsilon}_{k_3}}\frac{\|b-Ax\|_2^2}{\tau_{max}}\\
&\geq\frac{1-\varepsilon_{k_3}}{1+\tilde{\varepsilon}_{k_3}}\frac{\lambda_{min}(A^*A)}{\tau_{max}}\|\overline{\mathbf{x}}_k\|_2^2.
\end{aligned}
\end{equation}
Form Eqs.\eqref{5-12},\eqref{5-23} and \eqref{5-24}, we have
$$
\begin{aligned}
\mathbb{E}_{k}\|\overline{\mathbf{x}}_{k+1}\|_2^2&=\|\overline{\mathbf{x}}_k\|_2^2-\mathbb{E}_{k}\|x_{k+1}-x_{k}\|_2^2\\
&=\|\overline{\mathbf{x}}_k\|_2^2-\frac{\varphi_k}{\|A^{(i_{k_1})}\|_2^2\|A^{(i_{k_2})}\|_2^2-|A^{(i_{k_1})}(A^{(i_{k_2})})^*|^2}\\
&\leq \|\overline{\mathbf{x}}_k\|_2^2-\left(|v_{k_1}|^2+\left|\frac{1}{\|u_k\|_2}v_{k_2}-\frac{|\mu_k|}{\|u_k\|_2}v_{k_1}\right|^2\right)\\
&\leq\left(1-\frac{1-\varepsilon_{k_3}}{1+\tilde{\varepsilon}_{k_3}}\frac{\lambda_{min}(A^*A)}{\tau_{max}}-\frac{1-\varepsilon_{k_3}}{1+\tilde{\varepsilon}_{k_3}}\frac{\lambda_{min}(A^*A)}{\tau_{max}}\frac{\omega_1}{c_1^2(1-\delta^2)}\right)\|\overline{\mathbf{x}}_k\|_2^2.
\end{aligned}
$$

\end{document}